\documentclass[10pt,a4paper]{amsart}

\usepackage{graphicx}
\usepackage{tikz}
\usepackage[all]{xy}
\usepackage{amsthm}
\usepackage{amssymb} 
\usepackage{amsmath,amscd} 
\usepackage{mathrsfs}
\usepackage{color}
\usepackage{enumerate}

\newcommand{\agl}[1]{\langle #1 \rangle}
\newcommand{\mukaiagl}[1]{\langle #1 \rangle_{\mathsf{Mukai}}}

\newcommand{\bmukaiagl}[1]{\langle #1 \rangle_{\mathsf{Shk\textup{-}Mukai}}}
\newcommand{\cwmukaiagl}[1]{\langle #1 \rangle_{\mathsf{CW\textup{-}Mukai}}}
\newcommand{\ERagl}[1]{\langle #1 \rangle_{\mathsf{ER}}}

\newcommand{\serreagl}[1]{\langle #1 \rangle_{\mathsf{Serre}}}
\newcommand{\invserreagl}[1]{\langle #1 \rangle_{\mathsf{Serre}^{-1}}}
\newcommand{\isagl}[1]{\langle #1 \rangle_{\mathsf{Serre}^{-1}}}

\newcommand{\lunit}{ {}^{\vartriangleleft}\!\eta}
\newcommand{\lcounit}{ {}^{\vartriangleleft}\!\epsilon}

\newcommand{\runit}{{}^{\vartriangleright}\!\eta}
\newcommand{\rcounit}{{}^{\vartriangleright}\!\epsilon}

\newcommand{\vunit}{{}^{\Bbb{V}}\!\eta}
\newcommand{\vcounit}{{}^{\Bbb{V}}\!\epsilon}

\newcommand{\dunit}{{}^{\textup{D}}\!\eta}
\newcommand{\dcounit}{{}^{\textup{D}}\!\epsilon}

\def\DPic{{\textup{DPic}}}
\def\Euch{{\chi}}
\def\Euvect{{\underline{\chi}}}

\def\comp{{\mathsf{comp}}}
\def\flip{{\mathsf{flip}}}

\def\ch{{\textup{ch}}}

\def\turn!{\textup{!`}}

\def\sfAR{\mathsf{AR}}

\def\op{\textup{op}}

\def\Tr{\operatorname{Tr}}

\def\ind{\mathop{\mathrm{ind}}\nolimits}

\def\mre{\mathrm{e}}

\def\lvvee{\vartriangleleft} 
\def\rvvee{\vartriangleright} 

\def\kk{{\mathbf k}}

\def\FF{{\Bbb F}}

\def\HHH{{\Bbb H}\textup{H}}

\def\ZZ{{\Bbb Z}}

\def\cA{{\mathcal A}}

\def\cC{{\mathcal C}}
\def\cD{{\mathcal D}}

\def\sfa{{\mathsf{a}}}
\def\sfb{{\mathsf{b}}}
\def\sfc{{\mathsf{c}}}
\def\sfd{{\mathsf{d}}}
\def\sfe{{\mathsf{e}}}

\def\sfD{{\mathsf{D}}}

\def\sfS{{\mathsf{S}}}

\def\tuD{{\textup{D}}}

\def\tuH{{\textup{H}}}
\def\tuHH{{\textup{HH}}}

\def\id{\operatorname{id}}
\def\op{\operatorname{op}}

\def\mod{\operatorname{mod}}
\def\Mod{\operatorname{Mod}}

\def\Hom{\operatorname{Hom}}

\def\cpxHom{\operatorname{Hom}^{\bullet}}

\def\End{\operatorname{End}}
\def\ResEnd{\operatorname{ResEnd}}

\newcommand{\RHom}{\operatorname{\Bbb{R}Hom}}

\newcommand{\lotimes}{\otimes^{\Bbb{L}}}

\newcommand{\cone}{\operatorname{\mathsf{cn}}}

\def\RHom{\operatorname{\mathbb{R}Hom}}

\def\perf{\operatorname{\mathsf{perf}}}

\newtheorem{lemma}{Lemma}[section]
\newtheorem{proposition}[lemma]{Proposition}
\newtheorem{theorem}[lemma]{Theorem}
\newtheorem{corollary}[lemma]{Corollary}

\newtheorem{remark}[lemma]{Remark}
\newtheorem{example}[lemma]{Example}

\newtheorem{definition}[lemma]{Definition}
\def\vvee{\Bbb{V}}

\def\ppi{{\pi}}
\def\rrho{{\varrho}}

\def\Pa{A}

\def\rad{\operatorname{rad}}

\def\PPi{\Pi}

\title[Serre, Mukai and Auslander--Reiten]
{Serre duality, 
Mukai pairing and universal Auslander--Reiten triangle}
\date{}

\author{Hiroyuki Minamoto
}

\dedicatory{Dedicated to the memory of Idun Reiten}
\begin{document}
\maketitle

\begin{abstract}
We study the relationship between Serre duality and the Mukai pairing for smooth and proper dg-algebras.
We introduce an alternative definition of the Mukai pairing and prove that it coincides with the Mukai pairings
defined by Căldăraru--Willerton and by Shklyarov.
Our construction places both the Mukai pairing and Serre duality within a unified framework based on an
elementary pairing between Hochschild homology and Hochschild cohomology.
As a consequence, the adjointness of the boundary--bulk and bulk--boundary maps follows naturally. 

As an application, we investigate Auslander--Reiten theory for the perfect derived category of a non-positive smooth
and proper dg-algebra. 
We construct  an exact triangle of dg-$A$-$A$-bimodules,  called a universal Auslander--Reiten triangle 
in the sense that 
the derived tensor product of this triangle with an indecomposable dg-$A$-module $M$  yields 
an Auslander--Reiten triangle starting from $M$. 
In particular, this provides a functorial construction of Auslander--Reiten triangles.
In the case of path algebras of quivers, our construction recovers the universal Auslander--Reiten triangle
associated with quiver Heisenberg algebras.
\end{abstract}

%

\section{Introduction}

\subsection{Serre duality and Mukai pairing}

The Serre functor, introduced by Bondal and Kapranov \cite{BK} as a categorical reformulation of Serre duality in algebraic geometry, encodes a fundamental symmetry of triangulated categories and therefore plays a central role in categorical algebraic geometry, the representation theory of algebras, and mathematical physics.

The Mukai pairing was first introduced by Mukai \cite{Mukai} as a pairing on the cohomology groups of a K3 surface, in the course of his study of the moduli spaces of sheaves on K3 surfaces.
Later, Căldăraru and Willerton \cite{CW} generalized the Mukai pairing to a pairing on the Hochschild homology of smooth algebraic varieties $X$.
Independently, Shklyarov \cite{Shklyarov} introduced a pairing on the Hochschild homology of smooth and proper dg-algebras $A$. 
This pairing is sometimes referred to as the Mukai pairing (e.g. \cite{Sheridan}), or simply as a pairing (e.g. \cite{Petit, PV}). 
Like Serre duality, the Mukai pairing has emerged as an important structure in
categorical algebraic geometry,  representation theory of algebras, and mathematical physics \cite{Caldararu, CW, Carqueville, Han, PV, Sheridan}.
We point out that Serre duality and Mukai pairing are related via 
the boundary–bulk maps and the bulk–boundary maps (see Section \ref{202602012158} for the definitions),  
and that this relationship plays a key role in topological field theory (see e.g. \cite{Carqueville}).

As a basic issue, we explain that there are two definitions of Mukai pairing. 
The definition due to Căldăraru--Willerton is formulated in terms of derived categories and thus makes sense for dg-algebras.
Similarly, Shklyarov's definition applies naturally to  algebraic varieties. 
Ramadoss \cite{Ramadoss} proved that these two pairings coincide, at the level of cohomology groups, in the case of an algebraic variety $X$.
This raises the natural question of whether these two definitions agree in general.
It appears to be folklore that these two pairings coincide in general (see \cite[Remark 1.2.2]{PV}).
In this paper, first we propose an alternative definition of the Mukai pairing $\mukaiagl{-,+}$ 
and prove that three Mukai pairings coincide with each other. 
A precise statement of one of our main theorems is as follows. 

\begin{theorem}[{Theorem \ref{202508231520}}]\label{2025082315200}
The three Mukai pairings coincide with each other:
\[
\mukaiagl{-,+} = \bmukaiagl{-,+} =\cwmukaiagl{-,+}
\]
where $\bmukaiagl{-,+}$ and $\cwmukaiagl{-,+}$ are the Mukai pairings in the sense of Shklyarov and of Căldăraru–Willerton, respectively.
\end{theorem}

The main feature  of  our Mukai pairing $\mukaiagl{-,+}$ is that it 
arises as a special case of an elementary pairing between Hochschild homology and Hochschild cohomology.
Within the same framework, Serre duality is also obtained as another special case of this elementary pairing. 
As a result, the relationship between Serre duality and the Mukai pairing is made transparent, and the adjointness of the boundary–bulk and bulk–boundary maps follows immediately 
(Theorem \ref{202508171830}).

\subsection{Universal Auslander-Reinten triangle}
This research grew out of an effort to understand and generalize the trace formula (Corollary \ref{2025083017451}) and the universal Auslander--Reiten triangle for the path algebra $\kk Q$ of a quiver,  
obtained  in the author’s previous joint work with M. Herschend on quiver Heisenberg algebras \cite{QHA}. 

Auslander--Reiten (AR) theory is of fundamental theoretical importance and serves as a central tool in  representation theory of algebras. 
In recent years, it has also become influential in categorical algebraic geometry and mathematical physics, 
due to the close interactions among these areas. 
AR theory assigns to a $\Hom$-finite $\kk$-linear category $\cC$ a quiver $\Gamma(\cC)$, called the Auslander--Reiten quiver, 
which can be regarded as a skeleton of the category $\cC$.

Furthermore, for the module category $\mod \Lambda$ of a finite-dimensional algebra $\Lambda$, 
AR theory provides the Auslander--Reiten translation $\tau_{1}$ 
and Auslander--Reiten sequences, 
which can be effectively used in practice to compute the AR quiver $\Gamma(\mod \Lambda)$; 
they also reveal the rich internal structure of the category \cite{ARS, ASS}.

The same story holds for   a $\Hom$-finite triangulated category $\cD$ that possesses  a Serre functor, 
namely  Happel's Auslander--Reiten theory for triangulated categories 
\cite{Happel Book}. 
In this context, the Serre functor is also called the Nakayama functor and denoted by $\nu$. 
In Happel's AR-theory, AR-translation is given by the shifted Nakayama functor $\nu_{1}:= \nu[-1]$ and 
 the notion of AR-triangles is introduced. 
 These notions play the same important role as described above.

Happel’s AR-theory applies to the perfect derived category $\perf A$ of a smooth and proper dg-algebra $A$, 
which is the main case of interest in the literature. 
Our second main result, the universal Auslander--Reiten triangle,  shows that 
in this situation, if moreover the dg-algebra $A$ is assumed to be  non-positive, 
 AR-triangles are induced from an  exact triangle of dg-$A$-$A$-bimodules via derived tensor products, 
and thus, in particular, AR-triangles admit a functorial construction.

For an element  $v \in K_{0}(A)_{\kk}:= K_{0}(A)\otimes_{\ZZ} \kk$ of the extended Grothendieck group,  
we associate the following  exact triangle ${}^{v}\!\sfAR$  of dg-$A$-$A$-bimodules:
\[
{}^{v}\!\sfAR: A \xrightarrow{  \ {}^{v}\!\varrho \ } {}^{v}\Xi \xrightarrow{ \ {}^{v}\!\pi \ }\Sigma A^{\vvee} \xrightarrow{ \ - {}^{v}\!\theta \ } \Sigma A
\]
where $A^{\vvee} := \RHom_{A^{\mre}}(A, A^{\mre})$ is the bimodule dual of $A$. 
We recall that the derived tensor product functor $\Sigma A^{\vvee}\lotimes_{A}- $ coincides with the inverse  $\nu^{-1}_{1}$ of the AR-translation on $\perf A$.

\begin{theorem}[{Universal Auslander--Reiten triangle (Corollary \ref{202509171554})}]\label{202602041047}
Let $A$ be a non-positive smooth and proper dg-algebra. 
Assume that  $v \in K_{0}(A)_{\kk}$ is  regular (Definition \ref{202509171540}). 
Then for any indecomposable object 
$M \in \perf A$, the exact triangle ${}^{v}\!\sfAR_{M}:= {}^{v}\!\sfAR \lotimes_{A} M $ is an Auslander--Reiten triangle. 
\[
{}^{v}\!\sfAR_{M} : M  \xrightarrow{  \ {}^{v}\!\rho_{M} \ } {}^{v}\Xi \lotimes_{A} M 
\xrightarrow{ \ {}^{v}\!\pi_{M}  \ } \nu_{1}^{-1}(M) \xrightarrow{ \ - {}^{v}\!\theta_{M}  \ } \Sigma  M. 
\] 
\end{theorem}

The proof of this theorem proceeds by using the trace formula
(Corollary \ref{2025083017451}), which is obtained from the  Hirzebruch--Riemann--Roch theorem and the adjointness between the boundary–bulk and bulk–boundary maps, in order to apply Happel’s criterion for AR-triangles (Theorem \ref{Happel's criterion}).

We note that in the case where the dg-algebra $A=\kk Q$ is the path algebra of a quiver $Q$, 
the dg-bimodule ${}^{v}\Xi$  coincides with  the degree $1$-part ${}^{v}\!\Lambda(Q)_{1}$ of the quiver Heisenberg algebra 
${}^{v}\!\Lambda(Q)$ of $Q$ \cite{QHA}.

After establishing the universal AR triangles, 
we investigate  their  behavior  under the action of the derived Picard group $\DPic_{\kk}(A)$ as well as  $\kk$-duality $\tuD(-)$. 
In the case where $A = \kk Q$ is the path algebra, these results play crucial role in the study of the quiver Heisenberg algebras.

In \cite{Minamoto: Smith's theorem} for a smooth dg-algebra $A$ (which is not necessarily non-positive nor proper) 
we investigate the exact triangle ${}^{v}\!\sfAR$ and show that 
the dg-bimodule ${}^{v}\Xi$ constitutes the degree $1$ part  ${}^{v}\!\Lambda(A)_{1}$ of the generalized quiver Heisenberg algebra 
${}^{v}\!\Lambda(A)$ of $A$. 
Thus, it is  natural to expect that 
in the case where $A$ is smooth and proper,  the relationships of the universal AR triangles with the derived Picard group action and $\kk$-duality 
play important roles in the  study of the generalized quiver Heisenberg algebras.

\subsection{Related work and perspective}

An important result about the Mukai pairing is the Hirzebruch--Riemann--Roch (HRR) theorem for  smooth and proper dg-algebras \cite{CW, Petit, PV, Shklyarov},  
which we recall in Theorem \ref{202508261418}. 
A generalization of the HRR theorem from the viewpoint of noncommutative motives is given by Cisinski and Tabuada \cite{CT} . 
It is natural to ask whether the adjointness of the boundary–bulk and bulk–boundary maps also generalizes to this setting. 
Campbell-Ponto \cite{CP} generalized HRR theorem into the context of monoidal $2$-categories having appropriate dualities.
In the present paper, we do  not address the general theory of monoidal  $2$-categories, 
but we implicitly work with  the Morita $2$-category of dg-algebras. 
This paper does not require prior knowledge of $2$-categories. 
Nevertheless, readers will find that even when one is primarily interested in dg-algebras and their derived categories, 
it is indispensable to use the framework of $2$-categories  to handle  canonical dualities  (see Lemma \ref{202508201805}  and the subsequent discussion).
The only background on $2$-categories required to read this paper is the notion of adjoint $1$-morphisms, which is summarized in Appendix~\ref{202511221618}. 
Since the argument is carried out  in the $2$-categorical framework, it is natural to anticipate  that the adjointness of the boundary–bulk and bulk–boundary maps also extends to 
general monoidal $2$-categories, as  considered by  Campbell--Ponto. 

Concerning universal Auslander--Reiten triangles,
Auslander--Carlson \cite[Theorem 3.6]{AC} prove a related result for group algebras,
which was later generalized to Hopf algebras $H$
by Green--Marcos--Solberg \cite[Theorem 3.8]{GMS}.
They show that, by taking the internal tensor product $\otimes_{\kk}$
(i.e., the tensor product given by the monoidal structure of $H\mod$)
with an indecomposable module $M$
of an AR-sequence ending at the trivial module $\kk_{H}$,
one obtains an AR-sequence (modulo injectives) ending at $M$.
It is tempting to seek a framework that unifies this result
with the universal AR triangles.

\subsection{Organization of the paper}

In Section \ref{202602050954}, after preparing basic notations and notions about dg-algebras and dg-modules, including Hochschild (co)homology, 
we introduce string diagrams to express the derived tensor product of dg-bimodules, which are an important tool for handling complicated derived tensor products.
The formulas  in Lemma \ref{202508201805} for derived tensor products of  dg-bimodules endowed with various dualities form the technical heart of the present paper.

In Section \ref{202602050955}, we first establish the canonical pairing between Hochschild homology and cohomology. 
We show that this pairing specializes to Serre duality, and we provide  dg-enhanced versions of basic facts about the Serre functor.
By another specialization of the canonical pairing, we introduce our Mukai pairing,  and show that
the adjointness of the boundary-bulk and bulk-boundary maps (Theorem \ref{202508171830})  is a natural consequence. 
In Theorem \ref{202508231520}, we prove that three Mukai pairings coincide. 

In Section \ref{202602050956}, we establish universal Auslander–Reiten triangles in Corollary \ref{202509171554}, and study their properties.
%
%
%
%

In Appendix \ref{202602050957}, we recall Happel's criterion for AR-triangles. 
In Appendix \ref{202511221618}, we collect the necessary  basic facts about adjoint $1$-morphisms in $2$-categories,  
formulated in terms of dg-bimodules (i.e., the dg-Morita $2$-category).

\subsection{Notations and  conventions}

Throughout this paper,  the symbol $\kk$ denotes a field. 
 By a dg-algebra  means dg-$\kk$-algebra. 
Let $A$ be dg-algebra. 
Unless otherwise stated, the term dg-$A$-modules refers to  left dg-$A$-modules.  
We denote the derived category of $A$  by $\sfD(A)$ and the perfect derived category of $A$ by $\perf A$. 
Other conventions of dg-modules are given in Section \ref{202509141802}.

Suppose  that  a functor $F= X \lotimes_{A} -: \perf A \to \perf B$ (or in a similar situation) is given by the derived tensor product with a dg-$B$-$A$-bimodule $X$. 
Then we denote  by $F_{\RHom}$ the induced morphism  $\RHom_{A}(M,N) \to \RHom_{B}(F(M), F(N))$.   
.

Although this is a rough and technically imprecise usage, for the sake of exposition we sometimes regard an object $X$ of a derived category as a complex, and whenever we write $x \in X$, it is understood to mean a homogeneous cocycle. Moreover, we denote the cohomology class of $x$  by the same symbol $x$. 

Given a morphism $f: M \to N$, we sometimes denote by the same symbol $f$ 
the morphism
\[
\cdots \lotimes_{A} X  \lotimes_{B} f\lotimes_{C} Y \lotimes_{D} \cdots:
\cdots \lotimes_{A} X  \lotimes_{B} M \lotimes_{C} Y \lotimes_{D} \cdots  \ \ \to \ \  
\cdots \lotimes_{A} X  \lotimes_{B} N\lotimes_{C} Y \lotimes_{D} \cdots 
\] 
it induces on larger tensor products containing $M$ and $N$.

We note that we do not deal with dg-categories in the present paper; however, a smooth and proper dg-category $\cA$ is derived equivalent to a smooth and proper dg-algebra $A$; see  \cite[Proposition 1.45]{Tabuada}, \cite[Corollary 2.13]{TV}.

\section*{Acknowledgment}
The author would like to thank M. Herschend for stimulating discussions
on quiver Heisenberg algebras.
The author is also grateful to H. Krause for pointing out
the results of Auslander--Carlson \cite{AC}
and Green--Marcos--Solberg \cite{GMS}.

\section{Hochschild (co)homology, tensor formulas of dual dg-bimodules and string diagrams}\label{202602050954}

\subsection{Preliminaries of dg-algebras and dg-bimodules}

\subsubsection{Euler characteristic and (super) trace}\label{202509171618}

We define  the  \emph{Euler characteristic} $\Euch(V)$ of $V\in \perf \kk$ to be    
\[
\Euch(V) := \sum_{i \in \ZZ} ( -1)^{i}\dim \tuH^{i}(V). 
\]
Let $V \in \perf \kk$ and $f: V \to V$, we define the \emph{(super) trace} $\Tr(f)$ to be 
\[
\Tr(f) := \sum_{i \in \ZZ} ( -1)^{i}\Tr\left [\tuH^{i}(f):  \tuH^{i}(V) \to \tuH^{i}(V) \right]
\]
where $\Tr(\tuH^{i}(f))$ in the right hand side is the usual trace of the linear map $\tuH^{i}(f): \tuH^{i}(V) \to \tuH^{i}(V)$.  
We point out  the obvious identity $\Euch (V) = \Tr(\id_{V} )$.

\subsubsection{}\label{202509141802} 

Let $A, B$ be dg-algebras. 
We denote the opposite dg-algebra of $A$ by $A^{\op}$. 
We may identify right dg-$A$-modules with (left) dg-$A$-modules 
and also  ($\kk$-central) dg-$A$-$B$-bimodules with dg-$A\otimes B^{\op}$-modules.
The isomorphism $A\otimes B^{\op} \xrightarrow{\cong} B^{\op} \otimes (A^{\op})^{\op}, \ a\otimes b \mapsto (-1)^{|a||b|} b\otimes a$ 
of dg-algebras induces an isomorphism $(-)^{\op}: \sfD (A\otimes B^{\op}) \xrightarrow{\cong } \sfD (B^{\op} \otimes (A^{\op})^{\op})$ 
of triangulated categories,  which identifies dg-$A$-$B$-bimodules with dg-$B^{\op}$-$A^{\op}$-bimodules in the canonical way. 
Whenever it is necessary to make  the relevant dg-algebras explicit, we denote  $(-)^{\op}_{A-B}$.

We denote by  $A^{\mre} =A\otimes A^{\op}$  the enveloping dg-algebra of $A$. 
Applying the above argument in  the case $B =A$, 
we obtain an isomorphism $(-)^{\op}_{A-A}: \sfD (A^{\mre}) \xrightarrow{\cong } \sfD ( (A^{\op})^{\mre})$ 
of triangulated categories. 
We note that the functor $(-)^{\op}_{A-A}$ sends the diagonal dg-$A$-bimodule $A \in \sfD(A^{\mre}) $ to 
the diagonal dg-$A^{\op}$-bimodule $A^{\op} \in \sfD((A^{\op})^{\mre}) $. 
In what follows, 
we identify 
the dg-algebras $(A^{\mre})^{\op}$ and $(A^{\op})^{\mre}$  via the 
canonical isomorphism 
$(A^{\mre})^{\op} \xrightarrow{ \cong } (A^{\op})^{\mre}, a\otimes b \mapsto a \otimes b$ which is the identity map of the underlying complexes. 
Thus we may regard  
$\sfD((A^{\mre})^{\op}) =\sfD((A^{\op})^{\mre})$.

%

\subsubsection{}
Recall that a dg-algebra $A$ is said to be  \emph{proper} if $\sum_{ i \in \ZZ} \dim \tuHH^{i}(A) < \infty$. 
We also recall that  a dg-algebra $A$ is said to be  \emph{smooth} if the (quasi-isomorphism class of) dg-$A$-$A$-bimodules $A$ belongs to the perfect derived category $\perf A^{\mre}$.
Note that  $A$ and $B$ are smooth and proper, then so is $A\otimes B^{\op}$.

\subsection{Hochschild (co)homology complex}

\begin{definition}\label{202509031449}
Let $A$ be a dg-algebra and $X \in \sfD(A^{\mre})$. 
We define the \emph{Hochschild cohomology complex} $\HHH^{\bullet}(A; X)$ and 
the \emph{Hochschild homology complex} $\HHH^{\bullet}(A; X)$ with   coefficient in $X$ 
are defined as follows:
\[
\begin{split}
\HHH^{\bullet}(A; X) & := \RHom_{A^{\mre}}(A, X),\\
\HHH_{\bullet}(A; X) &:= A\lotimes_{A^{\mre}} X.
\end{split}
\]

\end{definition}

We follow the standard abbreviation $\HHH^{\bullet}(A) =\HHH^{\bullet}(A;A)$ and $\HHH_{\bullet}(A) =\HHH_{\bullet}(A; A)$. 
We sometimes use the notation $X_{\natural} := \HHH_{\bullet}(A;X)$ for $X \in \sfD(A^{\mre})$. 
We note that for $X \in \sfD(A \otimes B^{\op})$ and $Y\in \sfD(B \otimes A^{\op})$, 
we have the following canonical isomorphism 
\begin{equation}\label{202509031502}
\HHH_{\bullet}(A; X\lotimes_{B} Y) = (X \lotimes_{B} Y)_{\natural} \cong (Y \lotimes_{A} X)_{\natural} = \HHH_{\bullet}(B;Y \lotimes_{A} X)
\end{equation}
which sends $x \lotimes y$ to $(-1)^{|x||y|}y \lotimes x$. 

In the case $B=\kk$, the above isomorphism specialized to the isomorphism 
\begin{equation}\label{202509151945}
\HHH_{\bullet}(A; M \otimes N) \cong N \lotimes_{A} M 
\end{equation}
for $M \in \sfD(A), \ N \in \sfD(A^{\op})$.

\subsection{Representation of derived tensor products by string diagrams}

In this paper, 
we need to manipulate dg-modules arising from  derived tensor products in complicated forms. 
For this purpose, we use string diagrams to represent these derived tensor products.

\begin{remark}
An important remark is that our string diagrams are different from those  used 
in the theory of $2$-categories (see e.g., \cite{Asashiba:book, NY}).
A string diagram in $2$-category theory consists of  pointed strings  on a plane (or a surface). 
In such a diagram, points, strings and region of a plane represent $2$-morphisms. $1$-morphisms and objects of a $2$-category, respectively. 
 Roughly speaking, our string diagram is obtained from such diagrams by cutting  the plane at a certain level. 
 \end{remark}

We represent a single dg-bimodule $X \in \sfD(A\otimes B^{\op})$ by the diagram 
\[
\textup{-} X \textup{-}
\] 
in which the left arm corresponds to the left dg-$A$-modules structure of $X$ and 
the right arm corresponds to the right dg-$B$-modules structure of $X$. 
Derived tensor products are represented by connecting   strings that emanating  from the relevant dg-bimodules. 
For example,  for $X \in \sfD(A\otimes B^{\op})$ and $Y \in \sfD(B\otimes C^{\op})$, we represent   $X \lotimes_{B} Y$ as $\textup{-} X \textup{-} Y \textup{-}$. 
So for example, the canonical isomorphisms 
$X \lotimes_{B} Y \cong A\lotimes_{A} X \lotimes_{B} Y \cong X \lotimes_{B} B \lotimes_{B} Y \cong X \lotimes_{B} Y \lotimes_{C} C$ 
are shown as 
\[
\textup{-} X \textup{-} Y \textup{-}\cong 
\textup{-}A \textup{-} X \textup{-} Y \textup{-} \cong \textup{-} X \textup{-} B \textup{-} Y \textup{-}
\cong \textup{-} X \textup{-} Y \textup{-} C\textup{-}.
\]

Hochschild homology complex 
$\HHH_{\bullet}(A;X) = X_{\natural}$ of $X \in \sfD(A^{\mre})$ is represented  as 
\centerline{
\begin{tikzpicture} 
 \draw[very thick] (0.1,0) arc [start angle=-70, end angle=250, radius=0.3];
\draw(0,0) node[midway]{$X$};
\end{tikzpicture}.}
%
By the canonical isomorphism \eqref{202509031502}, 
 for $X \in \sfD(A\otimes B^{\op})$ and $Y \in \sfD(B\otimes C^{\op})$, 
 the Hochschild homology  $\HHH_{\bullet}(A; X \lotimes_{B}Y) =(X\lotimes_{B} Y)_{\natural}$ has two representations 
 by  string diagrams 
 
 {\begin{center}{
\begin{tikzpicture} 
\draw (0, 0) node {$X$};
\draw (1, 0) node {$Y$};
\draw[very thick] (-0.1, 0) .. controls (-0.3, 0.125) and (-0.3, 0.375) .. (-0.1, 0.5);
\draw[very thick] (-0.1,0.5) -- (1.1,0.5);
\draw[very thick] (1.1, 0) .. controls (1.3, 0.125) and (1.3, 0.375) .. (1.1, 0.5);
\draw[very thick] (0.1,0) -- (0.9,0);

\draw[->] (1.4, 0.2) -- node[auto]{$\cong$} (2,0.2);

\begin{scope}[shift={(2.5,0)}]
\draw (0, 0) node {$Y$};
\draw (1, 0) node {$X$};
\draw[very thick] (-0.1, 0) .. controls (-0.3, 0.125) and (-0.3, 0.375) .. (-0.1, 0.5);
\draw[very thick] (-0.1,0.5) -- (1.1,0.5);
\draw[very thick] (1.1, 0) .. controls (1.3, 0.125) and (1.3, 0.375) .. (1.1, 0.5);
\draw[very thick] (0.1,0) -- (0.9,0);
\end{scope}
\end{tikzpicture}
}\end{center}}
%

As an exception, $\kk$-module structures are not represented in string diagrams. 
Moreover the $\kk$-module $\kk$ is often suppressed in string diagrams.  
Elementary examples are given in before Lemma \ref{202508201805} and in the proof of the same lemma.

We remark that our string diagrams only represent how dg-bimodules are connected by the derived tensor product. 
Thus, in particular, the way that  two strings cross,  is not an essential information.  
It is shown only to clarify how  the dg-bimodules are connected.

{\begin{center}{
\begin{tikzpicture}
\node (s) at (0,0) {}; 
\node (t) at (1.2,0) {};
\node (u) at (0,-0.8) {}; 
\node (v) at (1.2,-0.8) {};

\draw[very thick] (s) -- (v);
\draw[preaction={draw=white,line width=6pt}, very thick]
(u) --(t);

\draw (1.4, -0.35) -- (1.8, -0.35);
\draw (1.4, -0.45) -- (1.8, -0.45);

\begin{scope}[shift={(2,0)}]
\node (s) at (0,0) {}; 
\node (t) at (1.2,0) {};
\node (u) at (0,-0.8) {}; 
\node (v) at (1.2,-0.8) {};

\draw[very thick] (u) -- (t);
\draw[preaction={draw=white,line width=6pt}, very thick]
(s) --(v);
\end{scope}
\end{tikzpicture}
}\end{center}}

When we draw a string diagram,  
we identify $\sfD(A\otimes B^{\op})$ and $\sfD(B^{\op} \otimes (A^{\op})^{\op})$ 
via the isomorphism $(-)^{\op}$. 
This,  in particular, means that for $X \in \sfD(A\otimes B^{\op})$, 
 the corresponding object $X^{\op} \in \sfD(B^{\op} \otimes (A^{\op})^{\op})$ 
is depicted by the same diagram $ \textup{-} X \textup{-}$ as  that of $X$.

\subsection{Dual of dg-bimodules and their derived tensor products}

Let $A, B$ be smooth and proper dg-algebras, and  $X \in \perf (A\otimes B^{\op})$. 
We denote the left $A$-duality of $X$ by $X^{\lvvee} :=\RHom_{A}(X, A)$, which belongs to $\perf(B\otimes A^{\op})$. 
We denote by $\lcounit_{X}: X \lotimes_{B}X^{\lvvee} \to A, \ x \otimes \xi \mapsto (-1)^{|x||\xi|}\xi(x)$ the evaluation map. 
We denote by $\lunit_{X} : B \to X^{\lvvee} \lotimes_{A} X$  the composite morphism 
of the canonical morphism $B \to \RHom_{A}(X,X)$, which is  induced by the right dg-$B$-module structure of $X$, 
followed by the inverse  of 
 the canonical  isomorphism $ X^{\lvvee}\lotimes_{A} X = \RHom_{A}(X, A) \lotimes_{A} X \xrightarrow{\cong }\RHom_{A}(X, X)$. 
It is well-known and strightforward to check that these morphisms satisfies the triangle identities 
$(\lcounit_{X} \lotimes_{A} X)(X \lotimes_{B} \lunit_{X}) = \id_{X}$ and 
$(X^{\lvvee} \lotimes_{A} \lcounit_{X})(\lunit_{X}\lotimes_{B}X^{\lvvee}) = \id_{X^{\lvvee}}$.
\[
\begin{split}
&\id_{X} = \left[ \ X \xrightarrow{ \ X \lotimes_{B} \lunit_{X} \ } X \lotimes_{B} X^{\lvvee} \lotimes_{A} X \xrightarrow{ \ \lcounit_{X} \lotimes_{A} X \ } X  \ \right]\\ 
& \id_{X^{\lvvee}} = \left[ \  X^{\lvvee} \xrightarrow{ \ \lunit_{X} \lotimes_{B} X^{\lvvee} \ } X^{\lvvee} \lotimes_{A} X \lotimes_{B} X^{\lvvee} 
\xrightarrow{ \ X^{\lvvee} \lotimes_{A} \lcounit_{X} \ } X^{\lvvee} \ \right]
\end{split}
\]

We recall that 
these identities imply that 
the functor $X^{\lvvee}\lotimes_{A} -: \perf A \to \perf B$ is a right adjoint to  the functor $X\lotimes_{B}-: \perf B \to \perf A$. 
In the framework of Morita $2$-category, this simply says that 
the pair $(X, X^{\lvvee})$ (or more precisely the quadruple $(X, X^{\lvvee}, \lunit_{X}, \lcounit_{X})$) 
is an adjoint pair. 
We note that, as is recalled in Appendix \ref{202511221618} (in particular Proposition \ref{202509221237}), 
the triangle  identities characterize the object $X^{\lvvee}$ of $\perf (B\lotimes A^{\op})$.

Whenever it is necessary to make  the relevant dg-algebras explicit, we indicate this by the notation $\lunit_{X}^{A-B}$ and $\lcounit_{X}^{A-B}$.

\begin{example}\label{202511031429}
Let $X \in \perf (A\otimes B^{\op})$. 
To give an explicit description of the unit morphism $\lunit: B \to X^{\lvvee} \lotimes_{A} X$, 
we assume that $X$ is (represented by)  a cofibrant $A \otimes B^{\op}$-module 
whose  underlying  graded $A \otimes B^{\op}$-module is  graded free of finite rank. 
We take a $A\otimes B^{\op}$ basis $E =\{ e_{i} \}$ of $X$. 
For each $e_{i}$, let $e_{i}^{\lvvee}$ denotes the dual element belonging to $X^{\lvvee}$. 

We recall an explicit description of the canonical isomorphism 
$h: X^{\rvvee} \lotimes_{A} X \xrightarrow{ \cong } 
\RHom_{A}(X, X)$.  
Let $f \otimes x \in X^{\lvvee} \lotimes_{A} X$ be a homogeneous element. 
Then the image $h(f \otimes x)  \in \RHom_{A}(X, X)$ is the endomorphism of $X$  given by 
$y \mapsto  (-1)^{|x||y|}f(y)x$. 
Thus, under this isomorphism $h$, 
the homogeneous element $\sum_{i} (-1)^{|e_{i}|} e_{i}^{\lvvee} \otimes e_{i} \in X^{\lvvee} \lotimes_{A} X$ corresponds to the identity morphism $\id_{X} 
\in \RHom_{A}(X, X)$. 
Therefore,  the unit morphism $\lunit_{X} : B \to X^{\lvvee} \lotimes_{A} X$ is given by 
$\lunit_{X}(1):= \sum_{i} (-1)^{|e_{i}|} e_{i}^{\lvvee} \otimes e_{i}$.
\end{example}

We denote 
the right $B$-duality of $X$ by $X^{\rvvee}:=\RHom_{B^{\op}}(X, B)$, which belongs to $\perf(B\otimes A^{\op})$. 
Similarly, we have canonical morphisms 
$\runit_{X}: A \to X \otimes_{B} X^{\rvvee}$ and $\rcounit_{X}: X^{\rvvee} \lotimes_{A} X \to B$ that 
satisfies the triangle identities. 
Thus in particular, the functor $X^{\rvvee}\lotimes_{A}- : \perf A \to \perf B$ is a left adjoint functor to the functor $X \lotimes_{B}- : \perf B \to \perf A$. 
\[
\begin{xymatrix}@R=40pt{
\perf B \ar[d]_{X \lotimes_{B} - }\\
\perf A \ar@/^45pt/[u]^{X^{\rvvee} \lotimes_{A}- } \ar@/_45pt/[u]_{ X^{\lvvee} \lotimes_{A}- }
}\end{xymatrix}
\]

We collect a well-known result that shows that 
the inverse of an invertible complex $X$ is given by its left dual $X^{\lvvee}$ as well as by its right dual $X^{\rvvee}$. 
We note this follows from a $2$-categorical argument (see Proposition \ref{202509221331}).

\begin{lemma}\label{202509031414}
Let $A, B$ be  smooth and proper dg-algebras 
and $X \in \perf (A \otimes B^{\op})$. 
Assume that there exists $Y \in \perf(B \otimes A^{\op})$ such that $X \lotimes_{A} Y \cong A$ in $\perf(A^{\mre})$ 
and that $Y \lotimes_{A} X \cong B $ in $\perf(B^{\mre})$. 
Then 
$Y $ is canonically isomorphic to $X^{\lvvee}$ and to $X^{\rvvee}$ respectively. 
In particular, $X^{\lvvee}$ and $X^{\rvvee}$ are canonically isomorphic to each other. 
\end{lemma}

Let $C$ be another smooth and proper dg-algebra and $Y \in \perf (B\otimes C^{\op})$. 
Then it follows from a uniqueness of a right (resp. left) adjoint functor  that 
there are canonical isomorphisms 
\begin{equation}\label{202508261311}
(X \lotimes_{B} Y)^{\lvvee} \cong Y^{\lvvee} \lotimes_{B} X^{\lvvee}, \ \ 
(X \lotimes_{B} Y)^{\rvvee} \cong Y^{\rvvee} \lotimes_{B} X^{\rvvee}
\end{equation}
in $\perf(C \otimes A^{\op})$. 
We may identify $X$ with $X^{\lvvee\rvvee}$ (resp. $X^{\rvvee\lvvee}$) via the canonical isomorphism. 
We note that under these isomorphisms, we have $(\lcounit_{X})^{\rvvee} = \runit_{X}, \ (\lunit_{X})^{\rvvee} =\rcounit_{X}$ 
(resp. $(\rcounit_{X})^{\lvvee} =\lunit_{X}, \ (\runit_{X})^{\lvvee}=\lcounit_{X}$).
The precise statement is given in the following lemma. 

\begin{lemma}\label{202508211758}
We define  morphisms $\eta'_{X}: A \to X \lotimes_{B} X^{\rvvee}$ and $\epsilon'_{X}: X^{\rvvee} \lotimes_{A} X \to B$ 
to be 
\[
\begin{split}
&\eta'_{X}: A\cong A^{\rvvee} \xrightarrow{ \ (\lcounit_{X} )^{\rvvee} \ } (X \lotimes_{B} X^{\lvvee})^{\rvvee} \cong X\lotimes_{B} X^{\rvvee}, \\
& \epsilon'_{X}: X^{\rvvee} \lotimes_{A} X\cong (X^{\lvvee} \lotimes_{A} X)^{\rvvee} 
\xrightarrow{ \ (\lunit_{X})^{\rvvee} \ } B^{\rvvee} \cong B.
\end{split}
\]
Then we have $\eta'_{X} = \runit_{X}$ and $\epsilon'_{X} = \rcounit_{X}$.
\end{lemma}

It is clear that the isomorphism $(-)^{\op}$ of the first paragraph of Section \ref{202509141802} restricts to the isomorphism 
 $(-)^{\op}: \perf (A\otimes B^{\op}) \xrightarrow{\cong } \perf (B^{\op} \otimes (A^{\op})^{\op})$. 
 For $X \in \perf(A \otimes B^{\op})$ and $Y \in \perf(B \otimes C^{\op})$, there is a canonical isomorphism 
 $(X\lotimes_{B} Y)^{\op} \cong Y^{\op}\lotimes_{B^{\op}} X^{\op} $ in $\perf(C^{\op} \otimes (A^{\op})^{\op})$. 
 We also have canonical isomorphisms $(X^{\op})^{\lvvee} \cong (X^{\rvvee})^{\op}$ and $(X^{\op})^{\rvvee} \cong (X^{\lvvee})^{\op}$ 
 which satisfy obvious compatibility with the unit and co-unit morphisms. 
 Thus, recalling that we are identifying right dg-$A$-modules with dg-$A^{\op}$-modules, 
 we obtain the following  adjoint pairs: 
 \[
\begin{xymatrix}@R=40pt{
\perf A^{\op} \ar[d]_{- \lotimes_{A} X }\\
\perf B^{\op} \ar@/^45pt/[u]^{  -\lotimes_{B}X^{\lvvee} } \ar@/_45pt/[u]_{ - \lotimes_{B}X^{\rvvee} }
}\end{xymatrix}
\]

\subsubsection{}

We denote  the $A\otimes B^{\op}$-dulaity of $X$ by $X^{\vvee} := \RHom_{A\otimes B^{\op}} (X, A \otimes B)$, 
which belongs to $\perf (B\otimes A^{\op})$ 
and the $\kk$-duality of $X$ by $\tuD(X) :=\RHom_{\kk}(X, \kk)$. 
Since $A \otimes B^{\op} \cong (A\otimes B^{\op}) \otimes \kk^{\op}$, 
we may regard $X$ as an object of $\perf ((A\otimes B^{\op}) \otimes \kk^{\op})$. 
Thus applying to the above consideration, we obtain 
adjoint pairs $(\tuD(X),  X)$ and $(X,  X^{\vvee})$. 
\[
\begin{xymatrix}@R=40pt{
\perf \kk \ar[d]_{X \lotimes_{\kk} - }\\
\perf (A \otimes B^{\op}) \ar@/^45pt/[u]^{\tuD(X)  \lotimes_{A\otimes B^{\op} }- } \ar@/_45pt/[u]_{ X^{\vvee} \lotimes_{A\otimes B^{\op}}- }
}\end{xymatrix}
\]
We set the notations of the unit and the counit morphisms for these adjoint pairs as follows:
\[
\begin{split}
& \dunit_{X}:= \runit_{X}^{A\otimes B^{\op} -\kk}: A\otimes B^{\op} \to X  \otimes \tuD(X),\\
\ \ &\dcounit_{X}:= \rcounit_{X}^{A\otimes B^{\op} -\kk}: \tuD(X) \lotimes_{A\otimes B^{\op}} X \to \kk, \\
& \vunit_{X}:= \lunit_{X}^{A\otimes B^{\op} -\kk}: \kk \to X^{\vvee} \lotimes_{A\otimes B^{\op}} X, \\ 
&\vcounit_{X}:= \lcounit_{X}^{A\otimes B^{\op} -\kk}: X \otimes X^{\vvee} \to A \otimes B^{\op}. 
\end{split}
\]
 
One of the triangle identities 
 $(X \otimes \dcounit_{X})( \dunit_{X}  \lotimes_{A\otimes B^{\op}} X ) = \id_{X}$ of the adjoint pair $(\tuD(X), X)$ 
  tells that the following composition 
coincides with the identity morphism:
\[
\begin{split}
{}_{1}X_{2} & \cong {}_{1}A_{3} \lotimes_{A} {}_{3}X_{4} \lotimes_{B}{}_{4}B_{2} \cong ({}_{1}A_{3} \otimes {}_{4}B_{2}) \lotimes_{A\otimes B^{\op}} {}_{3} X_{4} \\
& \xrightarrow{ \ \ \dunit_{X}  \lotimes_{A\otimes B^{\op}} X \ \ } ({}_{1}X_{2} \otimes {}_{4}\tuD(X)_{3}) \lotimes_{A \otimes B^{\op}} {}_{3}X_{4}
\cong {}_{1}X_{2} \otimes ({}_{4}\tuD(X)_{3} \lotimes_{A \otimes B^{\op}} {}_{3}X_{4}) \\
&\xrightarrow{ \ \ (X \otimes \dcounit_{X}) \ \ } {}_{1}X_{2} \otimes \kk \cong {}_{1}X_{2}
\end{split}
\]
where where the indices $1,2,3,4$ indicate the places that are (potentially) connected by the derived tensor products.

Using a string diagram,   the above  composition is written in the following way: 

\begin{tikzpicture}

\begin{scope}[shift={(-1.5,-0.3)}]
\node (s) at (0,0) {$X$}; 

\draw[very thick] (-0.4,0) -- (s);
\draw[very thick] (s)  --  (0.4, 0);
\end{scope}

 \node at (-0.7,-0.3){$\cong$};

\node (s) at (0,0) {$A$}; 
\node (t) at (1,0) {$X$};
\node (u) at (0,-0.6) {$B$};

\draw[very thick] (s) --  (t);
\draw[very thick] (-0.5,0)--(s);
\draw[very thick, rounded corners=10pt] (t) -- (1.6, -0.45) -- (1,-1) -- (0, -1) -- (-0.5,-0.8) -- (u);
\draw[very thick] (u) --(1, -0.6);

\draw[->] (2,-0.3) --node[auto] {$\dunit_{X}$} (3, -0.3);

\begin{scope}[shift={(4,0)}]
\node (s) at (0,0) {$X$}; 
\node (t) at (1.2,0) {$X$};
\node (u) at (0,-0.8) {$\tuD(X)$};

\draw[very thick] (u) -- (t);
\draw[very thick] (s)  --  (1, -0. 6);
\draw[preaction={draw=white,line width=3pt}, very thick]
(s) --(1,-0.6);
\draw[very thick] (-0.5,0)--(s);
\draw[very thick, rounded corners=6pt] (t) -- (1.6, -0.45) -- (1.2,-1.1) -- (-0.5, -1.1) -- (-1,-0.98) -- (u);
\end{scope}

\draw[->] (6,-0.3) --node[auto] {$\dcounit_{X}$} (7, -0.3);

\begin{scope}[shift={(8,-0.3)}]
\node (s) at (0,0) {$X$}; 

\draw[very thick] (-0.4,0) -- (s);
\draw[very thick] (s)  --  (0.4, 0);
\end{scope}

\end{tikzpicture}
%
%
%
%
%
%
%
%
%
%
%
%
%
%
%
%
%
%
%

We note that in the rightmost term of the above diagram, 
we identify $X \otimes \kk$ with $X$. 
As noted earlier,  we will adopt this  identification tacitly.  

The canonical isomorphisms given in the next lemma constitute the technical heart of the paper.
\begin{lemma}\label{202508201805}
Let $A, B, C$ be smooth and proper dg-algebras and $X \in \perf (A\otimes B^{\op}), \ Y \in \perf(B \otimes C^{\op})$. 
Then there are canonical isomorphisms 
\[
\begin{split} 
& Y^{\vvee} \lotimes_{B} X^{\lvvee} \cong (X \lotimes_{B} Y)^{\vvee} \cong Y^{\rvvee} \lotimes_{B} X^{\vvee}, \\  
& Y^{\lvvee} \lotimes_{B} \tuD(X) \cong \tuD(X  \lotimes_{B} Y) \cong \tuD(Y) \lotimes_{B} X^{\rvvee}, \\ 
& \tuD(Y) \lotimes_{B} X^{\vvee} \cong (X \lotimes_{B} Y)^{\lvvee}, \ \ 
Y^{\vvee} \lotimes_{B} \tuD(X) \cong (X\lotimes_{B} Y)^{\rvvee}
\end{split}
\]
in $\perf(C \otimes A^{\op})$. 
\end{lemma}

\begin{remark}
Even if we drop some assumptions on $A, B, C$ and $X, Y$, the same isomorphisms 
in the above lemma and related results given in the sequel  hold (at least in part), 
provided that the relevant dualities are defined. 
\end{remark}

Although  the isomorphisms  in the lemma above can be proved  using basic isomorphisms involving $\lotimes$ and $\RHom$, 
we give an alternative proof based on the uniqueness of adjoint functors. 
This approach works in the context of 
$2$-categories and guarantees the canonicity of the isomorphisms, as explained below.

 We emphasize that  one advantage of use of $2$-categorical language  is  the uniqueness of adjoint functors (Proposition \ref{202509221237}) 
 and  as a consequence, for example, 
 we can verify that two isomorphisms  
 $(X \lotimes_{B} Y)^{\vvee} \xrightarrow{ \cong } Y^{\vvee} \lotimes_{B} X^{\lvvee}  \xrightarrow{ \cong } Y^{\rvvee} \lotimes_{B} B^{\vvee} 
 \lotimes_{B} X^{\lvvee}$ 
 and 
 $(X \lotimes_{B} Y)^{\vvee} \xrightarrow{ \cong } Y^{\rvvee} \lotimes_{B} X^{\vvee}  \xrightarrow{ \cong } Y^{\rvvee} \lotimes_{B} B^{\vvee} 
 \lotimes_{B} X^{\lvvee}$ constructed from  isomorphisms of Lemma \ref{202508201805} coincide.  
 More precisely, 
 the following diagram  is commutative: 
 \[
 \begin{xymatrix}{
 (X \lotimes_{B} Y)^{\vvee} \ar[rr]^-{\cong} \ar[dr]^{\cong} \ar[ddd]_{\cong}  &&
 Y^{\rvvee} \lotimes_{B} X^{\vvee} \ar[d]^{\cong} \\
&
( (X \lotimes_{B}  B) \lotimes_{B}  Y)^{\vvee} \ar[d]^-{\cong} \ar[r]^-{\cong}&
 Y^{\rvvee} \lotimes_{B} (X \lotimes_{B}B)^{\vvee} \ar[d]^{\cong} \\
&(X \lotimes_{B} ( B \lotimes_{B}  Y))^{\vvee}  \ar[d]_{\cong}  &
 Y^{\rvvee}\lotimes_{B}  ( B^{\vvee} \lotimes_{B} X^{\lvvee} )   \ar[d]_{\cong} \\
 Y^{\vvee} \lotimes_{B} X^{\lvvee} \ar[r]_-{\cong } &
 (B \lotimes_{B} Y)^{\vvee} \lotimes_{B} X^{\lvvee} \ar[r]_-{\cong } &
(Y^{\rvvee}\lotimes_{B}  B^{\vvee} )\lotimes_{B} X^{\lvvee}
 }\end{xymatrix}
 \]
 In the rest of the paper, we use such equalities of isomorphisms tacitly.

\begin{proof}[Proof of Lemma \ref{202508201805}]
We only deal with  the isomorphism $\tuD(Y) \lotimes_{B} X^{\vvee} \cong (X \lotimes_{B} Y)^{\lvvee}$. 
The other isomorphisms are proved in the same manner: we construct morphisms satisfying the triangle identities from the unit and counit associated with the factors of the derived tensor product.

To prove the existence of  the isomorphism, 
we construct 
morphisms  $\eta': C \to (\tuD(Y) \lotimes_{B}X^{\vvee}) \lotimes_{A} (X \lotimes_{B} Y)$ 
 and $\epsilon':   (X \lotimes_{B} Y) \lotimes_{C} (\tuD(Y) \lotimes_{B}X^{\vvee}) \to A$, 
 and then we show that these morphisms satisfy the triangle identities 
 \[
 \begin{split}
& (\epsilon' \lotimes (X \lotimes Y))((X\lotimes Y) \lotimes \eta') = \id_{X \lotimes Y}, \\  
& ((\tuD(Y) \lotimes_{B}X^{\vvee} )\lotimes \epsilon')(\eta' \lotimes (\tuD(Y) \lotimes_{B}X^{\vvee})) = \id_{\tuD(Y) \lotimes_{B}X^{\vvee}}.  
\end{split}
 \]
 Then the uniqueness of the right adjoint yields the desired canonical  isomorphism $\tuD(Y) \lotimes_{B} X^{\vvee} \cong (X \lotimes_{B} Y)^{\lvvee}$. 
 
A morphism $\eta': C \to (\tuD(Y) \lotimes_{B}X^{\vvee}) \lotimes_{A} (X \lotimes_{B} Y)$ 
is defined as follows 
\[
\begin{split}
C   \cong \kk  \otimes C &\xrightarrow{ \ \vunit_{X} \ }  (X^{\vvee} \lotimes_{B} X)_{\natural } \otimes C
  \cong  ( {}_{1}\!B_{2} \lotimes_{B} {}_{2}\! X^{\vvee}_{3} \lotimes_{A} {}_{3}\!X_{1}    )_{\natural } \otimes {}_{4}\!C_{5}\\
  & \xrightarrow{ \ \dunit_{Y} \ }  ( {}_{1}\!Y_{5} \  \ {}_{4}\!\tuD(Y)_{2} \lotimes_{B}   {}_{2}\! X^{\vvee}_{3} \lotimes_{A} {}_{3}\!X_{1}   )_{\natural } 
   \cong (\tuD(Y) \lotimes_{B} X^{\vvee}) \lotimes_{A} (X \lotimes_{B} Y)
\end{split}
\]
where the indices $1,2,3,4,5$ indicate the places that are (potentially) connected by the derived tensor products.
The string diagram of this morphisms is 

\vspace{8pt} 

\begin{tikzpicture}

\begin{scope}[shift={(-2.5,0.4)}]
\node (u) at (0,-0.8) {$C$}; 
\node (v) at (0.7,-0.8) {}; 
\node (v') at (-0.7,-0.8) {};  
\draw[very thick] (u) -- (v);
\draw[very thick] (v') -- (u);

\draw[->] (1,-0.8) -- node[auto] {$\vunit_{X}$} (1.6,-0.8);
\end{scope}

\node (s) at (0,0) {$B$}; 
\node (t) at (1.2,0) {$X^{\vvee}$};
\node (r) at (2.4,0) { $X$};

\node (u) at (0,-0.8) {$C$}; 
\node (v) at (0.7,-0.8) {}; 
\node (v') at (-0.7,-0.8) {}; 

\draw[very thick] (s) -- (t);
\draw[very thick] (t) -- (r);
\draw[very thick] (u) -- (v);
\draw[very thick] (v') -- (u);

%
\draw [very thick, rounded corners=10pt] (s) -- (-0.7,0.3) -- (0,0.6) --(2.4,0.6) --(3.1,0.3) --  (r);

\draw[->] (3.2,-0.4) --node [auto]{$\dunit_{Y}$} (4,-0.4);

\begin{scope}[shift={(5,0)}]
\node (s) at (0,0) {$Y$}; 
\node (t) at (1.2,0) {$X^{\vvee}$};
\node (r) at (2.4,0) { $X$};

\node (u) at (0,-0.8) {$\tuD(Y)$}; 
\node (v) at (1.2,-0.8) {}; 
\node (v') at (-1,-0.8) {}; 
%

\draw[very thick] (t) -- (r);
\draw[very thick] (v') -- (u);

\draw[very thick] (u)  --  (t);
\draw[preaction={draw=white,line width=2.5pt}, very thick](s) --(v);
\draw [very thick, rounded corners=10pt] (s) -- (-0.7,0.3) -- (0,0.6) --(2.4,0.6) --(3.1,0.3) --  (r);

\end{scope}
\end{tikzpicture}

We give the definition 
of  a morphism 
$\epsilon':   (X \lotimes_{B} Y) \lotimes_{C} (\tuD(Y) \lotimes_{B}X^{\vvee}) \to A$ 
in the string diagram below: 

\vspace{8pt} 

\begin{tikzpicture}
\node (s) at (0,0) {$\tuD(Y)$}; 
\node (t) at (1.2,0) {$X$};
\node (r) at (2.4,0) { $Y$};

\node (v) at (1.2,-0.8) {$X^{\vvee}$}; 

\draw[very thick] (t) -- (r);

\draw[very thick] (2.0, -0.8) --(v);
\draw[very thick] (0.76,0) --(t);

\draw[preaction={draw=white,line width=2.5pt}, very thick](s) --(v);
\draw [very thick, rounded corners=10pt] (s) -- (-0.7,0.3) -- (0,0.6) --(2.4,0.6) --(3.1,0.3) --  (r);

%
%

\draw[->] (3.4,-0.4) --node [auto]{$\vcounit_{X}$} (4.2,-0.4);

\begin{scope}[shift={(5,0)}]
\node (s) at (0,0) {$\tuD(Y)$}; 
\node (t) at (1.2,0) {$A$};
\node (r) at (2.4,0) { $Y$};

\node (v) at (1.2,-0.8) {$B$}; 
\node (w) at (2.4,-0.8) {};

\draw[very thick] (s) -- (v);
\draw[very thick] (v) -- (r);

\draw[very thick] (0.76,0) --(t);

\draw[preaction={draw=white,line width=3pt}, very thick](t) --(w);
\draw [very thick, rounded corners=10pt] (s) -- (-0.7,0.3) -- (0,0.6) --(2.4,0.6) --(3.1,0.3) --  (r);

%
%

\draw[->] (3.4,-0.4) --node [auto]{$\dcounit_{Y}$} (4.2,-0.4);
\end{scope}

\begin{scope}[shift={(9,0.4)}]
\node (v) at (1.2,-0.8) {$A$}; 

\draw[very thick] (v) -- (1.7,-0.8);
\draw[very thick] (0.7, -0.8) --(v);
\end{scope}

\end{tikzpicture}

Now it is tedious but straightforward to verify the triangle identities. 
\end{proof}

Applying  of Lemma \ref{202508201805} for the case 
$B=A, C= \kk, X=A$ and $Y =M$ regarded as an object of $\perf (A \otimes \kk^{\op})$, 
we obtain the following corollary. 

\begin{corollary}\label{202508181452}
Let $A$ be a smooth and proper dg-algebra. 
For a perfect dg-module $M \in \perf A$, we have the following canonical isomorphisms 
\[
\tuD(M) \lotimes_{A} A^{\vvee} \cong M^{\lvvee}, \ \ M^{\lvvee} \lotimes_{A} \tuD(A) \cong \tuD(M)
\]
in $\perf A^{\op}$ which are natrual in $M$. 
\end{corollary}

A consequence of Lemma \ref{202508201805} together with the isomorphism $A^{\lvvee} \cong  A \cong A^{\rvvee}$, 
 is a well-known fact that  
$A^{\vvee}$ and $\tuD(A)$ are the inverse to each other with respect to the derived tensor product.

\begin{corollary}\label{202508181504}
Let $A$ be a smooth and proper dg-algebra. 
Then we have canonical isomorphisms 
\[
A^{\vvee} \lotimes_{A} \tuD(A) \xrightarrow{ \ f  \  \ \cong \ }  A, \ \tuD(A) \lotimes_{A} A^{\vvee} \xrightarrow{ \ g  \ \ \cong \ }  A
\]
in $\sfD(A^{\mre})$. 
Moreover we have $\id_{\tuD(A)} = (g \lotimes \tuD(A) )(\tuD(A) \lotimes f^{-1}),$ \  
 $\id_{A^{\vvee}} = (A^{\vvee} \lotimes g)(f^{-1} \lotimes A^{\vvee})$.
\[
\begin{split}
\id_{\tuD(A) } &= \left[ \tuD(A)  \xrightarrow{ \ \tuD(A)  \lotimes f^{-1}  \ }
 \tuD(A) \lotimes_{A} A^{\vvee} \lotimes_{A} \tuD(A) 
\xrightarrow{ \ g \lotimes \tuD(A)  \ } \tuD(A) \right], \\
\id_{A^{\vvee} } &= \left[ A^{\vvee}  \xrightarrow{ \ f^{-1} \lotimes A^{\vvee}  \ } A^{\vvee}  \lotimes_{A} \tuD(A) \lotimes_{A} A^{\vvee} 
\xrightarrow{ \ A^{\vvee} \lotimes g \ } A^{\vvee}  \right].
\end{split}
\]
\end{corollary}

Combining  Lemma \ref{202508201805}  with the isomorphism $X^{\lvvee\rvvee} \cong  X \cong X^{\lvvee\lvvee}$ 
for an invertible dg-$A\otimes B^{\op}$-module $X$, 
we obtain the following result.

\begin{corollary}\label{202509031431}
Let $A, B$ be  smooth and proper dg-algebras 
and $X \in \perf (A \otimes B^{\op})$. 
Assume that $X$ is invertible, i.e.,  there exists $Y \in \perf(B \otimes A^{\op})$ such that $X \lotimes_{A} Y \cong A$ in $\perf(A^{\mre})$ 
and that $Y \lotimes_{A} X \cong B $ in $\perf(B^{\mre})$. 
Then there are the following  canonical isomorphisms
\[
\gamma_{X}: A^{\vvee} \lotimes_{A} X \xrightarrow{ \ \cong \ } X \lotimes_{B} B^{\vvee}, \ \  
\gamma'_{X}:  \tuD(A) \lotimes_{A} X \xrightarrow{ \ \cong \ } X \lotimes_{B} \tuD(B).  
\]
\end{corollary} 

It will be  shown in Section \ref{202511031303} that  
these isomorphisms are dg-enhanced version of well-known commutativity 
between the (inverse of) Serre functors and exact equivalences.  

Let $\Sigma$ denote the shift functor of dg-modules. 
For a homogeneous element $x \in X$ of dg-$A$-$B$-bimodule $X$, 
we denote by $\downarrow \! x$ (resp. $\uparrow\! x $) the corresponding element of $\Sigma X$ (resp. $\Sigma^{-1} A$). 
We note that $|\!\!\downarrow \! x | = |x| -1$ and $|\! \! \uparrow \! x| = |x| +1$.  
We denote the canonical isomorphism  by $\delta_{X} : X \lotimes_{B} \Sigma B \xrightarrow{ \ \cong \ } \Sigma A \lotimes_{A} X \cong \Sigma X, \ x \otimes \! \downarrow \! 1 
\mapsto (-1)^{|x| }\!\! \downarrow\! x$. 
By abusing notation, we denote the induced natural transformation $X \lotimes_{B} \Sigma(-) \to \Sigma X\lotimes_{B}-$ by $\delta_{X}$. 
Then, the pair $(X\lotimes_{A} -, \delta_{X})$ is an exact functor from  $\perf B$ to $\perf A$.

 The following lemma is a dg-enhanced version of a result, due to Van den Bergh (see \cite[Appendix A]{Bocklandt}), concerning the canonical commutation isomorphisms between the Serre functor 
$\sfS$ and the shift functor $\Sigma$.

\begin{lemma}\label{202511031307}
We have the equality $\gamma_{\Sigma A} = - \delta_{A^{\vvee}}$
of the morphisms from  $A^{\vvee} \lotimes_{A} \Sigma A $ to $\Sigma A \lotimes_{A} A^{\vvee}$.
\end{lemma}

\begin{proof}
For notational simplicity, we set $\downarrow := \downarrow\! 1 \in \Sigma A$ and 
$\uparrow := \uparrow \! 1 \in \Sigma^{-1}A$. 
The  pair $(\eta', \epsilon')$ of morphisms given below satisfies the triangle identities
\[
\begin{split}
&\eta': A \to  (\Sigma^{-1} A ) \lotimes_{A}(\Sigma A), \ 1 \mapsto \  \uparrow  \otimes \downarrow,\\
&\epsilon':  (\Sigma A ) \lotimes_{A}(\Sigma^{-1} A) \to A, \  \uparrow \otimes \downarrow  \ \mapsto 1.\\
\end{split}
\]
On the other hand, the unit morphism $\lunit_{\Sigma^{-1}A}: A \to (\Sigma^{-1} A)^{\lvvee} \lotimes_{A} (\Sigma^{-1} A)$ is 
given by $1 \mapsto - \uparrow ^{\lvvee} \otimes \uparrow$ (see Example \ref{202511031429}). 
According to 
 the proof of Proposition \ref{202509221237},  
 the canonical isomorphism $f: \Sigma A \xrightarrow{ \cong } (\Sigma^{-1} A)^{\lvvee}$ between 
the right adjoint bimodules of $\Sigma^{-1} A$
 is given by the composition 
\[
f: \Sigma A \xrightarrow{ \ \lunit_{\Sigma A}\lotimes \Sigma A \ } 
(\Sigma^{-1} A)^{\lvvee} \lotimes_{A} (\Sigma^{-1} A)\lotimes_{A} (\Sigma A) \xrightarrow{ \ (\Sigma^{-1} A)^{\lvvee} \lotimes \epsilon' \ } 
(\Sigma^{-1} A)^{\lvvee}.
\]
It is easy to see that we have $f(\downarrow ) = - \uparrow^{\lvvee}$. 

In the same way, we can show that 
the canonical isomorphism $g: \Sigma A \xrightarrow{ \cong } (\Sigma^{-1} A)^{\rvvee}$ of the left adjoint bimodules of $\Sigma^{-1}A$ 
is given by $g(\downarrow) = \uparrow^{\rvvee}$. 

Again, using the construction of the canonical isomorphism of the adjoint objects given in the proof of Proposition \ref{202509221237}, 
it is tedious but straightforward to check that for any perfect dg-$A^{\mre}$-modules $X$, 
the canonical morphisms 
\[
h: X^{\vvee} \lotimes_{A} (\Sigma^{-1} A)^{\lvvee} \xrightarrow{ \cong } (\Sigma^{-1}A \lotimes_{A} X)^{\vvee}, \ \ 
k: (\Sigma^{-1} A)^{\rvvee} \lotimes_{A} X^{\vvee} \xrightarrow{ \cong } (X \lotimes_{A} \Sigma^{-1} X)^{\vvee} 
\]
are roughly given by 
\[
\begin{split}
&h(\phi \otimes \uparrow^{\lvvee}): \Sigma^{-1}A \lotimes_{A} X \to A\otimes A, \ \uparrow \! \otimes x \mapsto \phi(x),\\
&k(\uparrow^{\rvvee}\! \otimes \phi):  X \lotimes_{A} \Sigma^{-1}A  \to A\otimes A, \ x \otimes \! \uparrow \  \mapsto (-1)^{|x| + |\phi|}\phi(x).\\
\end{split}
\]
Therefore applying above isomorphisms in the case $X = A$, 
the canonical isomorphism 
$\gamma': A^{\vvee} \lotimes_{A} (\Sigma^{-1} A)^{\lvvee} \xrightarrow{ h } (\Sigma^{-1}A)^{\vvee} \xrightarrow{ k^{-1} } (\Sigma^{-1} A)^{\rvvee} \lotimes_{A} A^{\vvee} $ is given by $\gamma'(\phi\otimes \! \uparrow^{\lvvee} ) = (-1)^{|\phi|} \uparrow^{\rvvee} \!\otimes \phi$. 
Since $\gamma_{\Sigma A} = (g^{-1} \otimes A^{\vvee})\gamma' (A^{\vvee} \otimes f)$, 
we see that $\gamma_{\Sigma A}( \phi \otimes \downarrow) = (-1)^{|\phi| +1}\downarrow \otimes \phi$ as desired. 
\end{proof}

\subsubsection{} 

Later we use the following lemma.

\begin{lemma}\label{202509010738}
The following morphisms coincide to each other 
\[
\begin{split}
(A\lotimes_{A} A) \otimes A^{\vvee} \xrightarrow{ \ {}_{l}\vcounit_{A} \ } A^{\mre}\lotimes_{A} A\cong A^{\mre},  \\
(A\lotimes_{A} A) \otimes A^{\vvee} \xrightarrow{ \ {}_{r}\vcounit_{A} \ } A\lotimes_{A} A^{\mre}\cong A^{\mre}
\end{split}
\]
where ${}_{l}\vcounit_{A}$ (resp. ${}_{r}\vcounit_{A}$) apply $\vcounit_{A}$ to the left (resp. right ) factor of 
$A \lotimes_{A} A$ and  $A^{\vvee}$. 
By a string diagram, the commutativity is explained as follows: 

\begin{tikzpicture}

\node (q) at (-1,0) {};
\node (s) at (0,0) {$A$}; 
\node (t) at (1,0) {$A$};
\node (r) at (2,0) { };

\node (v') at (-1,-0.8) {}; 
\node (u) at (0,-0.8) {$A^{\vvee}$}; 
\node (v) at (2,-0.8) {}; 
\node (w) at (2,-0.8) {};

\draw[very thick] (q) -- (s);
\draw[very thick] (s) -- (t);
\draw[very thick] (t) -- (r);
\draw[very thick] (u) -- (w);
\draw[very thick] (v') -- (u);
\draw[very thick] (v) -- (w);

%

\draw[->] (0.6,-1.2) --node [auto]{$\vcounit_{A}$} (0.6,-1.7);

\draw[->] (2.3,-0.4) -- node[auto]{$\cong$}(6.8, -0.4);

\begin{scope}[shift={(8,0)}]

\node (q) at (-1,0) {};
\node (s) at (0,0) {$A$}; 
\node (t) at (1,0) {$A$};
\node (r) at (2,0) { };

\node (v') at (-1,-0.8) {}; 
\node (u) at (0,-0.8) {}; 
\node (v) at (1,-0.8) {$A^{\vvee}$}; 
\node (w) at (2,-0.8) {};

\draw[very thick] (q) -- (s);
\draw[very thick] (s) -- (t);
\draw[very thick] (t) -- (r);
\draw[very thick] (v') -- (v);
\draw[very thick] (v) -- (w);

%

\draw[->] (0.6,-1.2) --node [auto]{$\vcounit_{A}$} (0.6,-1.7);
\end{scope}

\begin{scope}[shift={(0,-2)}]
\node (q) at (-1,0) {};
\node (s) at (0,0) {$A$}; 
\node (t) at (1,0) {$A$};
\node (r) at (2,0) { };

\node (v') at (-1,-0.8) {}; 
\node (u) at (0,-0.8) {$A$}; 
\node (v) at (1,-0.8) {}; 
\node (w) at (2,-0.8) {};

\draw[very thick] (q) -- (s);
\draw[very thick] (t) -- (r);
\draw[very thick] (v') -- (u);

\draw[very thick] (u)  --  (t);
\draw[preaction={draw=white,line width=6pt}, very thick, rounded corners=3pt](s) --(1,-0.8)--(w);

\draw[->] (2.3,-0.4) --node [auto]{$\cong$} (2.8,-0.4);

\end{scope}

\begin{scope}[shift={(4,-2)}]
\node (q) at (-1,0) {};
\node (s) at (0.4,0) {$A$}; 
\node (t) at (1,0) {};
\node (r) at (2,0) { };

\node (v') at (-1,-0.8) {}; 
\node (u) at (0.4,-0.8) {$A$}; 
\node (v) at (1,-0.8) {}; 
\node (w) at (2,-0.8) {};

\draw[very thick] (q) -- (s);
\draw[very thick] (t) -- (r);
\draw[very thick] (v') -- (u);

\draw[very thick, rounded corners=3pt] (u)  --  (1,0)--(r);
\draw[preaction={draw=white,line width=6pt}, very thick, rounded corners=3pt](s) --(1,-0.8)--(w);

\end{scope}

\begin{scope}[shift={(8,-2)}]
\node (q) at (-1,0) {};
\node (s) at (0,0) {$A$}; 
\node (t) at (1,0) {$A$};
\node (r) at (2,0) { };

\node (v') at (-1,-0.8) {}; 
\node (u) at (0,-0.8) {}; 
\node (v) at (1,-0.8) {$A$}; 
\node (w) at (2,-0.8) {};

\draw[very thick] (q) -- (s);
\draw[very thick] (s) -- (t);
\draw[very thick] (v') -- (v);

\draw[very thick, rounded corners=3pt] (v)  --  (1.5,0)--(r);
\draw[preaction={draw=white,line width=6pt}, very thick, rounded corners=3pt](t) --(1.5,-0.8)--(w);

\draw[<-] (-1.8,-0.4) --node [auto]{$\cong$} (-1.3,-0.4);

\end{scope}

\end{tikzpicture}

where the left vertical arrow is $ {}_{l}\vcounit_{A}$ and the right is $ {}_{r}\vcounit_{A}$.

\end{lemma}

\begin{proof}
Let $X \in \perf A^{\mre}$. 
It is easy to see that the following diagram is commutative
\[
\begin{xymatrix}@R=10pt{
(A\lotimes_{A} X) \otimes X^{\vvee} \ar[r]^-{\cong} \ar[d]_{\vcounit_{X}} & 
X \otimes X^{\vvee} \ar[d]_{\vcounit} \ar[r]^-{\cong} & 
(X \lotimes_{A} A)\otimes X^{\vvee} \ar[d]^{\vcounit_{X}} \\
A\lotimes_{A} A^{\mre} \ar[r]_{\cong} &
A \ar[r]_{\cong} & A^{\mre} \lotimes_{A} A
}\end{xymatrix}
\]
Observe  that  the morphism $A \lotimes_{A} A \xrightarrow{ \cong }A \xrightarrow{ \cong }A \lotimes_{A} A$ 
which is a factor of the upper horizontal morphism in the above diagram of the case $X =A$, is the identity morphism of $A \lotimes_{A} A$.  
This observation completes the proof.
\end{proof}

\subsubsection{}
To simplify notation, we set $(-)^{\op^{2}} = ((-)^{\op})^{\op}$. 
Let $A, B$ be smooth and proper dg-algebras 
and $X \in \perf (A \otimes B^{\op})$. 
Observe that if  we regard $X \in \perf ( (A\otimes B^{\op})\otimes \kk^{\op})$, the $X^{\op}_{A-B}$ corresponds to 
$X^{\op}_{(A\otimes B^{\op})-\kk}$ 
under the isomorphism $\perf (B^{\op} \otimes A^{\op^{2}}) \cong \perf (\kk^{\op} \otimes  (A\otimes B^{\op})^{\op^{2}})$ 
induced from the canonical isomorphism 
$(B^{\op} \otimes A^{\op^{2}}) \cong \kk^{\op} \otimes  (A\otimes B^{\op})^{\op^{2}}$ of dg-algebras.

We identify $\perf (A \otimes B^{\op})^{\op \mre}$ 
with $ \perf (B \otimes A^{\op})^{\mre}$ 
via  the functor induced from the
 canonical isomorphism $(A \otimes B^{\op})^{\op} \cong B \otimes A^{\op}$ of dg-algebras. 
Thus we denote by 
$(-)^{\op}: \perf (A \otimes B^{\op})^{\mre}  \to \perf (B \otimes A^{\op})^{\mre}$ 
the functor corresponding under this identification, 
 to 
$(-)^{\op}_{(A \otimes B^{\op}) -(A \otimes B^{\op})}: 
 \perf (A \otimes B^{\op})^{\mre}  \to \perf (A \otimes B^{\op})^{\op \mre} $.

For the proof of Theorem \ref{202508231520}, 
we need to point out the following lemma which is a special case of the argument  given  after Lemma \ref{202508211758}.

\begin{lemma}\label{202509151654}

Under the identifications 
\[B \otimes A^{\op} \cong (A \otimes B^{\op})^{\op}, \ \ (X^{\op}_{(A\otimes B^{\op})-\kk})^{\lvvee} \cong \tuD(X)^{\op}_{\kk -(A\otimes B^{\op})},\] 
we have 
 $\lunit_{X^{\op}}^{\kk^{\op} - (B\otimes A^{\op})} = (\dunit_{X})^{\op}$
 and  $\lcounit_{X^{\op}}^{\kk^{\op} - (B\otimes A^{\op})} = (\dcounit_{X})^{\op}$. 
\[
\begin{xymatrix}@R=10pt@C=18pt{
B \otimes A^{\op} \ar[dd]_{\cong} \ar[rr]^-{ \lunit_{X^{\op}}^{\kk^{\op} - (B\otimes A^{\op})} } && 
(X^{\op})^{\lvvee} \otimes X^{\op}  \ar[d]^{\cong} &
X^{\op} \lotimes_{B\otimes A^{\op}} (X^{\op})^{\lvvee}   \ar[d]^{\cong} 
\ar[rr]^-{ \lcounit_{X^{\op}}^{\kk^{\op} - (B\otimes A^{\op})} }
& & 
\kk^{\op} \ar[dd]^{\cong} \\
&& 
\tuD(X)^{\op} \otimes X^{\op} \ar[d]^{\cong} &
X^{\op} \lotimes_{B\otimes A^{\op}} \tuD(X)^{\op}  \ar[d]^{\cong} & &\\
(A\otimes B^{\op})^{\op} \ar[rr]_-{(\dunit_{X})^{\op}} & &
(X \otimes \tuD(X))^{\op} & 
(\tuD(X) \lotimes_{A\otimes B^{\op}} X)^{\op} \ar[rr]_-{(\dcounit_{X})^{\op}} &&
\kk^{\op}
}\end{xymatrix}
\]

Thus in particular if we identify $\perf (\kk^{\op})^{\mre} $ with $\perf \kk^{\mre} \cong \perf \kk$ via the canonical isomorphism, 
 then 
 $\lcounit_{X^{\op}}^{\kk^{\op} - (B\otimes A^{\op})} = \dcounit_{X}$. 

\end{lemma}

\section{Mukai pairing and Serre duality}\label{202602050955}

\subsection{A canonical duality of Hochschild (co)homology group, Serre duality and Mukai pairing}

\subsubsection{A canonical duality of Hochschild (co)homology group}\label{202509191836}

Our starting point is the following duality of Hochschild homology and Hochschild cohomology, which is an immediate cunsequence of $\lotimes-\RHom$-adjunction.  

\begin{proposition}
\label{202508171520}
Let $A$ be a dg-algebra and $X \in \sfD(A^{\mre})$.
We have an isomorphism 
\[
\tuD\HHH_{\bullet}(A;X) \cong \HHH^{\bullet}(A; \tuD(X))
\]
which is natrual in $X$. 
\end{proposition}

\begin{proof}
Let $R$ be a dg-algebra and $M \in \sfD(R), \ N \in \sfD(R^{\op})$. 
We have the following isomorphism 
\[
\tuD(N\lotimes_{R} M ) = \RHom_{\kk}(N \otimes M, \kk) \cong 
\RHom_{R}(M, \RHom_{\kk}(N, \kk)) = \RHom_{R}(M, \tuD(N)) 
\]
which is natural in $M$ and $N$.   
Substituting $R, M, N$ with $A^{\mre}, A, X$, we complete the proof. 

We note that in the case that $A$ is smooth and proper and $X \in \perf A^{\mre}$, 
we can deduce the desired isomorphism from Lemma \ref{202508201805} as follows
\[
\tuD(\HHH_{\bullet}(A; X)) =\tuD(A \lotimes_{A^{\mre}} X)\cong \tuD(X)  \lotimes_{A^{\mre}} A^{\vvee} \cong \HHH^{\bullet}(A; \tuD(X))
\]
where we regard $A \in \perf(\kk \otimes (A^{\mre})^{\op}) $ and $X \in \perf(A^{\mre} \otimes \kk^{\op})$. 
\end{proof}

We denote by 
\[
\agl{-,+}_{X}:  \HHH_{\bullet}(A, X) \otimes \HHH^{\bullet}(A, \tuD(X))  \to \kk
\]
the non-degenerate pairing induced from the isomorphism of Proposition \ref{202508171520}.

We draw the string diagram of the pairing $\agl{-,+}_{X}$  in the case that $A$ is smooth and proper and $X \in \perf A^{\mre}$, 

\vspace{4pt}

\begin{equation}\label{202509010636}
\begin{tikzpicture}
\node (s) at (0,0) {$A$}; 
\node (t) at (1.2,0) {$X$};
\node (u) at (0,-0.8) {$A^{\vvee}$}; 
\node (v) at (1.2,-0.8) {$\tuD(X)$};

\draw[very thick] (s) -- (t);
\draw[very thick] (u) -- (v);

\draw [very thick, rounded corners=10pt] (s) -- (-0.7,0.3) -- (0,0.6) --(1.2,0.6) --(1.9,0.3) --  (t);
\draw [very thick, rounded corners=10pt] (u) -- (-0.7,-1.1) -- (0,-1.4) --(1.2,-1.4) --(1.9,-1.1) --  (v);

\draw[->] (2,-0.4) --node [auto]{$\vcounit_{A}$} (3,-0.4);

\begin{scope}[shift={(4,0)}]

\node (s) at (0,0) {$A$}; 
\node (t) at (1.2,0) {$X$};
\node (u) at (0,-0.8) {$A$}; 
\node (v) at (1.2,-0.8) {$\tuD(X)$};

%

\draw[very thick] (u)  --  (t);
\draw[preaction={draw=white,line width=3pt}, very thick](s) --(v);
\draw [very thick, rounded corners=10pt] (s) -- (-1,0.3) -- (0,0.6) --(1.2,0.6) --(2.2,0.3) --  (t);
\draw [very thick, rounded corners=10pt] (u) -- (-1,-1.1) -- (0,-1.4) --(1.2,-1.4) --(2.2,-1.1) --  (v);
\end{scope}

\draw[->] (6.5,-0.4) --node [auto]{$\dcounit_{X}$} (7.5,-0.4);

\node at (8,-0.4){$\kk$};

\end{tikzpicture}
\end{equation}
%
%
%
%
%
%
%

It follows from the naturality of the isomorphism that 
for a morphism $f: X \to Y $ in $\sfD(A^{\mre})$, 
 there exists the following commutative diagram
\begin{equation}\label{202508171549}
\begin{xymatrix}@C=60pt{
\HHH_{\bullet}(A; X) \otimes \HHH^{\bullet}(A; \tuD(Y)) \ar[r]^-{ f_{*} \otimes \id } \ar[d]_-{\id \otimes f^{*}} &
\HHH_{\bullet}(A; Y) \otimes \HHH^{\bullet}(A; \tuD(Y))  \ar[d]^{\agl{-,+}_{Y}} \\
\HHH_{\bullet}(A; X) \otimes \HHH^{\bullet}(A; \tuD(X))  \ar[r]_-{\agl{-,+}_{X}} & \kk 
}\end{xymatrix}
\end{equation}
where we set 
$f_{*} := \HHH_{\bullet}(A; f)$ and $f^{*}:= \HHH^{\bullet}(A; \tuD(f))$. 
Roughly speaking, this shows that we have the identity 
\[
\agl{\xi, f^{*}(\theta)}_{X} = \agl{f_{*}(\xi), \theta}_{Y}
\]
for $\xi \in \HHH_{\bullet}(A;X), \ \theta \in \HHH^{\bullet}(A; \tuD(Y))$. 
Thus, we refer the above commutative diagram the reciprocity law of the canonical pairing.

We point out the other reciprocity law. 
For this purpose, we assume that $A$ and $B$ are smooth and proper. 
Let $X \in \perf (A \otimes B^{\op})$ and $Y \in \perf B^{\mre}$. 
For simplicity, we set ${}^{X}Y := X \lotimes_{B} Y \lotimes_{B} X^{\lvvee}$. 
We define a morphism $\tilde{\Phi}_{X, Y}: \HHH_{\bullet}(B; Y) \to \HHH_{\bullet}(A; {}^{X}Y)$ to  be 
the composition 
\[
\begin{split}
\tilde{\Phi}_{X, Y}: 
\HHH_{\bullet}(B;Y) =  B\lotimes_{B^{\mre}} Y  
\xrightarrow{ \ \lunit_{X} \lotimes \id \ } \ 
 & (X^{\lvvee} \lotimes_{A} X) \lotimes_{B^{\mre}} Y \\
\cong &  A \lotimes_{A^{\mre}} ( X \lotimes_{B} Y \lotimes_{B} X^{\lvvee} ) \\
 = & \HHH_{\bullet}(A;  {}^{X}Y).
\end{split}
\]
We define a morphism $\tilde{\Psi}_{X, Y}: \HHH^{\bullet}(A; \tuD({}^{X}Y)) 
\to \HHH^{\bullet}(B;X) $  to be the composition 
\[
\begin{split}
\tilde{\Psi}_{X, Y}: \HHH^{\bullet}(A; \tuD({}^{X}Y))    
&= \RHom_{A^{\mre}}(A, \tuD(  X \lotimes_{B} Y \lotimes_{B} X^{\lvvee}) ) \\
&\cong \RHom_{A^{\mre}}(A, X^{\lvvee\lvvee} \lotimes_{B} \tuD(Y) \lotimes_{B} X^{\rvvee}) \\
& \cong \RHom_{B^{\mre}}(X^{\lvvee} \lotimes_{A} X ,\tuD(Y)) \\
 & \xrightarrow{ \ \RHom(\lunit_{X}, \tuD(Y))  \ } 
\RHom_{B^{\mre}}(B, \tuD(Y)) =
 \HHH^{\bullet}(B;X)
\end{split}
\]
where for the second isomorphism we use Lemma \ref{202508201805} and 
the third is deduced from the adjoint properties. 

Since the isomorphisms of Lemma \ref{202508201805} are induced from the   adjoint properties 
of relevant dg-bimodules, it is straightforward to check the following commutativity.

\begin{proposition}\label{202509191815}
In the above situation, we have the following commutative diagram: 
\[\begin{xymatrix}@C=60pt{
\HHH_{\bullet}(B, Y) \otimes \HHH^{\bullet}(A, \tuD({}^{X}Y)) \ar[r]^-{ \id  \otimes \tilde{\Psi}_{X,Y} } \ar[d]_-{\tilde{\Phi}_{X, Y}  \otimes \id } &
\HHH_{\bullet}(B, Y) \otimes \HHH^{\bullet}(B, \tuD(Y))  \ar[d]^{\agl{-,+}_{Y}} \\
\HHH_{\bullet}(A, {}^{X}Y) \otimes \HHH^{\bullet}(A, \tuD({}^{X} Y))  \ar[r]_-{\agl{-,+}_{ {}^{X}Y } } & \kk 
}\end{xymatrix}
\]
\end{proposition}

We give a description of $\tilde{\Phi}_{X, Y}$ under the canonical identifications  
%
%
%
$\HHH_{\bullet}(B; Y) \cong \RHom_{B^{\mre}}(B^{\vvee}, Y)$ and $\HHH_{\bullet}(A; {}^{X} Y) \cong \RHom_{A^{\mre}}(A^{\vvee}, {}^{X}Y)$. 
Observe that by Lemma Lemma \ref{202508201805}, we have ${}^{X}(B^{\vvee}) \cong (X \lotimes_{B} X^{\lvvee})^{\vvee}$.  
Now, the morphims $\tilde{\Phi}_{X, Y}$ coincides with the following composition 
\begin{equation}\label{202509201657}
\begin{split}
\HHH_{\bullet}(B;Y ) 
 \cong \RHom_{B^{\mre}}(B^{\vvee}, Y ) 
 \xrightarrow{ \  ({}^{X}(-) )_{\RHom} \ } & \RHom_{A^{\mre}}({}^{X}(B^{\vvee}), {}^{X} Y) \\
 \cong \ \ \ \ \ \ &\RHom_{A^{\mre}}((X \lotimes_{B} X^{\lvvee})^{\vvee}, {}^{X}Y) \\
 \xrightarrow{ \ \RHom( \lcounit_{X}^{\vvee}, {}^{X}Y) \ } & \RHom_{A^{\mre}}(A^{\vvee}, {}^{X} Y ) \\
 \cong \ \ \ \ \ \ & \HHH_{\bullet}(A; {}^{X}Y).
\end{split}
\end{equation}

\subsubsection{Serre duality}\label{202511031303}

In this section, we show that the isomorphism of Proposition \ref{202508171520} specializes to Serre duality, 
and we discuss its basic properties. 
In particular, Corollary \ref{202509201321} and Lemma \ref{202511031307}, proved earlier,  provide dg-enhanced versions of basic properties 
of the Serre functor.

\begin{proposition}[{Serre duality}]\label{202508171742}
Let $A$ be a smooth and proper dg-algebra and $M \in \perf A,\ N \in \sfD(A)$. 
Then there exists an isomorphism 
\[
\tuD\RHom_{A}(N, \tuD(A)\lotimes_{A} M) \cong \RHom_{A}(M, N)
\]
which is natrual in $M$ and $N$. 
\end{proposition}

\begin{proof}
By \eqref{202509151945}, we have 
\begin{equation}\label{202508171811}
\RHom_{A}(M,N) \cong M^{\lvvee} \lotimes_{A}N \cong  \HHH_{\bullet}(A; N \otimes M^{\lvvee}). 
\end{equation}
On the other hand, 
by Lemma \ref{202508181452}, we have 
\begin{equation}\label{202508171812}
\begin{split}
\RHom_{A}(N , \tuD(A)\lotimes_{A}M) \cong & 
\RHom_{A}(N , \tuD(M^{\lvvee}))\\
 \cong &
\HHH^{\bullet}(A; \RHom_{\kk}(N, \tuD(M^{\lvvee})))\\
\cong &\HHH^{\bullet}(A; \tuD(N \otimes M^{\lvvee})). 
\end{split}
\end{equation}
Applying Proposition \ref{202508171520} to the case $X= N \otimes M^{\lvvee}$, 
we complete the proof.
\end{proof}

\begin{remark}

Proposition \ref{202508171742} holds for a dg-algebra $A$ which is not necessarily smooth or proper, since it is well known and easy to prove that there is an isomorphism 
$\tuD(A)\lotimes_{A}M \cong \tuD(M^{\lvvee})$ for $M \in \perf A$.
\end{remark}

It follows from Proposition \ref{202508171742} that 
 there exists a non-degenerate pairing 
\[
\serreagl{-,+}: \RHom_{A}(M,N) \otimes \RHom_{A}(N, \tuD(A) \lotimes_{A} M)  \to \kk,  
\]
which is called the \emph{Serre pairing}.

Assume that $A$ is smooth and proper. 
Since  $A^{\vvee} \lotimes_{A} \tuD(A) \cong A$ by Corollary \ref{202508181504},  
the above pairing induces another non-degenerate pairing 
\[
\invserreagl{-,+}: \RHom_{A}(M,N) \otimes \RHom_{A}( A^{\vvee} \lotimes_{A} N, M)  \to \kk,  
\]
which is also called the \emph{Serre pairing}. 
This shows that the endofunctor $\sfS := \tuD(A)\lotimes_{A} -$  is a Serre functor of $\perf A$ 
and $\sfS^{-1} := A^{\vvee} \lotimes -$ is its quasi-inverse. 
Whenever it is necessary to make the relevant dg-algebras explicit, we denote  $\sfS_{A}, \sfS_{A}^{-1}$. 

The following lemma immediately follows from the triangle identities of the adjoint pair $(M^{\lvvee}, M)$. 

\begin{lemma}\label{202508311741}
Let $A$ be a smooth and proper dg-algebra and $M \in \perf A$. 
We define  a morphism $T_{M}: \RHom_{A}(M, \tuD(A) \lotimes_{A}M) \to \kk$ to be the composition 
\[
 \RHom_{A}(M, \tuD(A) \lotimes_{A}M) \cong (\tuD(A) \lotimes_{A} M \otimes M^{\lvvee})_{\natural} 
 \xrightarrow{ \ \lcounit_{M} \ } (\tuD(A) \lotimes_{A} A)_{\natural} 
 \xrightarrow{ \ \dcounit_{A} \ } \kk. 
\]
Then the morphism $T_{M}$ coincides with the composition 
\[
 \begin{split} 
 \RHom_{A}(M, \tuD(A) \lotimes_{A} M) & \ \ \ \cong \ \ \  \kk \otimes \RHom_{A}(M, \tuD(A) \lotimes_{A} M) \\ 
 &\xrightarrow{ \ \lunit_{M} \ }  \RHom_{A}(M,M) \otimes \RHom_{A}(M, \tuD(A) \lotimes_{A} M) \\
 & \xrightarrow{  \ \ \serreagl{-,+} \ \  } \kk 
\end{split} 
\]
where in the first arrow we implicitly use the isomorphism $\RHom_{A}(M,M) \cong M^{\lvvee} \lotimes_{A} M$. 
\end{lemma}

Roughly speaking we have $T_{M}(f) = \serreagl{\id_{M}, f}$. 
In \cite{CW}, the morphism $T_{M}$ is called trace and denoted by $\Tr$ (but we do not follow this terminology).

We state a reciprocity law for Serre duality deduced from Proposition \ref{202509191815}. 
For this purpose, let $A, B$ be smooth and proper dg-algebras  
$X \in \perf ( A\otimes B^{\op})$ and $M ,N \in \perf B$. 
For simplicity, we denote by $F:= X \lotimes_{B}-: \perf B \to \perf A$. 
Recall that by Lemma \ref{202508201805}, there is the canonical isomorphism $A^{\vvee} \lotimes_{A} X \cong X^{\rvvee\rvvee} \lotimes_{B} B^{\vvee}$. 
We define a morphism 
$\phi_{X}: \RHom_{A}(\sfS^{-1}_{A}F (N), F(M)) \to \RHom_{B}(\sfS_{B}^{-1}N, M)$ 
to be the following composition:
\[
\begin{split}
\RHom_{A}(\sfS^{-1}_{A}F (N), F(M)) &= \RHom_{A}(A^{\vvee} \lotimes_{A} X \lotimes_{B}N ,X \otimes_{B} M) \\
& \cong \RHom_{A}(X^{\rvvee\rvvee}  \lotimes_{B} B^{\vvee}  \lotimes_{B}N ,X \otimes_{B} M)\\
& \cong \RHom_{B}(X^{\rvvee} \lotimes_{A} X^{\rvvee\rvvee}  \lotimes_{B}B^{\vvee}  \lotimes_{B}N ,M)\\
& \xrightarrow{ \ \RHom(\runit_{X^{\rvvee}}\otimes \id, M) \ } \RHom_{B}( B^{\vvee}  \lotimes_{B}N,  M)\\
& =\RHom_{B}(\sfS^{-1}_{B}(N), M)
\end{split}
\]

\begin{proposition}\label{20250920136}
In the above setting, we have the following commutative diagram: 
\[\begin{xymatrix}@C=20pt{
{}_{B}(M,N) \otimes {}_{A}(\sfS^{-1}_{A}F (N), F(M))  
 \ar[r]^-{ \id  \otimes \phi_{X} } \ar[d]_-{F_{\RHom} \otimes \id} &
{}_{B}(M,N)  \otimes {}_{B}(\sfS_{B}^{-1}(N), M)  \ar[d]^{\invserreagl{-,+}} \\
{}_{A}(F(M), F(N) )  \otimes {}_{A}(\sfS^{-1}_{A}F (N), F(M))   
 \ar[r]_-{\invserreagl{-,+}} & \kk 
}\end{xymatrix}
\]
where ${}_{A}(-,+)$ and ${}_{B}(-,+)$ are the abbreviations for $\RHom_{A}(-,+)$ and $\RHom_{B}(-,+)$, respectively. 
%
\end{proposition}

Now assume further that $X \in \perf(A\otimes B^{\op})$ is invertible 
i.e., satisfies the condition of Lemma \ref{202509031431}. 
Then functor $F = X \lotimes_{B}-: \perf B \to \perf A$ is an equivalence 
and the morphism $\phi_{X}$ is an isomorphism. 
The inverse norphism $\phi_{X}^{-1}$ is given as follows. 
By Lemma \ref{202509031431}, there is a canonical isomorphism 
$\gamma_{X}:   A^{\vvee} \lotimes_{A} X \xrightarrow{ \ \cong  \ } X \lotimes_{B}B^{\vvee} $. 
We denote the induced natural isomorphism $\gamma_{X} : \sfS^{-1}_{A}F \to F\sfS^{-1}_{B}$ by the same symbol. 
Now the inverse $\phi^{-1}_{X}$ is given by the composition: 
\[
\begin{split}
\RHom_{B}(\sfS_{B}^{-1}(N), M) 
& \xrightarrow{ \ F_{\RHom} \ } \RHom_{A}(F\sfS_{B}^{-1}(N), F(M)) \\
& \xrightarrow{ \ \RHom(\gamma_{X, N} , F(M))\ } \RHom_{A}(\sfS_{A}^{-1}F(N), F(M)). \\
\end{split}
\]

It is well-known that Serre functors are compatible with equivalences of triangulated categories.  
The following corollary shows that if we deal with derived categories of  dg-modules, 
the compatibility is induced from the dg-enhancement

\begin{corollary}\label{202509201321}
In the above situation, the following diagram is commutative: 
\[
\begin{xymatrix}@R=15pt@C=40pt{
\RHom_{B}(M,N)  \otimes \RHom_{B}(\sfS_{B}^{-1}(N), M) \ \ \ \ \ 
\ar[rd]^-{ \invserreagl{-,+} } \ar[d]_{ \ F_{\RHom} \otimes F_{\RHom} \ }& \\ 
\RHom_{A}(F(M), F(N))  \otimes \RHom_{A}(F\sfS_{B}^{-1}(N), F(M)) \ \ \ \ \
 \ar[d]_{ \ \id \otimes \RHom(\gamma_{X, N} , F(M))  \ }  
& \kk \\
\RHom_{A}(F(M), F(N) )  \otimes \RHom_{A}(\sfS^{-1}_{A}F (N), F(M))   \ \ \ \ \ 
\ar[ru]_-{\invserreagl{-,+}} & 
 }\end{xymatrix}
 \]
 \end{corollary}

Roughly speaking, this corollary shows that $\invserreagl{f, g} = \invserreagl{F(f), F(g) \gamma_{X, N} }$ 
for $f \in \RHom_{B}(M, N) $ and $g \in \RHom_{B}(\sfS_{B}(N), M)$.

\subsubsection{Mukai pairing}
In this section, we define our Mukai pairing by specializing the isomorphism of Proposition \ref{202508171520}, and we discuss its basic properties.
\begin{proposition}[{Mukai pairing}]\label{202508171603}
Let $A$ be a smooth and proper dg-algebra. 
Then we have an isomorphism 
\[
\tuD\HHH_{\bullet}(A) \cong \HHH_{\bullet}(A). 
\]
\end{proposition}

\begin{proof}
By Corollary \ref{202508181504}, 
we have $A \cong A^{\vvee}\lotimes_{A} \tuD(A)$ in $\tuD(A^{\mre})$. 
Thus we have an isomorphism $\HHH^{\bullet}(A; \tuD(A)) \cong \HHH_{\bullet}(A)$
as follows:
\begin{equation}\label{202508171817}
\HHH^{\bullet}(A; \tuD(A)) \cong  \RHom_{A^{\mre}}(A^{\vvee}\lotimes_{A} \tuD(A), \tuD(A))
 \cong \RHom_{A^{\mre}}(A^{\vvee}, A )\cong \HHH_{\bullet}(A). 
\end{equation}
Applying Proposition \ref{202508171520} to the case $X =A$, we complete the proof.
\end{proof}

\begin{definition}[{Mukai pairing}]\label{202602050936}
We define the \emph{Mukai pairing} $\mukaiagl{-,+}$ to be the non-degenerate pairing 
associated with the isomorphism of Proposition \ref{202508171603}: 
\[
\mukaiagl{-,+}: \HHH_{\bullet}(A) \otimes \HHH_{\bullet}(A) \to \kk,
\]
\end{definition}

We provide another description of Mukai pairing, 
which follows immediately from  Lemma \ref{202509010738} combined with the description \eqref{202509010636} of the pairing. 

\begin{lemma}\label{202509010639}
Let $A$ be a smooth and proper dg-algebra. 
Then the Mukai pairing $\mukaiagl{-,+}$ coincides with the following composition: 
\[
\begin{split}
\HHH_{\bullet}(A) 
\otimes \HHH_{\bullet}(A) 
\ \ \ \ \cong \ \ \ \  
&\RHom_{A^{\mre}}(A^{\vvee}, A) \otimes  \RHom_{A^{\mre}}(A^{\vvee}, A) \\
\xrightarrow{ \  \id \otimes (\tuD(A)\lotimes -)_{\RHom} \ } 
&\RHom_{A^{\mre}}(A^{\vvee}, A) \otimes\RHom_{A^{\mre}}(A, \tuD(A)) \\
\xrightarrow{ \ \ \comp \ \  } &\RHom_{A^{\mre}}(A^{\vvee}, \tuD(A)) \cong \tuD(A)\lotimes_{A^{\mre}} A \\
\xrightarrow{ \ \dcounit_{A} \ } & \ \ \kk
\end{split}
\]
where the morphism $\comp$ is the morphism 
which composes morphisms $f: A^{\vvee} \to A$ and $g: A \to \tuD(A)$. 
\end{lemma}

The following corollary seems to be well-known for specialists. 

\begin{corollary}\label{202508180111}
Let $A$ be a smooth and proper dg-algebra concentrated in non-positive degree.  
Then the Hochschild homology $\tuHH_{*}(A)$ is  concentrated in the zero-th degree.  
\end{corollary}

\begin{proof}
Since $A$ is concentrated in non-positive degree,  we have  $\tuHH_{i}(A) =\tuH^{-i}(A\lotimes_{A^\mre}A) = 0$ for $i < 0$. 
Since $\tuHH_{i}(A) \cong \tuD(\tuHH_{-i}(A))$ by Proposition \ref{202508171603}, 
we have $\tuHH_{i}(A) = 0$ for $i >0$.  
\end{proof}

\begin{remark}
\begin{enumerate}[(1)]
\item 
In the case $A=\kk Q/I$ is a finite dimensional algebra given by a quiver $Q$ with an admissible ideal $I$. 
Then $A$ is elementary and hence it is smooth if and only if it is of finite global dimension. 
 In this case, Corollary \ref{202508180111} was first proved by Cibilis \cite{Cibilis} under the additional assumption that $Q$ is directed. 
For the general case, see \cite[1.7]{Happel: Hochschild}, \cite[Proposition 2.5]{Keller:invariance}.

\item The same result is not necessarily true for a finite dimensional algebra $A$ of finite global dimension. 
Let $p$ be a prime number, $\kk:= \FF_{p}(t)$ a rational function field in one variable  over a  prime field of characteristic $p$ and $A = \kk(\sqrt[p]{t})$. 
Then we have $\tuHH_{i}(A) \cong A $ for $i \geq 0$.  
\end{enumerate}
\end{remark}

\subsection{Boundary-bulk map, bulk-boundary map, Chern character and the adjointness}\label{202602012158}

From now on we assume that our dg-algebra $A$ is smooth and proper.

\subsubsection{Boundary-bulk map}

Let $M \in \perf A$. 
The \emph{boundary-bulk} map $\tau_{M}$ is defined as the composition 
\[
\tau_{M}: \RHom_{A}(M,M) \cong \HHH_{\bullet}(A; M\otimes M^{\lvvee}) \xrightarrow{ (\lcounit_{M})_{*} } \HHH_{\bullet}(A)
\]
where the first isomorphism is \eqref{202508171811} for the case $N=M$.

\begin{example}\label{202508261336}
We deal with the case $A = \kk$. 
Identify $\perf \kk^{\mre}$ with $\perf \kk$ via the canonical isomorphism $\kk^{\mre}= \kk \otimes \kk \cong \kk$, 
then the functor $\HHH_{\bullet}(\kk; -)$ coincides with the identity functor of $\perf \kk$. 
For  $V \in \perf \kk$, 
the boundary-bulk map $\tau_{V}$ coincides with $\lcounit_{V}: \RHom_{\kk}(V, V) \to \kk$. 
It is straightforward to check that the $0$-th cohomology morphism $\tuH^{0}(\lcounit_{V}) : \Hom_{\kk}(V,V) \to \kk$ 
coincides with  the trace map $\Tr: \Hom_{\kk}(V,V) \to \kk$ introduced in Section \ref{202509171618}. 
\end{example}

We leave the proof of following lemma to the readers. 

\begin{lemma}\label{202508261812}
Let $M \in \perf A$. 
We have the following commutative diagram:  
\[
\begin{xymatrix}@R=20pt{
\RHom_{A}(M,M) \ar[r]^-{\tau_{M}} \ar[d]_{(-)^{\lvvee}_{\RHom} }& \HHH_{\bullet}(A) \ar[d]^{\cong} \\ 
\RHom_{A^{\op}}(M^{\lvvee},M^{\lvvee}) \ar[r]_-{\tau_{M^{\lvvee} } }  & \HHH_{\bullet}(A^{\op})
}\end{xymatrix}
\]
\end{lemma}

The following proposition  shows that the boundary-bulk map satisfies is additive  with respect 
to a certain class of endomorphism of an exact triangle. 

\begin{proposition}\label{202508301431}
Let $L\xrightarrow{ \alpha } M \xrightarrow{ \beta } N  \xrightarrow{ h} L[1]$ be an exact triangle in $\perf A$ 
and $f: L \to L, \ g: M \to M$ morphisms in $\perf A$ such that 
$\alpha f = g \alpha$. 
Then there exists a morphism $h: N \to N$ in $\perf A$ such that 
$\tau_{L}(f) - \tau_{M}(g) + \tau_{N}(h) =0$ 
and that completes the following commutative diagram: 
\[
\begin{xymatrix}@R=15pt{
L\ar[r]^{\alpha} \ar[d]_{f}  & M \ar[r]^{\beta} \ar[d]_{g}  & N \ar[r]^{\gamma} \ar[d]^{h} & L[1] \ar[d]^{f[1]} \\ 
L\ar[r]_{\alpha} & M \ar[r]_{\beta} & N \ar[r]_{\gamma} & L[1] 
}\end{xymatrix}
\]
\end{proposition}

\begin{proof}
We denote cofibrant dg-$A$-modules that represent $L ,M$, by the same symbol $L, M$. 
We also denote  cochain maps of dg-$A$-modules that represent $\alpha, \beta, f, g$, by the same symbol. 
Then $g\alpha$ and $\alpha f$ are homotopic to each other. 
Let $H: L \to M$ be a morphism of  graded-$A$-modules of degree $-1$ such that 
$g\alpha - \alpha f= d_{M}H +Hd_{L}$. 
 We recall an explicit construction of a cone dg-$A$-module $\cone \alpha$ of $\alpha$. 
 The underlying graded modules of $\cone\alpha$ is $M \oplus L[1]$ 
 and the differential $d_{\cone \alpha}$ is given by a matrix 
$
 d_{\cone \alpha} := 
 \begin{pmatrix} d_{M} & \alpha' \\ 0 & d_{L[1]} \end{pmatrix}
$ 
 where $\alpha': L[1] \to M$ is the morphism of dg-$A$-modules of degree $1$ induced from $\alpha$. 
 Then,
 the cone $\cone \alpha$ is a cofibrant  dg-$A$-module that represents $N$, and so it is denoted by $N$. 
 We define a morphism $h: N \to N$ of dg-$A$-modules to be 
 $h:= \begin{pmatrix} g & H'    \\ 0 & f[1]\end{pmatrix}$ where $H' : L[1] \to M$ be 
 the morphism of graded $A$-modules of degree $0$ induced from $H$. 
 Then, it is well-known that $h$ completes the desired diagram in $\perf A$. 

We claim that $\tau_{N}(h)= \tau_{L[1]}(f[1])+\tau_{M}(g)$. 
Indeed, 
for a projectively cofibrant  dg-$A$-module $K$, 
the counit morphism $\lcounit_{K}: K \otimes K^{\lvvee} \to A$ is 
represented by the morphism $\tilde{\epsilon}_{K}: K \otimes \cpxHom_{A}(K,A) \to A, \ l \otimes \phi \mapsto (-1)^{|l||\phi|}\phi(l)$. 
Observe that this morphism $\tilde{\epsilon}_{K}$ is only relevant to the graded $A$-module structure of $K$ and is irrelevant to the differential $d_{K}$ of $K$. 
Since as a graded $A$-modules $N =\cone \alpha$ coincides with the direct sum $M \oplus L[1]$, 
the morphism $\tilde{\epsilon}_{N}$ is decomposed as a canonical projection 
$N \otimes \cpxHom_{A}(N, A) \to (M \otimes \cpxHom_{A}(M,A)) \oplus (L[1] \otimes\cpxHom_{A}(L[1], A))$ 
followed by $( \tilde{\epsilon}_{M} , \tilde{\epsilon}_{L[1]})$. 
Hence the claim follows.

We apply the consideration until now to the case $M=L, \alpha =\id_{L},  g=f$. 
Then $N= \cone \id_{L} = 0$ in $\perf A$. Hence $\tau_{N} (h) = 0$. 
It follows from the clam that $\tau_{L[1]}(f[1]) = - \tau_{L}(f)$. 
Combining this equation and the claim, we obtain the desired equality. 
\end{proof}

\begin{remark}
\begin{enumerate}[(1)]
\item 
In the above proof, we make  use of  dg-$A$-module structures. 
We do not know that the above proposition can be verified by  formal $2$-categorical methods.
 
 \item 
 The above proposition shows that additivity of the boundary-bulk map holds for  an endomorphism  of  exact triangle 
 that arising from dg-category of dg-modules.  
We do not know that additivity holds for arbitrary endomorphisms of exact triangles.  
(cf. \cite[Remark 1.10]{May})
 \end{enumerate}
\end{remark}

\subsubsection{Chern character}

Following \cite{CW,PV, Shklyarov} we introduce the Chern character of an object $M \in\perf A$. 

\begin{definition}\label{202508301618}
For $M \in \perf A$, 
we define an element $\ch_{M}$ of $\tuHH_{0}(A)$ which is called the \emph{Chern character of $M$}  to be 
\[
\ch_{M} :=\tau_{M}(\id_{M}).  
\]
\end{definition}

For $M \in \perf A$, 
we denote by $[M]$ the corresponding  element of the Grothendiec group $K_{0}(A)$ of $A$.

\begin{corollary}[{\cite[Proposition 13]{CW}}]\label{202508301624}
Let $L\xrightarrow{ \alpha } M \xrightarrow{ \beta } N  \xrightarrow{ h} L[1]$ be an exact triangle in $\perf A$. 
Then we have 
$\ch_{L} -\ch_{M} +\ch_{N} = 0$. 

Consequently, the Chern character induces a group homomorphism 
\[
\ch: K_{0}(A) \to \tuHH_{0}(A), \ [M] \mapsto \ch_{M}.
\]
\end{corollary} 

\begin{proof}
Consider the proof of Proposition \ref{202508301431} 
in the case $f= \id_{L}, g=\id_{M}$. 
Then,  
we may take the homotopy $H$ to be zero 
and hence $h=\id_{N}$. 
\end{proof}

\subsubsection{Bulk-boundary map}

The evaluation map $\epsilon_{M}$  induces a morphism 
\[
\begin{split}
\sigma_{M}: \HHH_{\bullet}(A) \cong \HHH^{\bullet}(A; \tuD(A)) 
\xrightarrow{(\lcounit_{M})^{*}} &\HHH^{\bullet}(A; \tuD(M \otimes M^{\lvvee})) \\ 
\cong &\RHom_{A}(M, \tuD(A) \lotimes_{A}M) \\
\cong &\RHom_{A}(A^{\vvee} \lotimes_{A} M, M)
\end{split} 
\] 
in $\sfD(A^{\mre})$, 
where the first isomorphism is \eqref{202508171817} and the third is 
\eqref{202508171812} for the case $N=M$.
The morphism $\sigma_{M}$ is called \emph{bulk-boundary map} (cf. \cite[Section 3.2.3]{Carqueville}). 

The following lemma identifies the bulk-boundary map.

\begin{lemma}\label{202508171844}
Under the canonical isomorphism $\HHH_{\bullet}(A) \cong \RHom_{A^{\mre}}(A^\vvee, A)$, 
the the bulk-boundary map $\sigma_{M}$ corresponds to the morphism 
$\RHom_{A^{\mre}}(A^{\vvee}, A) \to \RHom_{A}(A^{\vvee}\lotimes_{A} M, M)$ 
induced from the functor $-\lotimes_{A} M : \perf A^{\mre} \to \perf A$. 
\[
\begin{xymatrix}@C=40pt{
\HHH_{\bullet}(A) \ar[r]^-{\sigma_{M}} \ar[d]_{\cong} & \RHom_{A}(A^{\vvee} \lotimes_{A} M, M) \ar[d]^{\cong}\\
 \RHom_{A^{\mre}}(A^\vvee, A) \ar[r]_-{(-\lotimes_{A} M)_{\RHom}}  & \RHom_{A}(A^{\vvee}\lotimes_{A} M , A\lotimes_{A}M) \\
 }\end{xymatrix}
 \]
\end{lemma}

\begin{proof}
Consider the following diagram  
whose upper horizontal line is a part of \eqref{202508171817}
\[
\begin{xymatrix}@C=18pt{
\HHH^{\bullet}(A; \tuD(A)) \ar@{=}[r]  \ar[d]_{ (\epsilon_{M})_{*} } & 
\RHom_{A^{\mre}}(A, \tuD(A)) \ar[d]_{(-\lotimes M)_{\RHom} } \ar[r]^{(A^{\vvee} \lotimes - )_{\RHom} }& 
\RHom_{A^{\mre}}(A^{\vvee}, A) \ar[d]^{(-\lotimes M)_{\RHom} }\\
\HHH^{\bullet}(A; \tuD(M \otimes M^{\lvvee}))  \ar[r]_{\cong \eqref{202508171812} } &
 \RHom_{A}(M, \tuD(A) \lotimes_{A} M) \ar[r]_-{ \ \ (A^{\vvee} \lotimes -)_{\RHom} } &
 \RHom_{A}(A^{\vvee} \lotimes_{A} M, M)
}\end{xymatrix}
\]
It  is enough to show that this diagram is commutative.

The commutativity of the right square is clear. 
The commutativity of the left square deduced from the following commutative diagram 
\[
\begin{xymatrix}{
\tuD(A) \ar[d]_{\tuD(\lcounit_{M})} \ar[r]^{\cong} & 
\tuD(A) \lotimes_{A} A^{\rvvee} \ar[d]_{\tuD(A) \lotimes (\lcounit_{M})^{\rvvee}} \ar[r]^{\cong} & 
\tuD(A) \lotimes_{A} A \ar[d]^{ \tuD(A) \lotimes \runit_{M} } \\
\tuD(M \otimes M^{\lvvee}) \ar[r]_-{\cong} & 
\tuD(A) \lotimes_{A} (M \otimes M^{\lvvee})^{\rvvee} \ar[r]_-{\cong} &
\tuD(A) \lotimes_{A}M \otimes \tuD(M)
}\end{xymatrix}
\]
in which we regard $M$ as an object of $\perf(A \otimes \kk^{\op})$. 
\end{proof}

\subsubsection{The adjointness  of the boundary-bulk  and the bulk–boundary maps}
As an immediate consequence of the reciprocity law for the canonical pairing \eqref{202508171549}, we obtain the following adjointness property of the boundary–bulk and bulk–boundary maps.
On the level of cohomology, this was proved by direct computation in \cite[Proposition 11]{CW}.

\begin{theorem}[{cf. \cite[Proposition 11]{CW}}]\label{202508171830}
We have the following commutative diagram
\begin{equation}\begin{xymatrix}@C=60pt{
\RHom_{A}(M,M) \otimes \HHH_{\bullet}(A) \ar[r]^-{ \tau_{M}  \otimes \id } \ar[d]_-{\id \otimes \sigma_{M}} &
\HHH_{\bullet}(A) \otimes \HHH_{\bullet}(A) \ar[d]^{\mukaiagl{-,+}} \\
\RHom_{A}(M, M) \otimes \RHom_{A}(A^{\vvee} \lotimes_{A}M,M) \ar[r]_-{\invserreagl{-,+}} & \kk 
}\end{xymatrix}
\end{equation}
\end{theorem}

Roughly speaking, this theorem shows that 
\[
\invserreagl{ f, \sigma_{M}(\phi) } = \mukaiagl{ \tau_{M}(f), \phi}
\]
for $f \in \RHom_{A}(M,M)$ and $\phi \in \HHH_{\bullet}(A)$. 

\subsection{Other Mukai pairings  }

We recall the definitions of the other pairings in the previous literature 
and show that  they coincide to each other. 

\subsubsection{The Shklyarov's Mukai pairing}

Let $A,B, C$ be a smooth and proper dg-algebras.

For $X \in \perf (A\otimes B^{\op})$,  
following 
\cite{CW, Petit, PV, Shklyarov}, we define a morphism $\Phi_{X} : \HHH_{\bullet}(B) \to \HHH_{\bullet}(A)$ as 
the composition 
\[
\Phi_{X} : \HHH_{\bullet}(B) = B_{\natural} \xrightarrow{ (\lunit_{X})_{\natural} } (X^{\lvvee} \lotimes_{A} X)_{\natural} \cong 
(X \lotimes_{B} X^{\lvvee})_{\natural} \xrightarrow{ (\lcounit_{X})_{\natural} }  A_{\natural} = \HHH_{\bullet}(A). 
\]
The string diagram of $\Phi_{X}$ is as follows: 

\vspace{4pt}

\begin{tikzpicture}
\node (s) at (0,0) {$B$}; 

\draw [very thick, rounded corners=10pt] (s) -- (-0.6,0.3) -- (0,0.6) --(0.6,0.3)  --  (s);

\draw[->] (1, 0) -- node [auto] {$(\lunit_{X})_{\natural}$}(2,0);

\begin{scope}[shift={(3,0)}]

\node (s) at (0,0) {$X^{\lvvee}$}; 
\node (t) at (1.2,0) {$X$};

\draw[very thick] (s) -- (t);
\draw [very thick, rounded corners=10pt] (s) -- (-0.7,0.3) -- (0,0.6) --(1.2,0.6) --(1.9,0.3) --  (t);
\end{scope}

\node at (5,0){$\cong$};

\begin{scope}[shift={(6,0)}]

\node (s) at (0,0) {$X$}; 
\node (t) at (1.2,0) {$X^{\lvvee}$};

\draw[very thick] (s) -- (t);
\draw [very thick, rounded corners=10pt] (s) -- (-0.7,0.3) -- (0,0.6) --(1.2,0.6) --(1.9,0.3) --  (t);

\draw[->] (1.8, 0) -- node [auto] {$(\lcounit_{X})_{\natural}$}(2.8,0);
\end{scope}

\begin{scope}[shift={(9.5,0)}]

\node (s) at (0,0) {$A$}; 

\draw [very thick, rounded corners=10pt] (s) -- (-0.6,0.3) -- (0,0.6) --(0.6,0.3)  --  (s);
\end{scope}

\end{tikzpicture}

The diagonal dg-$A^{\op}$-bimodule  $A^{\op} \in \perf( (A^{\mre})^{\op}) \cong \perf(\kk^{\op} \otimes (A^{\mre})^{\op})$ induces  a pairing 
\[
\begin{split}
\bmukaiagl{-,+}: \HHH_{\bullet}(A) \otimes \HHH_{\bullet}(A) \cong 
\HHH_{\bullet}(A) \otimes \HHH_{\bullet}(A^{\op}) \cong
\HHH_{\bullet}(A^{\mre}) \\  \xrightarrow{ \ \Phi_{A^{\op}} \ } \HHH_{\bullet}(\kk^{\op}) \cong \HHH_{\bullet}(\kk) \cong \kk, 
\end{split}\]
which is the pairing introduced by Shklyarov \cite{Shklyarov}. 
%
%

The following Lemma is easily follows from the isomorphism \eqref{202508261311}. 
So its proof is left to the readers.

\begin{lemma}\label{202508261357}
\begin{enumerate}[(1)]
\item 
Let $M \in \perf B$ and $X\in \perf (A \otimes B^{\op})$. 
Then we have the following commutative diagram 
\[
\begin{xymatrix}@R=20pt{ 
\RHom_{B}(M,M) \ar[r] \ar[d]_{\tau_{M}} & \RHom_{A}(X \lotimes_{B}M, X\lotimes_{B}M) \ar[d]^{\tau_{X\lotimes M}} \\
\HHH_{\bullet}(B) \ar[r]_{\Phi_{X}} & \HHH_{\bullet} (A) 
}\end{xymatrix}
\]
where the top horizontal arrow is induced from the functor $X\lotimes_{B}-$. 

In particular, we have $\Phi_{X}(\ch_{M}) =\ch_{X\lotimes_{B} M}$.

\item 
Let $X\in \perf(A\otimes B^{\op})$ and  $Y \in \perf (B \otimes C^{\op})$. 
Then we have $\Phi_{X \lotimes_{B} Y} = \Phi_{X} \Phi_{Y}$. 
\end{enumerate}
\end{lemma}

We note that in the notation of Section \ref{202509191836}, we have $\Phi_{X}=(\lcounit_{X})_{\natural}\tilde{\Phi}_{X, A}$. 
We may regard $\Phi_{X}$ as a morphism $\RHom_{B^{\mre}}(B^{\vvee}, B) \to \RHom_{A^{\mre}}(A^{\vvee}, A)$. 
It follows from  \eqref{202509201657} 
that 
for $f \in \Hom_{B^{\mre}}(B^{\vvee}, B )$ 
the image $\tuH^{0}(\Phi_{X})(f)$ is given by the following composition 
\begin{equation}\label{202509210026}
A^{\vvee} \xrightarrow{ \ \lcounit_{X}^{\vvee} \ } (X \lotimes_{B} X^{\lvvee} )^{\vvee} 
\cong {}^{X} B^{\vvee} \xrightarrow{ \ {}^{X}\!f \  \ } 
{}^{X} B =X\lotimes_{B} X^{\lvvee} \xrightarrow{ \ \lcounit_{X} \ } A
\end{equation}

Recall that the derived Picard group $\DPic_{\kk}(A)$ of $A$ is the group 
consists of quasi-isomorphism classes of invertible dg-$A$-$A$-bimodules,  
with  the multiplication is induced by the derived tensor product.
It follows from Lemma \ref{202508261357}(2) that 
the derived Picard group $\DPic_{\kk}(A)$ of $A$ acts on $\HHH_{\bullet}(A)$ 
and hence also on $\tuHH_{0}(A)$. 
For $T \in \DPic_{\kk}(A)$ and $f \in \tuHH_{0}(A)$, we denote the action of $T$ on $f$ 
by $T\cdot f :=\tuH^{0}(\Phi_{T})(f)$. 
As a corollary of Lemma \ref{202508261357}(1), we show that 
the Chern character map is equivariant with respect to the action of $\DPic_{\kk}(A)$. 
The precise statement is the following. 

\begin{corollary}\label{202509171827}
For $M \in \perf A$ and $T\in \DPic_{\kk}(A)$, we have $\ch_{T\lotimes_{A} M} = T\cdot \ch_{M}$. 
\end{corollary}

\subsubsection{The Mukai pairing in the sense of Căldăraru--Willerton}

Next we recall the definition of 
Mukai pairing $\cwmukaiagl{+,-}$ by Căldăraru--Willerton \cite[p. 81]{CW} in our setting. 

We define a morphism $\cwmukaiagl{+,-}: \HHH_{\bullet}(A) \otimes \HHH_{\bullet}(A) \to \kk$ as the following composition 
\[
\begin{split}
\HHH_{\bullet}(A) \otimes \HHH_{\bullet}(A) 
& \xrightarrow{ \ \flip \ } \HHH_{\bullet}(A) \otimes \HHH_{\bullet}(A) 
\cong \RHom_{A^{\mre}}(A^{\vvee}, A) \otimes \RHom_{A^{\mre}}(A^{\vvee}, A) \\
& \xrightarrow{ \ l\otimes r \ } 
\RHom_{A^{\mre}}(A, \tuD(A)) \otimes \RHom_{A^{\mre}}(A, \tuD(A)) \\
& \xrightarrow{ \ \ (-\lotimes_{A} -)_{\RHom} \ \ }  \RHom_{A^{\mre}}(A, \tuD(A)\lotimes_{A} \tuD(A)) \\
&  \cong \RHom_{A^{\mre}}(A, \tuD(A^{\mre})\lotimes_{A^{\mre}} A)\xrightarrow{  \ T_{A} \ } \kk
\end{split}
\]
where 
the first arrow flips the two factors of the tensor product,
we set the morphisms $l, r: \RHom_{A^{\mre}}(A^{\vvee}, A) \to \RHom_{A^{\mre}}(A, \tuD(A))$ to be 
$l:=(\tuD(A)\lotimes -)_{\RHom}, \ r:=(-\lotimes \tuD(A))_{\RHom}$   
and the  arrow $T_{A}$ is the morphism of Lemma \ref{202508311741} for the dg-$A^{\mre}$-modules $A$. 
Roughly speaking, the pairing $\cwmukaiagl{-,+}$ is given by 
\[
\cwmukaiagl{f,g} := (-1)^{|f||g|}\serreagl{\id_{A}, (\tuD(A)\lotimes g) \lotimes (f \lotimes \tuD(A))}
\]
for $f,g \in \HHH_{\bullet}(A)\cong \RHom_{A^{\mre}}(A^{\vvee}, A)$.

\subsubsection{Coincidence of the three Mukai pairings}

\begin{theorem}\label{202508231520}
The three Mukai parings coincide with each other:
\[
\mukaiagl{-,+} = \bmukaiagl{-,+} =\cwmukaiagl{-,+}.
\]
\end{theorem}

\begin{remark}
In \cite[Theorem 1]{Ramadoss}, 
it is shown that $\bmukaiagl{-,+} =\cwmukaiagl{-,+}$ at the level of cohomology 
in the case where $A$ is derived equivalent to a variety $X$. 
\end{remark}

\begin{proof}
First we prove $\mukaiagl{-,+} = \bmukaiagl{-,+}$. 
We have the following commutative diagram where $f, g$ are  the isomorphisms of Corollary \ref{202508181504} 
and we use the notations of Lemma \ref{202509151654} in the case of $B=A$ and $X =A$. 

\vspace{8pt}

\begin{tikzpicture}
\node (s) at (0,0) {$A$}; 
\node (t) at (1.2,0) {$A$};
\node (r) at (2.4,0) { $A$};

\node (u) at (0,-0.8) {$A$}; 
\node (v) at (1.2,-0.8) {$A$}; 
\node (w) at (2.4,-0.8) {$A$};

\draw[very thick] (s) -- (t);
\draw[very thick] (t) -- (r);
\draw[very thick] (u) -- (v);
\draw[very thick] (v) -- (w);

\draw [very thick, rounded corners=10pt] (s) -- (-0.7,0.3) -- (0,0.6) --(2.4,0.6) --(3.1,0.3) --  (r);
\draw [very thick, rounded corners=10pt] (u) -- (-0.7,-1.1) -- (0,-1.4) --(2.4,-1.4) --(3.1,-1.1) --  (w);

\draw[->] (3.3,-0.4) --node [auto]{$(\dunit_{A})^{\op}$} (4.3,-0.4);

\draw[->] (1.2,-1.6) --node [auto]{$\cong  f^{-1}$} (1.2,-2.4);

\begin{scope}[shift={(5,0)}]
%
%
%
%
%

%
%
%
%
%
%

\node (s) at (0,0) {$\tuD(A)$}; 
\node (t) at (1.2,0) {$A$};
\node (r) at (2.4,0) { $A$};

\node (u) at (0,-0.8) {$A$}; 
\node (v) at (1.2,-0.8) {$A$}; 
\node (w) at (2.4,-0.8) {$A$};

\draw[very thick] (t) -- (r);
\draw[very thick] (v) -- (w);

\draw[very thick] (u)  --  (t);
\draw[preaction={draw=white,line width=2.5pt}, very thick](s) --(v);
\draw [very thick, rounded corners=10pt] (s) -- (-0.7,0.3) -- (0,0.6) --(2.4,0.6) --(3.1,0.3) --  (r);
\draw [very thick, rounded corners=10pt] (u) -- (-0.7,-1.1) -- (0,-1.4) --(2.4,-1.4) --(3.1,-1.1) --  (w);

\draw[->] (3.3,-0.4) --node [auto]{$(\dcounit_{A})^{\op}$} (4.3,-0.4);

\draw[->] (1.2,-1.6) --node [auto]{$\cong f^{-1}$} (1.2,-2.4);

\node at (5, -0.4){$\kk$};
\end{scope}

\begin{scope}[shift={(0,-3.3)}]
\node (s) at (0,0) {$A$}; 
\node (t) at (1.2,0) {$A$};
\node (r) at (2.4,0) { $A$};

\node (u) at (0,-0.8) {$A$}; 
\node (v) at (1.2,-0.8) {$A^{\vvee}$}; 
\node (w) at (2.4,-0.8) {$\tuD(A)$};

\draw[very thick] (s) -- (t);
\draw[very thick] (t) -- (r);
\draw[very thick] (u) -- (v);
\draw[very thick] (v) -- (w);

\draw [very thick, rounded corners=10pt] (s) -- (-0.7,0.3) -- (0,0.6) --(2.4,0.6) --(3.1,0.3) --  (r);
\draw [very thick, rounded corners=10pt] (u) -- (-0.7,-1.1) -- (0,-1.4) --(2.4,-1.4) --(3.1,-1.1) --  (w);

\draw[->] (3.3,-0.4) --node [auto]{$(\dunit_{A})^{\op}$} (4.3,-0.4);

\draw[->] (1.2,-1.6) --node [auto]{$\vcounit_{A}$} (1.2,-2.4);

\begin{scope}[shift={(5,0)}]

\node (s) at (0,0) {$\tuD(A)$}; 
\node (t) at (1.2,0) {$A$};
\node (r) at (2.4,0) { $A$};

\node (u) at (0,-0.8) {$A$}; 
\node (v) at (1.2,-0.8) {$A^{\vvee}$}; 
\node (w) at (2.4,-0.8) {$\tuD(A)$};

\draw[very thick] (t) -- (r);
\draw[very thick] (v) -- (w);

\draw[very thick] (u)  --  (t);
\draw[preaction={draw=white,line width=2.5pt}, very thick](s) --(v);
\draw [very thick, rounded corners=10pt] (s) -- (-0.7,0.3) -- (0,0.6) --(2.4,0.6) --(3.1,0.3) --  (r);
\draw [very thick, rounded corners=10pt] (u) -- (-0.7,-1.1) -- (0,-1.4) --(2.4,-1.4) --(3.1,-1.1) --  (w);

%
%
%
%
%
%
%
%
%
%

%
%

\draw[->] (1.2,-1.6) --node [auto]{$\vcounit_{A}$} (1.2,-2.4);
\end{scope}

\end{scope}

\begin{scope}[shift={(0,-6.6)}]
%
%
%
%
%
%

%
%
%
%


\node (s) at (0,0) {$A$}; 
\node (t) at (1.2,0) {$A$};
\node (r) at (2.4,0) { $A$};

\node (u) at (0,-0.8) {$A$}; 
\node (v) at (1.2,-0.8) {$A$}; 
\node (w) at (2.4,-0.8) {$\tuD(A)$};

\draw[very thick] (s) -- (t);
\draw[very thick] (u) -- (v);

\draw[very thick] (t)  --  (w);
\draw[preaction={draw=white,line width=3.5pt}, very thick](v) --(r);
\draw [very thick, rounded corners=10pt] (s) -- (-0.7,0.3) -- (0,0.6) --(2.4,0.6) --(3.1,0.3) --  (r);
\draw [very thick, rounded corners=10pt] (u) -- (-0.7,-1.1) -- (0,-1.4) --(2.4,-1.4) --(3.1,-1.1) --  (w);

\draw[->] (3.3,-0.4) --node [auto]{$(\dunit_{A})^{\op}$} (4.3,-0.4);

\draw[->] (1.2,-1.6) --node [auto]{$\dcounit_{A}$} (1.2,-2.4);

\draw[->] (3.3,-1.6) --node [auto]{$\id$} (4.3,-2.4);

\begin{scope}[shift={(5,0)}]
\node (s) at (0,0) {$\tuD(A)$}; 
\node (t) at (1.2,0) {$A$};
\node (r) at (2.4,0) { $A$};

\node (u) at (0,-0.8) {$A$}; 
\node (v) at (1.2,-0.8) {$A$}; 
\node (w) at (2.4,-0.8) {$\tuD(A)$};


\draw[very thick] (t)  --  (w);
\draw[preaction={draw=white,line width=3.5pt}, very thick](v) --(r);

\draw[very thick] (u)  --  (t);
\draw[preaction={draw=white,line width=3pt}, very thick](s) --(v);
\draw [very thick, rounded corners=10pt] (s) -- (-0.7,0.3) -- (0,0.6) --(2.4,0.6) --(3.1,0.3) --  (r);
\draw [very thick, rounded corners=10pt] (u) -- (-0.7,-1.1) -- (0,-1.4) --(2.4,-1.4) --(3.1,-1.1) --  (w);

%
%

\draw[->] (1.2,-1.6) --node [auto]{$\dcounit_{A}$} (1.2,-2.4);
\end{scope}

\end{scope}

\begin{scope}[shift={(5,-9.6)}]

\node (t) at (1.2,0) {$A$};

\node (u) at (0,-0.8) {$A$}; 
\node (w) at (2.4,-0.8) {$\tuD(A)$};


\draw[very thick] (t)  --  (w);
\%draw[preaction={draw=white,line width=3.5pt}, very thick](v) --(r);

\draw[very thick] (u)  --  (t);
\draw [very thick, rounded corners=10pt] (u) -- (-0.7,-1.1) -- (0,-1.4) --(2.4,-1.4) --(3.1,-1.1) --  (w);

%
%
%
%
%
%
%

%
%
\draw[->] (-1,-0) --node [auto]{$\dcounit_{A}$} (-2,-0);

\end{scope}

\node at (1.2,-9.9) {$\kk$};

\draw[rounded corners=20pt, ->] (8,-3.7) -- (8.9, -4.2) -- node[auto]{$g$} (8.9, -9.3) --(8,-9.8);
\end{tikzpicture}

\vspace{10pt} 

We note that   $(\dcounit_{A})^{\op} =\dcounit_{A}$ under the canonical identification 
$\perf \kk^{\op} \cong \perf \kk$.  
Hence $\dcounit_{A} (\dunit_{A})^{\op} = (\dcounit_{A})^{\op} (\dunit_{A})^{\op} =\id$ by the triangle identity. 
The top horizontal line  gives $\bmukaiagl{-,+}$. 
The left vertical line gives $\mukaiagl{-,+}$. 
It follows from  the construction of the isomorphisms of Corollary \ref{202508181504}  
(see the proof of Lemma \ref{202508201805}),  
the right vertical line is identity morphism.  
Consequently, we obtain the desired equality $\mukaiagl{-,+} = \bmukaiagl{-,+}$.

\vspace{8pt}

Next, we prove $\mukaiagl{-,+} = \cwmukaiagl{-,+}$. 
For this purpose,  we consider the following  diagram: 
\begin{equation}\label{202511222219}
\begin{xymatrix}{
(A^{\vvee}, A) \otimes (A^{\vvee}, A) \ar[rr]^{\flip} \ar[d]_{\id \otimes l } &&
(A^{\vvee}, A) \otimes (A^{\vvee}, A) \ar[d]_{l\otimes r} \\
  (A^{\vvee}, A) \otimes (A, \tuD(A)) \ar[r]^{\flip} \ar[dr]_{\comp} & 
  (A, \tuD(A)) \otimes (A^{\vvee}, A) \ar[r]^{\id \otimes r} \ar[d]^{(-\lotimes-)_{\RHom}}&
  (A, \tuD(A)) \otimes (A, \tuD(A)) \ar[d]^{(-\lotimes-)_{\RHom}} \\
 &   (A^{\vvee}, \tuD(A)) \ar[d]_{\cong} \ar[r]^-{r} &
 (A, \tuD(A)\lotimes_{A} \tuD(A)) \ar[d]^{T_{A}} \\
 & (A\lotimes_{A} \tuD(A))_{\natural} \ar[r]_-{\dcounit_{A}} &
 \kk 
}\end{xymatrix}
\end{equation}
where we use the abbreviation $(-,+) =\RHom_{A^{\mre}}(-,+)$ 
and for simplicity we set $l: = (\tuD(A) \lotimes-)_{\RHom}, \ r:= (-\lotimes \tuD(A))_{\RHom}$. 

We claim that this diagram is commutative. 
Commutativity of the top and the middle rectangles is clear. 
To check the commutativity of the middle triangle, we take a projectively cofibrant dg-$A^{\mre}$-module $P$ that represents $A^{\vvee}$ 
and an injectively fibrant dg-$A^{\mre}$-module $I$ that represents $\tuD(A)$.  
Let $f$ and $g$ be homogeneous elements of $\cpxHom_{A^{\mre}}(P,A)$ and $\cpxHom_{A^{\mre}}(A, I)$ respectively. 
Then for homogeneous elements $a \in A, \ p\in P$, we have 
$(g\otimes f)(a \otimes p) = (-1)^{|a||f|}g(a)f(p)=(-1)^{|a||f|}g(af(p))=gf(ap)$. 
This equation immediately yields the desired commutativity. 
Finally,   we check the commutativity of the bottom rectangle. 
Under the canonical isomorphisms $\RHom_{A^{\mre}}(A^{\vvee}, \tuD(A)) \cong \tuD(A)_{\natural}$ and 
$\RHom_{A^{\mre}}(A, \tuD(A) \lotimes_{A} \tuD(A)) \cong (A^{\vvee} \lotimes_{A} \tuD(A) \lotimes_{A} \tuD(A))_{\natural}$, 
the morphism $r$ corresponds to 
the isomorphism $\tuD(A)_{\natural} \cong (\tuD(A) \lotimes_{A} A)_{\natural} \xrightarrow{ \cong } (\tuD(A) \lotimes_{A} \tuD(A) \lotimes_{A} A^{\vvee})_{\natural}$ 
induced from the canonical isomorphism $A \cong A^{\lvvee} \cong \tuD(A) \lotimes_{A} A^{\vvee}$. 
Now the desired commutativity follows from the commutativity of the following diagram

\begin{tikzpicture}
\begin{scope}[shift={(-4.1,-0.4)}]
\node (s) at (0,0) {$\tuD(A)$}; 
\node (t) at (1.2,0) {$A$};
\node (r) at (2.4,0) {};
%

\draw[very thick] (s) -- (t);

%

\draw [very thick, rounded corners=10pt] (s) -- (-0.7,0.3) -- (0,0.6) --(1.2,0.6) --(1.9,0.3) -- (t);

\draw[->] (2,0) --node [auto]{$r$} (3,0);
\draw[->] (0.4,-0.2) --node [below]{$\id$ \ \ } (3.4,-5.6);

\end{scope}

\node (s) at (0,0) {$\tuD(A)$}; 
\node (t) at (1.2,0) {$A$};
\node (r) at (2.4,0) {};

\node (u) at (0,-0.8) {\tuD(A)}; 
\node (v) at (1.2,-0.8) {$A^{\vvee}$}; 
\node (w) at (2.4,-0.8) {};

\draw[very thick] (s) -- (t);

\draw[very thick] (u) -- (v);

\draw [very thick, rounded corners=10pt] (s) -- (-0.7,0.3) -- (0,0.6) --(2.4,0.6) --(3.1,-0.1) --  (2.4,-0.8)--(v);
\draw [very thick, rounded corners=10pt,preaction={draw=white,line width=6pt}] (u) -- (-0.7,-1.1) -- (0,-1.4) --(2.4,-1.4) --(3.1,-0.7) --  (2.4, 0)--(t);

\draw[->] (3.3,-0.4) --node [auto]{$T_{A}$} (4.6,-6.2);

\draw[->] (1.2,-1.6) --node [auto]{$\vcounit_{A}$} (1.2,-2.4);

\begin{scope}[shift={(0,-3.2)}]
\node (s) at (0,0) {$\tuD(A)$}; 
\node (t) at (1.2,0) {$A$};
\node (r) at (2.4,0) {};

\node (u) at (0,-0.8) {\tuD(A)}; 
\node (v) at (1.2,-0.8) {$A$}; 
\node (w) at (2.4,-0.8) {};

\draw[very thick] (s) -- (t);

\draw[very thick] (u) -- (v);

\draw [very thick, rounded corners=10pt] (s) -- (-0.7,0.3) -- (0,0.6) --(2.4,0.6) --(3.1,-0.1) --  (2.4,-0.8)--(t);
\draw [very thick, rounded corners=10pt,preaction={draw=white,line width=6pt}] (u) -- (-0.7,-1.1) -- (0,-1.4) --(2.4,-1.4) --(3.1,-0.7) --  (2.4, 0)--(v);

\draw[->] (3.3,-0.4) --node [below]{$\dcounit_{A^{\mre}}$ \ \ \ } (4.5,-3);

\draw[->] (1.2,-1.6) --node [auto]{$\dcounit_{A}$} (1.2,-2.4);
\end{scope}

\begin{scope}[shift={(0,-6.4)}]
\node (s) at (0,0) {$\tuD(A)$}; 
\node (t) at (1.2,0) {$A$};
\node (r) at (2.4,0) {};

\node (u) at (0,-0.8) {}; 
\node (v) at (1.2,-0.8) {}; 
\node (w) at (2.4,-0.8) {};

\draw[very thick] (s) -- (t);


\draw [very thick, rounded corners=10pt] (s) -- (-0.7,0.3) -- (0,0.6) --(2.4,0.6) --(3.1,-0.1) --  (2.4,-0.8)--(t);

\draw[->] (3.3,-0) --node [below]{$\dcounit_{A}$} (4.3,-0);

\end{scope}

\node at (4.6,-6.4){$\kk$};
\end{tikzpicture}

\noindent 
where  the commutativity of the left triangle follows from  the triangle identity 
of the adjoint pair $(A, \tuD(A)\lotimes_{A} A^{\vvee})$ (see the proof of Lemma \ref{202508201805}).

Composing the morphisms in 
the clockwise direction  from the upper-left corner to the bottom-right corner in the diagram \eqref{202511222219}  
yields the Mukai pairing $\cwmukaiagl{-,+}$ by Căldăraru--Willerton. 
On the other hand, the counterclockwise composition yields our Mukai pairing $\mukaiagl{-,+}$ by Lemma \ref{202508261812}. 
Thus we conclude that these two pairings coincide. 
\end{proof}

From now on, we identify three Mukai pairings 
\[
\mukaiagl{-,+} = \bmukaiagl{-,+} = \cwmukaiagl{-,+}.
\]

\subsection{Generalized abstract Hirzebruch–Riemann–Roch theorem}

Now we give the generalized abstract Hirzebruch–Riemann–Roch theorem for dg-algebras. 

\begin{theorem}[{\cite[Theorem 1.3.1]{PV}, cf., \cite[Theorem 5.6]{Petit}}]\label{202508261418}
Let $A$ be a smooth and proper dg-algebra and $M, N \in \perf A$. 
Let $f $ and $g$ be homogeneous elements of $\tuH(\RHom_{A}(M,M))$ and $\tuH(\RHom_{A}(N,N))$ respectively. 
We set a morphism $\RHom(f,g): \RHom_{A}(N, M) \to \RHom_{A}(N, M )$ to be 
\[
\RHom(f,g)(\phi) := (-1)^{|g||\phi|}f\circ \phi \circ g.
\]
Then we have 
\[
\mukaiagl{\tau_{M}(f), \tau_{N}(g)} = \Tr \left(  \RHom(f,g) \right). 
\]
\end{theorem}

\begin{remark}
The reason for the difference in sign, compared to \cite[Theorem 1.3.1]{PV}, is that we are working with left dg-modules while \cite{PV} dealt with right dg-modules.
\end{remark}

We provide a proof of the theorem, following \cite{PV}, to bridge the gap in notation (and convention) between this paper and other papers.

\begin{proof}
We have the following  diagram 
\[
\begin{xymatrix}@R=20pt{
\RHom_{A}(M, M) \otimes \RHom_{A}(N, N) \ar[r]^-{ \ \tau_{M} \otimes \tau_{N} \ } \ar[d]_{\cong} &
\HHH_{\bullet}(A) \otimes \HHH_{\bullet}(A) \ar[d]^{\cong} \\ 
\RHom_{A}(M, M) \otimes \RHom_{A^{\op}}(N^{\lvvee}, N^{\lvvee}) \ar[r]^-{ \ \tau_{M} \otimes \tau_{N^{\lvvee} } \ } \ar[d]_{\cong} &
\HHH_{\bullet}(A) \otimes \HHH_{\bullet}(A^{\op}) \ar[d]^{\cong} \\ 
 \RHom_{A^{\mre}}(M\otimes N^{\lvvee}, M\otimes N^{\lvvee}) \ar[r]^-{ \ \tau_{M \otimes N^{\lvvee} } \ } \ar[d]_{(A \lotimes_{A^{\mre}}-)_{\RHom}} &
\HHH_{\bullet}(A^{\mre})  \ar[d]^{\Phi_{A}} \\ 
\RHom_{\kk}(\RHom_{A}(N, M), \RHom_{A}(N, M)) \ar[r]^-{ \ \tau_{\RHom_{A}(N, M) } \ }  &
\HHH_{\bullet}(\kk) 
 }\end{xymatrix}
 \]
 The commutativity of the first square follows from Lemma \ref{202508261812}. 
 It is straightforward to check that the second square is commutative. 
 The case that $B=A, X =A$ of   Lemma \ref{202508261357}(1) tells that the third square is commutative.

The right vertical arrow is the Mukai paring. 
The left vertical arrow sends a pair $(f,g)$ of homogeneous elements to $\RHom(f,g)$. 
Now by Example \ref{202508261336}, we obtain the desired equality. 
\end{proof}

Recall that the Grothendieck group $K_{0}(A)$ has a pairing 
\[
\ERagl{-,+}: K_{0}(A) \otimes_{\ZZ} K_{0}(A) \to \ZZ, \ \ERagl{[N], [M]} := \Euch(\RHom_{A}(N,M))
\]
which is called \emph{Euler-Ringel form}. 

As a corollary of Theorem \ref{202508261418}, 
we obtain the Hirzebruch–Riemann–Roch theorem for  dg-algebras. 
Although this formula plays no role in the rest of the paper, we include  it for  completeness.  

\begin{corollary}[{\cite[Theorem 1.3]{Shklyarov}, \cite[Theorem 14]{CW}}]\label{202508301737}
Let $A$ be a smooth and proper dg-algebra and $M, N \in \perf A$. 
Then we have 
\[
\mukaiagl{\ch_{M}, \ch_{N}} = \ERagl{[N], [M]}. 
\]
\end{corollary}
What plays a key role in the sequel is the following formula for computing the Serre pairing, which is obtained by combining Theorem \ref{202508171830} and Theorem \ref{202508261418}.

\begin{corollary}\label{202508301745}
Let $A$ be a smooth and proper dg-algebra and $M, N \in \perf A$ 
and  $f: M \to M $ a morphism in $\perf A$. 
Then we have 
\[
\invserreagl{f, \sigma_{M}(\ch_{N}) } = \Tr\left[f\circ- :\RHom_{A}(N,M) \to \RHom_{A}(N,M) \right]. 
\]
\end{corollary}

\section{Universal Auslander--Reiten triangle}\label{202602050956}

\subsection{Preliminary}\label{202509181437}

In this Section \ref{202509181437}, 
we prepare basic notions of a smooth and  proper  dg-algebra $A$ which is non-positive, i.e., $\tuH^{>0}(A) = 0$.
(We remark that  several arguments work even if we drop the smoothness assumption.)
 
For simplicity, we set denote the cohomology algebra of $A$ by $H := \tuH(A)$ and 
the $0$-th component algebra by $H^{0} := \tuH^{0}(A)$. 
Replacing $A$ by the smart truncation $\tau^{\leq 0 }A$, which is quasi-isomorphic to $A$,   
we may assume that there is a morphism $p: A \to H^{0}$ of dg-algebras. 
We regard $H^{0}$-modules as dg-$A$-modules by restriction of scalar  along $p$.  
This morphism induces an adjoint pair $p^{*}; \sfD(A) \rightleftarrows \sfD(H^{0}): p_{*}$.
The heart of the standard $t$-structure of $\sfD(A)$ may be identified with the modules category $H^{0}\Mod$, 
which is the heart of the standard $t$-structure of $\sfD(H^{0})$, by the restriction functor $p_{*}$.

Let $e_{1}, e_{2}, \dots, e_{r}$ be a set of primitive orthogonal idempotent elements of $H^{0}$
such that the corresponding indecomposable projective modules $H^{0}e_{1},$  $H^{0}e_{2},$ $ \dots,$  
$H^{0}e_{r}$ gives a basis of $K_{0}(H^{0})$. 
Let $S_{1}, S_{2}, \dots, S_{r}$ be simple $H^{0}$-modules corresponding to $e_{1}, e_{2}, \dots, e_{r}$. 
To simplify arguments, we assume that $\dim S_{i} = 1$ for all $i =1,2, \dots, r$.

Since $\End_{A}(A)^{\op} \cong H^{0}$, 
the idempotent elements $e_{1}, e_{2}, \dots, e_{r}$ induce idempotent endomorphisms of $A$ in $\perf A$. 
We denote the direct summands of $A$  in  $\perf A$ 
that corresponds to these idempotent endomorphisms by $P_{1}, P_{2}, \dots, P_{r}$. 
We note that if the idempotent elements $e_{1}, e_{2}, \dots, e_{r}$ of $H^{0}$ 
lift up to idempotent elements of  $A$ along the morphism $p: A \to H^{0}$
 and we denote them by the same symbol, then $P_{1} =A e_{1}, P_{2} = A e_{2}, \dots, P_{r} = Ae_{r}$.
 
For $M\in \perf A$,  
under the canonical isomorphism $\tuH(\RHom_{A}(A, M))  \cong \tuH( M)$, 
we have $\tuH(\RHom_{A}(P_{i}, M)) \cong e_{i}\tuH(M)$. 
For simplicity, we set 
\[
\Euch_{i}(M) := \ERagl{P_{i},M }=\Euch(\RHom_{A}(P_{i}, M))= \sum_{n \in \ZZ} (-1)^{n} \dim e_{i}\tuH^{n}(M).
\]

We collect basic properties in the next lemma. 

\begin{lemma}\label{202509181508}
The following statements hold. 
\begin{enumerate}[(1)]
\item 
$\Euch_{i}(S_{j}) = \ERagl{P_{i}, S_{j} } = \dim e_{i}S_{j} = \delta_{ij}$ for $i,j = 1,2, \dots, r$.

\item 
The subsets $\{ [P_{i}] \}_{i=1}^{r}$ and $\{ [S_{i}] \}_{i =1}^{r}$ are  basis of $K_{0}(A)$ 
which are dual to each other with respect to the Euler-Ringel form. 

\item 
In $K_{0}(A)$, we have the following equality 
\[
[M] = \sum_{i =1}^{r} \Euch_{i}(M) [S_{i}]. 
\]

\end{enumerate}
\end{lemma}
 
 \begin{proof}
 (1) is clear. 
 
 (2) 
 Since the sets $\{ P_{i}  \}_{i=1}^{r}$ and $\{ S_{i} \}_{i =1}^{r}$ of objects of $\perf A$ generate the triangulated category 
 $\perf A$,  the subset $\{ [P_{i}] \}_{i=1}^{r}$ and $\{ [S_{i}] \}_{i =1}^{r}$ are generators of $K_{0}(A)$. 
 It follows from (1) that these subsets are linearly independent. 
 
 (3) 
 Since $M$ is obtained as an extension of the shift of the cohomology groups $\tuH^{n}(M)[-n]$, 
 all but finitely many of which  are zero, 
we have the first equality in the following equations that yields the desired equality:
\[
\begin{split}
[M] 
  = \sum_{n \in \ZZ} (-1)^{n} [\tuH^{n}(M)] 
  = \sum_{n \in \ZZ}\sum_{i =1}^{r}  (-1)^{n} \dim(e_{i}\tuH^{n}(M)) [S_{i}] 
  = \sum_{i =1}^{r}  \Euch_{i}(M)[S_{i}]. \\ 
 \end{split}
 \] 
 \end{proof}

\subsubsection{The Cartan matrix and the Coxeter matrix}\label{202105081902}

From now we follow notation of Auslander--Reiten theory of a triangulated category \cite{Happel Book},  
and set  
\[\nu:= \sfS=\tuD(A)\lotimes_{A} -, \ \nu_{1} :=  \nu [-1] =\tuD(A) \lotimes_{A} (\Sigma^{-1} A) \lotimes_{A} -. 
\]
It is the Auslander--Reiten translation of the triangulated category $\perf A$.   
We may identify the quasi-inverse  functor $\nu^{-1}$ with $A^{\vvee} \lotimes_{A}-$.
To simplify the notation, we set $\Pi_{1} :=\Sigma A^{\vvee}= (\Sigma A) \lotimes_{A} A^{\vvee}$ so that  $\nu_{1}^{-1} =\Pi_{1} \lotimes_{A} -$.

We recall that  the Cartan matrix $C$  is a matrix defined by 
\[
C := (\Euch(e_{i} \tuH(A) e_{j}) )_{i,j} =(\Euch(\RHom_{A}(P_{i}, P_{j}))_{i,j}.
\]
It follows from Lemma \ref{202509181508} that 
the matrix $C$ is the transition matrix from the basis $\{ [S_{i}]\}_{i}$ to the basis $\{ [P_{i} ]\}_{i}$ 
and hence that it is a nonsingular matrix. 

Then the Coxeter matrix $\Phi$ (for left modules) is defined by 
\[
\Phi := - C^{t} C^{-1}. 
\]

 We identify $K_{0}(A)$ with $\ZZ^{\oplus r}$ by the ordered basis $\{ [S_{i}] \}_{i=1}^{r}$. (Note that later we will use a different identification by the ordered basis $\{ [P_{i}] \}_{i=1}^{r}$.)
Then  Lemma \ref{202509181508}(3) shows that  for $M \in \perf A$, the corresponding element $[M] \in K_{0}(A)$ is identified with  
the dimension vector $\Euvect(M)$: 
\[ 
\Euvect(M) := (\Euch_{1}(M), \Euch_{2}(M), \ldots, \Euch_{r}(M) )^{t}  \in \ZZ^{\oplus r}
\]
where $(-)^{t}$ denotes the transpose of column vectors.

Let  $F$ be an exact endofunctor   of $\perf A$. 
We denote by $\underline{F}$ the representation matrix of 
the induced map $[F]: K_{0}(A) \to K_{0}(A), [M] \mapsto [F(M)]$  with respect to the ordered basis $\{[S_{i}]\}_{i=1}^{r}$. 
In other words, $\underline{F}$ is the square matrix satisfying  
\[
\underline{F}\Euvect(M) = \Euvect(F(M))  
\]
for all $M \in \perf A$.

The representation matrix of 
the  map $[F]: K_{0}(A) \to K_{0}(A)$ with respect to the ordered basis $\{[P_{i}]\}_{i=1}^{r}$ is 
given by $C^{-1}\underline{F}C$. 
 Using the orthogonality in Lemma \ref{202509181508}(1), we obtain the following lemma. 

\begin{lemma}\label{2025091822221}
If   $F$ is an exact autofunctor  of $\perf A$. Then, 
The representation matrix of 
the  map $[F]: K_{0}(A) \to K_{0}(A)$ with respect to the ordered basis $\{[P_{i}]\}_{i=1}^{r}$ is 
given by 
$\underline{F}^{-t}$.
\end{lemma}

We leave the verification of the following lemma to the readers. 

\begin{lemma}\label{202509182222}

The following statement hold. 

\begin{enumerate}[(1)]

\item 
Let  $F$ be an exact endofunctor   of $\perf A$. Then, 
$\underline{F} = (\ERagl{P_{i}, FS_{j}})_{i,j}$.

\item $\underline{\nu} = - \Phi$ and hence $\underline{\nu_{1}} = \Phi$.

\end{enumerate}
\end{lemma}

It follows from Lemma \ref{202509182222}(2) that we have 
\begin{equation}\label{202509211245}
\Euvect(\nu(M)) = -\Phi \Euvect(M)
\end{equation}
for $M \in \perf A$.

\subsubsection{Weighted trace and  weighted Euler characteristic}

It is convenient to introduce the notions of the weighted trace and the weighted Euler characteristic.

For $M \in \perf A$, we define the \emph{weighted Euler characteristic} ${}^{v}\!\Euch(M)$ to be  
\[
{}^{v}\!\Euch(M)  = v^{t}\Euvect(M)= \sum_{i=1}^{r}v_{i} \Euch_{i}(M),
\]
i.e., the product of the row vector $v^{t}$ and the column vector $\Euvect(M)$. 
It follows that 
\begin{equation}\label{202102061837}
{}^{v}\!\Euch(F(M)) = {}^{\underline{F}^{t}(v)}\Euch(M). 
\end{equation}

In particular, it follows from \eqref{202509211245} that setting $\Psi := \Phi^{-t}$, we have 
\begin{equation}\label{202111101530}
{}^{v}\!\Euch(\nu_{1}^{-1}(M) )  = {}^{\Psi(v)} \Euch(M). 
\end{equation}

Let $M \in \perf A$ and $f: M \to M$. 
For $v =(v_{i})_{i=1}^{r} \in \kk^{\oplus r}$, 
we define the \emph{weighted trace} ${}^{v}\!\Tr(f)$ of $f$ to be the weighted sum of traces of $e_{i} \tuH(f): e_{i}\tuH(M) \to e_{i}\tuH(M)$. 
\[
{}^{v}\!\Tr(f) := \sum_{i=1}^{r} v_{i} \Tr\left[e_{i}\tuH(f): e_{i}\tuH(M) \to e_{i}\tuH(M)  \right]. 
\]
We point out the equality  ${}^{v}\!\Tr(\id_{M} ) = {}^{v}\!\Euch(M)$.

%
%
%
%
%

\subsection{Universal Auslander--Reiten triangle}

From now, we identify $K_{0}(A)$ with $\ZZ^{\oplus r}$ by the ordered basis $\{ [P_{i}] \}_{i=1}^{r}$. 
We set $K_{0}(A)_{\kk} := K_{0} (A) \otimes_{\ZZ} \kk \cong \kk^{\oplus r}$.  
We denote the extension of coefficient to $\kk$  of the Chern character map by the same symbol: 
\[
\ch: K_{0} (A)_{\kk}  \to \tuHH_{0}(A).
\]

For $v \in K_{0}(A)_{\kk}$, we denote by ${}^{v}\!\theta: A^{\vvee} \to A$ the element of $\Hom_{A^{\mre}}(A^{\vvee}, A)$ 
that corresponds to $\ch_{v} \in \tuHH_{0}(A)$ under the isomorphism $\tuHH_{0}(A) \cong \Hom_{A^{\mre}}(A^{\vvee}, A)$.
We denote a cone of ${}^{v}\!\theta$ by ${}^{v}\Xi$ and  set the induced exact triangle in $\perf A^{\mre}$ as follows: 
\[
{}^{v}\!\sfAR: A \xrightarrow{  \ {}^{v}\!\varrho \ } {}^{v}\Xi \xrightarrow{ \ {}^{v}\!\pi \ }\Pi_{1} \xrightarrow{ \ - {}^{v}\!\theta \ } \Sigma A. 
\]

For simplicity, we set $\ResEnd_{A}(M) := \End_{A}(M)/\rad \End_{A}(M)$ for $M \in \perf A$. 
We adopt the following notation: for a morphism $f: X \to Y $ in $\perf A^{\mre}$ and $M \in \perf A$, 
we set $f_{M} := f \lotimes M: X \lotimes_{A} M \to Y \lotimes_{A} M$. 
We also use  similar notation in the case that $M$ is a right dg $A$-module. 

\begin{theorem}\label{202508302231}
Let $A$ be a non-positive smooth and proper dg-algebra, 
$M$ an indecomposable object of $\perf A$ and $v \in K_{0}(A)_{\kk}$. 
Assume that ${}^{v}\!\Euch(M) \neq 0$. 
Then the exact triangle 
\[
{}^{v}\!\sfAR_{M}: M  \xrightarrow{  \ {}^{v}\!\rho_{M} \ } {}^{v}\Xi \lotimes_{A} M 
\xrightarrow{ \ {}^{v}\!\pi_{M} \ } \nu_{1}^{-1}( M )\xrightarrow{ \ - {}^{v}\!\theta_{M} \ } \Sigma M. 
\]
is AR-triangle in $\perf A$.

Moreover, if $\dim \ResEnd_{A}(M)= 1$, then 
the exact triangle ${}^{v}\!\sfAR_{M}$ splits if and only if ${}^{v}\!\Euch(M) = 0$. 
\end{theorem}

A key to the proof  is  the following result which is immediately obtained from   Lemma \ref{202508171844} and  Corollary \ref{202508301745}. 

\begin{corollary}[]\label{2025083017451}
In the above setting,
 let  $v :=\sum_{i=1}^{r}v_{i} [P_{i}] \in K_{0}(A)_{\kk}$. 
For a morphism  $f: M \to M $,  we have  
\[
\invserreagl{f, {}^{v}\!\theta_{M} } ={}^{v}\!\Tr(f). 
\]
\end{corollary}

\begin{remark}
This formula and  the universal Auslander--Reiten triangle were first established 
in \cite{QHA} for  the case where $A=\kk Q$ is the path algebra of a finite acyclic quiver $Q$, 
by direct computation of complexes. 
\end{remark}

\begin{proof}[Proof of Theorem \ref{202508302231}]
We apply Happel's criterion (Theorem \ref{Happel's criterion}) to $s:= {}^{v}\! \theta_{M}$.
If an endomorphism $f: M \to M$ belongs to the radical   $\rad \End_{A}(M)$, 
then it is nilpotent and hence $\Tr\left[e_{i}\tuH(f): e_{i}\tuH(M) \to e_{i}\tuH(M)  \right] = 0$ for all $i=1,2,\dots, r$. 
Consequently, by Corollary \ref{2025083017451}, we have $\invserreagl{f,s} = 0$ for any $f \in \rad \End_{A}(M)$. 
On the other hand, by Corollary \ref{2025083017451}, we have $\invserreagl{\id_{M},s} ={}^{v}\!\Euch(M)$. 
Thus by Happel's criterion, in the case ${}^{v}\!\Euch(M) \neq 0$, the exact triangle ${}^{v}\!\sfAR_{M}$ is an AR-triangle. 

Assume that $\dim \ResEnd_{A}(M) = 1$.  
Then, it is immediately follows from  Lemma \ref{202102071323} that  ${}^{v}\!\Euch(M) = 0$ if and only if  $s =0$.  
\end{proof}

We call an element $v$ of $K_{0}(A) \cong \kk^{\oplus r}$ a weight. 

\begin{definition}\label{202509171540}
A weight $v \in K_{0}(A)$ is said to be \emph{regular} 
if ${}^{v}\!\Euch(M)\neq 0$ for any indecomposable object $M$ of $\perf A$. 
\end{definition}

\begin{remark}
(1) 
We explain  the terminology. 
Assume that the base field $\kk$ is algebraically closed and of characteristic $0$. 
Let $Q$ be a Dynkin quiver and $A := \kk Q$ the path algebra. 
Then the vector space $K_{0}(Q) \cong \kk Q_{0}$ may be identified with the Cartan subalgebra $\mathfrak{h}$ of the semi-simple Lie algebra $\mathfrak{g}$ corresponding to $Q$. 
By Gabriel's theorem \cite{Gabriel}, which states that  the dimension vectors of indecomposable $A$-modules are precisely the positive roots of $\mathfrak{g}$, 
the regularity defined above   coincides with the  usual notion of regularity for an element  of the Cartan subalgebra $\mathfrak{h}$. 

(2) Since $\Euvect(M)$ belong to $\ZZ^{\oplus r}$ for any $M\in \perf A$, 
a regular weight $v$ exists provided that $[\kk: P] \geq r$ where $P$ denotes the prime field of $\kk$.  
\end{remark} 

Finally we obtain the following result.

\begin{corollary}[{Universal Auslander--Reiten triangle}]\label{202509171554}
Let $A$ be a non-positive smooth and proper dg-algebra 
and $v \in K_{0}(A)_{\kk}$ a regular weight. 
Then for any indecomposable object 
$M \in \perf A$, the exact triangle ${}^{v}\!\sfAR_{M}$ is an Auslander--Reiten triangle. 
\[
{}^{v}\!\sfAR_{M} : M  \xrightarrow{  \ {}^{v}\!\rho_{M} \ } {}^{v}\Xi \lotimes_{A} M 
\xrightarrow{ \ {}^{v}\!\pi_{M}  \ } \Sigma A^{\vvee} \lotimes_{A} M \xrightarrow{ \ - {}^{v}\!\theta_{M}  \ } \Sigma  M. 
\] 

In particular, AR-triangles admit a functorial construction.
\end{corollary}

\subsubsection{Compatibility with the action of the derived Picard group}

Recall from the discussion before Corollary \ref{202509171827} that 
$\Hom_{A^{\mre}}(A^{\vvee}, A) \cong \tuHH_{0}(A)$ acquires the action of the derived Picard group $\DPic_{\kk} (A)$, 
denoted  $T \cdot \phi$ for $\phi \in \Hom_{A^{\mre}}(A^{\vvee}, A)$ and $T \in \DPic_{\kk}A$. 
Abusing the notation, if  we write  $F:= T\lotimes_{A} -$  for the functor induced by $T \in \DPic_{\kk}(A)$, 
we simply set $F\cdot \phi := T\cdot \phi$. 
Similarly, we denote by $\gamma_{F}:  \sfS^{-1}F  \to  F\sfS^{-1}$ the natural isomorphism 
induced from the canonical isomorphism $\gamma_{T}: A^{\vvee} \lotimes_{A} T \to T\lotimes_{A} A^{\vvee}$ 
given in the discussion before Corollary   \ref{202509201321}. 

As in the final part of the proof of Lemma \ref{202508171844}, 
we can show that the morphism $(\lcounit_{T})^{\vvee}:A^{\vvee} \to (T\lotimes_{A} T^{\lvvee})^{\vvee}$ 
coincides with the composition 
\[
A^{\vvee} \cong A \lotimes_{A} A^{\vvee} \xrightarrow{ \ \runit_{T} \lotimes \id \  } (T \lotimes_{A} T^{\rvvee})\lotimes_{A}A^{\vvee} 
\cong (T \lotimes_{A} T^{\lvvee})^{\vvee}.
\]
It follows from \eqref{202509210026} that 
the natural transformation $F \cdot \phi: \sfS^{-1} \to \id$ is  given by the  composition 
\[
F\cdot \phi : \nu^{-1}\xrightarrow{\cong} FF^{-1} \sfS^{-1}  \xrightarrow{ F((\gamma_{F^{-1}})^{-1}) } F\sfS^{-1} F^{-1} \xrightarrow{ \ F \phi F^{-1} \ } F F^{-1}\xrightarrow{\cong} \id.
\] 
Now it is immediate to deduce the following lemma.

\begin{lemma}\label{202101131703}
Let $T \in \DPic_{\kk}(A)$ and 
 $F:= T \lotimes_{\Pa}-$  the associated  exact autoequivalence of $\perf A$. 
Then for $ \phi \in \Hom_{A^{\mre}}(A^{\vvee}, A)$, we have 
\[
F( ( F^{-1}\cdot\phi)_{M}) = (\phi_{F(M)} ) \gamma_{F, M}. 
\]
In other words, the following diagram is commutative
\[
\begin{xymatrix}@C=60pt{
\sfS^{-1}F(M) \ar[r]^{\phi_{F(M)} }  \ar[d]_{\gamma_{F,M}} 
& F(M) \ar@{=}[d] \\
F\sfS^{-1}(M) \ar[r]_{ \  F( (F^{-1} \cdot \phi)_{M}) } &F(M)
}\end{xymatrix}
\]
\end{lemma}

The next lemma  that provides a way to compute $F\cdot ({}^{v}\!\theta)$ is an immediate consequence of Corollary \ref{202509171827} and   Lemma \ref{2025091822221}.
\begin{lemma}\label{202101081603}
Let $T \in \DPic_{\kk}(A)$ and 
 $F:= T \lotimes_{\Pa}-$  the associated  exact autoequivalence of $\perf A$. 
Then, for $v \in K_{0}(A)$, 
we have the following equality in $\Hom_{A^{\mre}}(A^{\vvee}, A)$ 
\[
F\cdot ( {}^{v}\!\theta) = {}^{\underline{F}^{-t}(v)} \theta. 
\]
\end{lemma}

Now, we give a functoriality of universal Auslander--Reiten triangle.

\begin{theorem}\label{202101131644}
Let $v \in K_{0}(A)$ and $T \in \DPic (A)$.  
We denote by 
 $F:= T \lotimes_{\Pa}-$  the associated  exact autoequivalence of $\perf A$. 
Then for $M \in \ind\perf A$ such that $\dim \ResEnd_{A}(M) =1$ and ${}^{v}\!\Euch(M) \neq 0$, the following holds. 

\begin{enumerate}[(1)] 
\item 
We have the following equality in $\Hom_{A}(\sfS^{-1}(M), M)$ 
\[
(F^{-1}\cdot{}^{v}\!\theta)_{M} = \frac{{}^{v}\!\Euch (F(M))}{{}^{v}\!\Euch (M)} {}^{v}\!\theta_{M}.
\]

\item We have the following equality in $\Hom_{A}(\sfS^{-1}F(M), F(M))$ 
\[
\frac{{}^{v}\!\Euch (F(M))}{{}^{v}\!\Euch (M)} F( {}^{v}\!\theta_{M}) \gamma_{F, M} = {}^{v}\!\theta_{F(M)}. 
\]


\item 
Assume moreover that ${}^{v}\!\Euch(F(M)) \neq 0$. 
Then,   
there exists an isomorphism ${}^{v}\!\sfa_{F, M}: {}^{v}\!\Xi \lotimes_{A} F(M) \to F({}^{v}\!\Xi\lotimes_{A} M)$ 
that gives the following isomorphism of exact triangles:
\[
\begin{xymatrix}@C=55pt{
\nu^{-1}F(M) \ar[r]^{{}^{v}\!\theta_{F(M)}}  \ar[d]_{\gamma_{F, M}} & 
F(M) \ar[r]^{{}^{v}\!\varrho_{F(M)}} \ar@{=}[d]  &
{}^{v}\Xi \lotimes_{A}F(M) \ar[d]^{{}^{v}\!\sfa_{F, M} } \ar[r]^{{}^{v}\!\pi_{F(M)} } &
\Sigma\nu^{-1}F(M) \ar[d]^{ \Sigma \gamma_{F, M}}    \\
F\nu^{-1}(M) \ar[r]_{x F({}^{v}\!\theta_{M}) } & 
F(M) \ar[r]_{F({}^{v}\!\varrho_{M})}  &
F({}^{v}\Xi \lotimes_{A} M) \ar[r]_{x^{-1}F({}^{v}\!\pi_{M}) } &
\Sigma F\nu^{-1}(M)   \\
}\end{xymatrix}
\]
where \[x: =\frac{{}^{v}\!\Euch (F(M))}{{}^{v}\!\Euch (M)}. 
\]
\end{enumerate}
\end{theorem}

\begin{proof}
(1) By Lemma \ref{202508171844}, we have $(F^{-1}\cdot {}^{v}\!\theta)_{M} = \sigma_{M}(\ch_{F^{-1}(v)})$. 
Thus by Corollary \ref{2025083017451}, we have 
\[
\begin{split}
\invserreagl{ \id_{M},  ( F^{-1}\cdot {}^{v}\!\theta)_{M} } 
&= 
\invserreagl{ \id_{M},    \sigma_{M}(\ch_{F^{-1}(v)}) } \\ 
&= \sum_{i =1}^{r} v_{i} \ERagl{F^{-1}(P_{i}), M }\\
& = \sum_{i =1}^{r} v_{i} \ERagl{P_{i}, F(M)}
= {}^{v}\Euch(F(M))
\end{split}
\]
We deduce the desired equation by using Lemma \ref{202102071323}. 

(2) follows from (1) together with Lemma \ref{202101131703}. 

(3) follows from (2) and the uniqueness of Auslander--Reiten triangle. 
\end{proof}
%
%
%
%

We remark that the morphism ${}^{v}\!\sfa_{F, M}$ is not uniquely determined, but it is determined modulo the radical bi-functor $\rad$.


  \subsection{Right modules and dualities}\label{Remarks on the right versions}

  Applying the same argument for $A^{\op}$ we can obtain  right versions of above results. 
 In this subsection, we discuss relationships between the previous results and their right versions 
 through dualities. 
 For this purpose first we fix notations and terminologies for right modules.

\subsubsection{The weighted Euler characteristic for right modules}

 Let $N \in \perf A$.   
  We set \[
  \Euvect(N) : = (\Euch(\tuH(N) e_{1} ), \ldots, \Euch(\tuH(N) e_{r}))\]
   and regard it as a row vector. 
    The weighted Euler characteristic for   $v \in \kk^{r}=K_{0}(A)$ is defined by the formula 
  \[
  {}^{v}\!\Euch(N) := \Euvect(N) v = \sum_{i =1}^{r}v_{i} \Euch(N e_{i}). 
  \]

  \subsubsection{Regularity for right dg-modules } 
 
The $\kk$-duality $\tuD(-)$ descents to the transpose of dimension vectors i.e, $\Euvect(\tuD(M)) = \Euvect(M)^{t}$. 
It follows that  ${}^{v}\!\Euch(\tuD(M)) = {}^{v}\!\Euch(M)$. 
Since the functor $\tuD(-):  \perf A \to \perf A^{\op}$ gives a contravariant equivalence of categories, 
we obtain the following lemma.

\begin{lemma}\label{202103021753}
An element  $v \in K_{0}(A)_{\kk}$  is regular (for left modules)
if and only if it satisfies regularity for right modules, 
i.e., 
we have ${}^{v}\!\Euch(N) \neq 0$ in $\kk$ for any indecomposable object $N$ of $\perf A^{\op}$.

\end{lemma}

  \subsubsection{The Coxeter matrix $\Phi_{\textup{right}}$ for right modules}
 We remark that the Coxeter matrix $\Phi_{\textup{right}}$ for right modules is different from that for left modules which we wrote $\Phi$. 
  In \cite{QHA}, we will deal with left modules and right modules at the same time, for this we need to know  relationship between $\Phi$ and $\Phi_{\textup{right}}$.

  Let $C$ be the Cartan matrix of $\Pa$. 
  Then the Coxeter matrix for right modules is given as 
   $\Phi_{\textup{right}}: = - C^{-1} C^{t}$. 
   Observe that $\Phi_{\textup{right}}  =\Phi^{-t} = \Psi$. 
   Thus, we have the equality below for all $N \in \perf A^{\op}$. 
  \[
  \begin{split}
 & \Euvect(\nu_{1}(N)) = \Euvect(N) \Phi_{\textup{right}}= \Euvect(N) \Psi, \ {}^{v}\Euvect(\nu_{1}(N)) = {}^{\Psi(v)}\Euvect(N),\\
&\Euvect(\nu_{1}^{-1}(N) ) = \Euvect(N) \Phi_{\textup{right}}^{-1} =\Euvect(N)\Psi^{-1}, \
{}^{v}\!\Euvect(\nu_{1}^{-1}(N) ) = {}^{\Psi^{-1}(v)}\Euvect(N).
\end{split}
\]

In particular we deduce the following assertion. 
\begin{lemma}\label{202103021744}
Assume that $v \in K_{0}(A)_{\kk}$ is an eigenvector of $\Psi$ with the eigenvalue $\lambda$. 
Then for $N \in \perf A^{\op}$, we have 
\[
{}^{v}\!\Euch(\nu_{1}^{-1}(N) ) = \frac{1}{\lambda} {}^{v}\!\Euch(N). 
\]
\end{lemma}

  \subsubsection{}

  We point out   relationships of the weighted Euler characteristics and $A$-dualities

For $M \in \perf A$,   
by combining 
  Corollary \ref{202508181452} and Corollary \ref{202508181504}
we obtain the canonical  isomorphism $M^{\lvvee} \cong  \tuD(M) \lotimes_{A} A^{\vvee} $,   
from which we  deduce the following identities:  
\begin{equation}\label{202104021807} 
\begin{split} 
&\Euvect(M^{\lvvee} ) = (-\Phi^{-1} \Euvect(M))^{t}, \ \ 
 {}^{v}\!\Euch(M^{\lvvee}) = - {}^{\Phi^{-1}(v)}\Euch(M). 
  \end{split}
  \end{equation}

Similarly, for $N \in \perf A^{\op}$, we have the following identities:
\begin{equation}\label{2021040218072} 
\begin{split} 
&\Euvect(N^{\rvvee} ) = (- \Euvect(N)\Psi)^{t}, \ \ 
 {}^{v}\!\Euch(N^{\rvvee}) = - {}^{\Psi(v)}\Euch(N). 
  \end{split}
  \end{equation}

\subsubsection{Duality}

\begin{proposition}\label{202103051610}
Let $N \in \ind \perf A^{\op}$. 
Assume that ${}^{\Psi(v)} \!\Euch(N)\neq 0, {}^{\Psi^{2}(v)} \!\Euch(N)\neq 0$. 
Then,  
there exists an isomorphism 
\[
{}^{v}\!\sfb_{N}: \RHom_{A^{\op}}(N \lotimes_{A} {}^{v} \Xi, \PPi_{1}) \to {}^{v} \Xi \lotimes_{A} N^{\rvvee}
\] 
that completes the following commutative diagram 
\begin{equation}\label{202103051621}
\begin{xymatrix}{
 (\Sigma  N, \PPi_{1}) \ar[r] \ar[d]_{\cong}  &
  (N\lotimes_{A} \PPi_{1}, \PPi_{1}) \ar[d]_{\cong} \ar[r] &
  (N\lotimes_{A} {}^{v} \Xi, \PPi_{1}) \ar[d]_{\cong}^{{}^{v}\!\sfb_{N}} \ar[r] &
  \Sigma (\Sigma N, \PPi_{1}) \ar[d]_{\cong} \\
A^{\vvee} \lotimes_{A} N^{\rvvee} 
\ar[r]_{  -y({}^{v}\!\theta_{N^{\rvvee} })} 
&  N^{\rvvee} \ar[r]_{{}^{v}\!\varrho_{N^{\rvvee}}} 
& {}^{v} \Xi \lotimes_{A} N^{\rvvee} \ar[r]_{-y^{-1}({}^{v}\!\ppi_{1, N^{\rvvee}}) }
& \PPi_{1} \lotimes_{A} N^{\rvvee}
}\end{xymatrix} 
\end{equation}
where we use the  abbreviation  $(-,+) =\RHom_{\Pa^{\op}}(-, +)$, 
the top row is the exact triangle $\RHom( {}_{N}{}^{v}\!\sfAR, \PPi_{1})$ 
and we set
\[
y := \frac{{}^{\Psi^{2}(v)} \!\Euch(N)}{{}^{\Psi(v)} \!\Euch(N)}.
\]
\end{proposition}

To prove the proposition, 
we need to recall a nice and important observation due to  van den Bergh. 

\begin{proposition}[{\cite[Proposition 14.1]{Van den Bergh: Calabi-Yau}}]\label{202511221600}
The  map $(-)^{\vvee}_{\Hom}: \Hom_{A^{\mre}}(A^{\vvee}, A) \to \Hom_{A^{\mre}}(A^{\vvee}, A)$ associated to the endofunctor $(-)^{\vvee}$ of 
$\perf A$ is the identity map. 
Thus in particular, $({}^{v}\!\theta)^{\vvee} = {}^{v}\theta$.  
\end{proposition}

\begin{proof}[Proof of Proposition \ref{202103051610}]
First note that we have 
${}^{\Psi(v)} \!\Euch(N) ={}^{v}\!\Euch(N^{\rvvee} ), 
{}^{\Psi^{2}(v)} \!\Euch(N) =
{}^{v}\!\Euch( \nu_{1}( N^{\rvvee}) )$. 

We claim that we have the following commutative diagram 
\begin{equation}\label{202105231814}
\begin{xymatrix}@C=70pt{
(\Sigma N)^{\rvvee} \ar[d]_{\cong}\ar[r]^{(-\Sigma ({}_{N}\!\! {}^{v}\! \theta))^{\rvvee}} 
&  (N \lotimes \PPi_{1})^{\rvvee}\ar[d]_{\cong}\\ 
\Sigma^{-1} \tuD(A) \lotimes_{A}  A^{\vvee} \lotimes_{A} N^{\rvvee} 
\ar[r]^{  -y({}_{ \Sigma^{-1} \tuD(A)}\!\!{}^{v}\!\theta_{N^{\rvvee} })} 
&  \Sigma^{-1} \tuD(A)  \lotimes_{A} N^{\rvvee} 
}\end{xymatrix}
\end{equation} 
Indeed we obtain this diagram in the following way:  
\[
\begin{xymatrix}@C=70pt{
( \Sigma N)^{\rvvee} \ar[d]_{\cong}\ar[r]^{(- \Sigma ( {}_{N}\!\!{}^{v}\! \theta))^{\rvvee}} 
&  (N \lotimes \PPi_{1})^{\rvvee}\ar[d]_{\cong}\\ 
(\Sigma A)^{\vvee} \lotimes_{A} \tuD(A) \lotimes_{A} N^{\rvvee} 
\ar[d]_{\cong}\ar[r]^{ (- \Sigma ( {}^{v}\!\theta) )^{\vvee}_{\tuD(A) \lotimes N^{\rvvee} }} 
&  (\Sigma A^{\vvee} )^{\vvee} \lotimes_{A} \tuD(A) \lotimes_{A} N^{\rvvee} \ar[d]_{\cong}\\ 
A^{\vvee} \lotimes_{A} \Sigma^{-1} \tuD(A) \lotimes_{A} N^{\rvvee} 
\ar[d]_{\cong}\ar[r]^{ (-{}^{v}\!\theta)^{\vvee}_{\Sigma^{-1} \tuD(A)  \lotimes N^{\rvvee} }} 
&  (A^{\vvee} )^{\vvee} \lotimes_{A} \Sigma^{-1} \tuD(A)  \lotimes_{A} N^{\rvvee} \ar[d]_{\cong}\\ 
A^{\vvee} \lotimes_{A} \Sigma^{-1} \tuD(A) \lotimes_{A} N^{\rvvee} 
\ar[d]_{\cong}^{\gamma_{ \Sigma^{-1} \tuD(A)} } \ar[r]^{ -{}^{v}\!\theta_{ \Sigma^{-1}\tuD(A) \lotimes N^{\rvvee} }} 
&  \Sigma^{-1} \tuD(A)   \lotimes_{A} N^{\rvvee} \ar@{=}[d]\\ 
\Sigma^{-1}\tuD(A) \lotimes_{A}  A^{\vvee} \lotimes_{A} N^{\rvvee} 
\ar[r]^{  -y\Sigma^{-1}({}_{\tuD(A)}\!\!{}^{v}\!\theta_{N^{\rvvee} })} 
& \Sigma^{-1}\tuD(A)   \lotimes_{A} N^{\rvvee} 
}\end{xymatrix}
\]
where for the commutativity of the first square 
we use the  canonical isomorphism   
\[
(L \lotimes_{A} X)^{\rvvee} 
\cong X^{\rvvee} \lotimes_{A} L^{\rvvee} \cong X^{\vvee} \lotimes_{A} \tuD(A) \lotimes_{A} L^{\rvvee}
\]
obtained from Lemma \ref{202508201805}.
For the third square we use Proposition \ref{202511221600}.  
Finally the commutativity of the fourth  square follows from Proposition \ref{202101131644}.

We apply $\PPi_{1}\lotimes_{\Pa} - $ to \eqref{202105231814} and use
  canonical isomorphisms 
\[
\RHom_{A^{\op}}(- ,\PPi_{1} ) \cong \PPi_{1} \lotimes_{A} (-)^{\rvvee}, \ 
\PPi_{1} \lotimes_{\Pa} \Sigma^{-1}\tuD(A) \cong \Pa,  
\]
then we obtain a commutative square that appeared in the  diagram \eqref{202103051621}.  
\end{proof}

We remark that $\sfb_{N}$ is not uniquely determined, but it is unique modulo $\rad$. 
We point out that $\sfb_{N}$ has functoriality modulo $\rad$.

\subsection{Bimodules version}

In this subsection, we establish isomorphisms over $\Pa^{\mre}$ that involve ${}^{v} \Xi$. 

\subsubsection{} 

We deduce the following result  from Lemma \ref{202101081603}. 
\begin{proposition}\label{2021010816031}
Let $v \in K_{0}(A)_{\kk}$ and $T \in \DPic_{k} (A)$. 
Then, 
there exists an isomorphism 
\[
{}^{v}\!\sfc_{ T}:  {}^{v} \Xi \lotimes_{\Pa} T \to T  \lotimes_{\Pa} {}^{\underline{T}^{t}(v)}\Xi
\]
in $\perf  \Pa^{\mre}$ that gives the following isomorphism of exact triangles:
\[
\begin{xymatrix}@C=60pt{
\Pa^{\vvee} \lotimes_{\Pa} T \ar[r]^{ {}^{v}\!\theta_{T} }  \ar[d]_{\gamma_{T}} 
& 
T\ar[r]^{{}^{v}\!\rrho_{T}} \ar@{=}[d]  &
{}^{v} \Xi \lotimes_{A}T \ar[d]^{{}^{v}\!\sfc_{T} } \ar[r]^{{}^{v}\!\ppi_{T} } &
\PPi_{1} \lotimes_{\Pa} T  \ar[d]^{\tilde{\gamma}_{T} } \\
T \lotimes_{\Pa} \Pa^{\vvee} \ar[r]_{ {}_{T}(  {}^{\underline{T}^{t}(v) } \theta) } & 
T  \ar[r]_{  {}_{T} ({}^{\underline{T}^{t}(v)} \rrho) }  &
T \lotimes_{\Pa}  {}^{\underline{T}^{t}(v)} \Xi \ar[r]_{ {}_{T} ({}^{\underline{T}^{t}(v)}\ppi) } &
T \lotimes_{\Pa} \PPi_{1}.   \\
}\end{xymatrix}
\]
where $\tilde{\gamma}_{T}$ is the composition 
\[
\tilde{\gamma}_{T} : ( \Sigma A) \lotimes_{A}   A^{\vvee} \lotimes_{A} T \xrightarrow{ \ (\Sigma A)\lotimes   \gamma_{T}  \ } \Sigma A \lotimes_{A} T \lotimes_{A} A^{\vvee} 
\xrightarrow{ \ \delta_{T}^{-1} \lotimes  A^{\vvee}  \ } 
T \lotimes_{A} ( \Sigma A ) \lotimes_{A} A^{\vvee}. 
\]
\end{proposition}

In the case where the weight $v \in K_{0}(A)_{\kk}$ is an eigenvector of $\underline{T}^{t}$ we obtain the following corollary.

\begin{corollary}\label{2021010816032}
Let $v \in K_{0}(A)_{\kk}$ and $T$ be a two-sided tilting complex over $\Pa$. 
Assume that $v$ is an eigenvector of  $\underline{T}^{t}$ with the eigenvalue $\lambda$. 
Then, 
there exists an isomorphism 
\[{}^{v}\!\sfc'_{T}: {}^{v} \Xi \lotimes_{\Pa} T \to T \lotimes_{\Pa} {}^{v} \Xi
\] in $\sfD(\Pa^{\mre})$ 
that gives the following isomorphism of exact triangles:
\[
\begin{xymatrix}@C=60pt{
\Pa^{\vvee} \lotimes_{\Pa} T \ar[r]^{{}^{v}\!\theta_{T} }  \ar[d]_{\gamma_{T}} 
& 
T   \ar[r]^{{}^{v}\!\rrho_{T}} \ar@{=}[d]  &
{}^{v} \Xi \lotimes_{A}T \ar[d]^{{}^{v}\!\sfc'_{T} } \ar[r]^{{}^{v}\!\ppi_{1, T} } &
\PPi_{1} \lotimes_{\Pa} T  \ar[d]^{\tilde{\gamma}_{T}}    \\
T \lotimes_{\Pa} \Pa^{\vvee} \ar[r]_{  \lambda ( {}_{T}  \!{}^{v}\!\theta) } & 
T  \ar[r]_{  {}_{T} {}^{v}\!\rrho }  &
T \lotimes_{\Pa}  {}^{v} \Xi \ar[r]_{ \lambda^{-1}( {}_{T}{}^{v}\!\ppi)} &
T \lotimes_{\Pa} \PPi_{1}   \\
}\end{xymatrix}
\]
\end{corollary}

\subsubsection{}

We collect the case $T = \PPi_{1}$ of Corollary \ref{2021010816032}, since it plays a key role in \cite{QHA}.
A point here is the equality $\tilde{\gamma}_{\Pi_{1}} =\id_{\Pi_{1} \lotimes \Pi_{1}}$, which follows from Lemma \ref{202511031307}. 

\begin{corollary}\label{20210108160321}
Assume that $v \in K_{0}(A)_{\kk}$ is an eigenvector of  $\Psi$ with the eigenvalue $\lambda$. 
Then, 
there exists an isomorphism 
\[{}^{v}\!\sfc'_{\PPi_{1}}: {}^{v} \Xi \lotimes_{\Pa} \PPi_{1} \to \PPi_{1} \lotimes_{\Pa} {}^{v} \Xi
\] in $\perf A^{\mre}$ 
that gives the following isomorphism of exact triangles:
\begin{equation}\label{202306211412}
\begin{xymatrix}@C=60pt{
\Pa^{\vvee} \lotimes_{\Pa} \PPi_{1} \ar[r]^{{}^{v}\!\theta_{\PPi_{1}} }  \ar[d]_{\gamma_{\PPi_{1}}} 
& 
\PPi_{1}   \ar[r]^{{}^{v}\!\rrho_{\PPi_{1}}} \ar@{=}[d]  &
{}^{v} \Xi \lotimes_{A}\PPi_{1} \ar[d]^{{}^{v}\!\sfc'_{\PPi_{1}} } \ar[r]^{{}^{v}\!\ppi_{\PPi_{1}} } &
\PPi_{1} \lotimes_{\Pa} \PPi_{1}  \ar@{=}[d]   \\
\PPi_{1} \lotimes_{\Pa} \Pa^{\vvee} \ar[r]_{  \lambda ({}_{\PPi_{1}}  {}^{v}\!\theta) } & 
\PPi_{1}  \ar[r]_{  {}_{\PPi_{1}} {}^{v}\!\rrho }  &
\PPi_{1} \lotimes_{\Pa}  {}^{v} \Xi \ar[r]_{ \lambda^{-1}( {}_{\PPi_{1}} {}^{v}\!\ppi) }  &
\PPi_{1} \lotimes_{\Pa} \PPi_{1}   \\
}\end{xymatrix}
\end{equation}
\end{corollary}

\subsubsection{The right duality}

\begin{lemma}\label{202102230750} 
There exists an isomorphism 
\[{}^{v}\!\sfd: ({}^{v} \Xi)^{\rvvee} \to \Sigma^{-1} \tuD(A)\lotimes_{\Pa} ({}^{\Psi^{-1}(v)} \Xi) \]
in $\perf A^{\mre}$  
that completes the  following commutative diagram
\begin{equation}\label{202105232003}
\begin{xymatrix}@C=18pt{
(\Sigma A)^{\rvvee} \ar[d]_{\cong}\ar[r]^{(-{}^{v}\!\theta [1])^{\rvvee}} &
\PPi_{1}^{\rvvee}\ar[d]_{\cong}  \ar[r]^{ ({}^{v}\!\ppi_{1})^{\rvvee} } & 
({}^{v} \Xi)^{\rvvee} \ar[d]^{{}^{v}\!\sfd} \ar[r]^{ ({}^{v}\!\rrho)^{\rvvee}} & 
A^{\rvvee} \ar[d]^{\cong} \\ 
\Sigma^{-1} \tuD(A) \lotimes_{A}  A^{\vvee}  
\ar[r]_-{  -{}_{ \Sigma^{-1}\tuD(A) }^{\Psi^{-1}(v)}\theta } 
& \Sigma^{-1} \tuD(A)    \ar[r]_-{{}_{ \Sigma^{-1} \tuD(A)}^{\Psi^{-1}(v)} \rrho} & 
\Sigma^{-1} \tuD(A) \lotimes_{\Pa} ({}^{\Psi^{-1}(v)} \Xi )  \ar[r]_-{- {}_{\Sigma^{-1}\tuD(A)}^{\Psi^{-1}(v)}\ppi_{1} }
 & 
\Sigma^{-1} \tuD(A)  \lotimes_{\Pa} \PPi_{1}.
}\end{xymatrix} 
\end{equation}
\end{lemma} 

\begin{proof}
We  can verify commutativity of  the left most  square of \eqref{202105232003} in the following way  
\[ 
\begin{xymatrix}@C=70pt{
(\Sigma A)^{\rvvee} \ar[d]_{\cong}\ar[r]^{(- \Sigma({}^{v}\!\theta))^{\rvvee}} 
&  \PPi_{1}^{\rvvee}\ar[d]_{\cong}\\ 
(\Sigma A)^{\vvee} \lotimes_{A} \tuD(A) 
\ar[d]_{\cong}\ar[r]^{ (- \Sigma({}^{v}\!\theta ))^{\vvee}_{\tuD(A) }} 
&  (\Sigma A^{\vvee} )^{\vvee} \lotimes_{A} \tuD(A) \ar[d]_{\cong}\\ 
A^{\vvee} \lotimes_{A} \Sigma^{-1} \tuD(A)  
\ar[d]_{\cong}\ar[r]^{ (-{}^{v}\!\theta)^{\vvee}_{\Sigma^{-1} \tuD(A)   }} 
&  (A^{\vvee} )^{\vvee} \lotimes_{A} \Sigma^{-1} \tuD(A)    \ar[d]_{\cong}\\ 
A^{\vvee} \lotimes_{A} \Sigma^{-1}\tuD(A)  
\ar[d]_{\cong}^{\gamma_{ \Sigma^{-1}\tuD(A)} } \ar[r]^{ -{}^{v}\!\theta_{\Sigma^{-1}\tuD(A)   }} 
& \Sigma^{-1}\tuD(A)  \ar@{=}[d]\\ 
\Sigma^{-1} \tuD(A) \lotimes_{A}  A^{\vvee}  
\ar[r]^{  - {}_{\Sigma^{-1} \tuD(A) }^{\Psi^{-1}(v)}\theta } 
& \Sigma^{-1}\tuD(A)   
}\end{xymatrix}
\]
where for the commutativity of the first square  we use  the canonical isomorphism $X^{\rvvee} 
\cong X^{\vvee} \lotimes_{A} \tuD(A)$ given in  Lemma \ref{202508201805}.
For the third square we use Proposition \ref{202511221600}.  
Finally the commutativity of the fourth  square follows from Lemma \ref{202101131703} and Lemma \ref{202101081603}.
\end{proof}

Since $\RHom_{\Pa^{\op}}(-, \PPi_{1}) \cong \PPi_{1} \lotimes_{\Pa} (-)^{\rvvee}$, 
applying $\PPi_{1} \lotimes_{\Pa} -$ to the diagram of Lemma \ref{202102230750} 
we deduce the following lemma. 

\begin{lemma}\label{2021022307501} 

There exists an isomorphism 
\[
{}^{v}\!\sfe: \RHom_{\Pa^{\op}} ( {}^{v} \Xi, \PPi_{1} ) \to {}^{\Psi^{-1}(v)} \Xi 
\]
in $\perf A^{\mre}$  
that completes the  following commutative diagram
\[
\begin{xymatrix}@C=60pt{
(A[1], \PPi_{1}) \ar[r]^{(- {}^{v}\!\theta[1], \PPi_{1} ) }\ar[d]_{\cong}  &
 (\PPi_{1}, \PPi_{1}) \ar[d]_{\cong} \ar[r]^{({}^{v}\!\ppi_{1}, \PPi_{1} ) } &
 ({}^{v} \Xi, \PPi_{1}) \ar[d]_{\cong}^{{}^{v}\!\sfe} \ar[r]^{ ({}^{v}\!\rrho, \PPi_{1} ) } &
 (A, \PPi_{1}) \ar[d]_{\cong} \\
A^{\vvee} 
\ar[r]_{  -{}^{\Psi^{-1}(v)}\!\theta} 
& A  \ar[r]_{{}^{\Psi^{-1}(v)}\!\rrho } 
& {}^{\Psi^{-1}(v)} \Xi \ar[r]_{ -{}^{\Psi^{-1}(v)}\!\ppi_{1} }
& \PPi_{1} .
}\end{xymatrix} 
\]
where we use the  abbreviation  $(-,+) =\RHom_{\Pa^{\op}}(-, +)$.

\end{lemma} 

Finally, we deduce the following corollary.

\begin{corollary}\label{2021022307502} 
Assume that $v$ is an eigenvector of $\Psi$ with the eigenvalue $\lambda$. 
Then, there exists an isomorphism 
\[
{}^{v}\!\sfe': \RHom_{\Pa^{\op}} ( {}^{v} \Xi, \PPi_{1} ) \to {}^{v} \Xi 
\]
in $\sfD(\Pa^{\mre})$  
that completes the  following commutative diagram
\[
\begin{xymatrix}@C=60pt{
(A[1], \PPi_{1}) \ar[r]^{(- {}^{v}\!\theta[1], \PPi_{1} ) }\ar[d]_{\cong}  &
 (\PPi_{1}, \PPi_{1}) \ar[d]_{\cong} \ar[r]^{({}^{v}\!\ppi_{1}, \PPi_{1} ) } &
 ({}^{v} \Xi, \PPi_{1}) \ar[d]_{\cong}^{{}^{v}\!\sfe'} \ar[r]^{ ({}^{v}\!\rrho, \PPi_{1} ) } &
 (A, \PPi_{1}) \ar[d]_{\cong} \\
A^{\vvee} 
\ar[r]_{  -\lambda^{-1} {}^{v}\!\theta} 
& A  \ar[r]_{{}^{v}\!\rrho } 
& {}^{v} \Xi \ar[r]_{ -\lambda {}^{v}\!\ppi_{1} }
& \PPi_{1} .
}\end{xymatrix} 
\]
where we use the  abbreviation  $(-,+) =\RHom_{\Pa^{\op}}(-, +)$.

\end{corollary}

\appendix 
\section{Happel's criterion}\label{subsection: Happel's criterion}\label{202602050957}

The aim of  this section is 
to recall Happel's criterion for AR-coconnecting morphisms from \cite{Happel Book} and to establish a lemma that is used in the main body of this paper. 

Let $\sfD$ be a $\Hom$-finite triangulated category and $\nu$ a Serre functor of $\sfD$.  
Recall that $\nu_{1}:= \nu[-1]$ is the Auslander--Reiten translation in the Auslander--Reiten theory of a triangulated category.
By $\ind \sfD$ we denote the set of indecomposable objects of $\sfD$. 
For simplicity, we set $\ResEnd_{\sfD}(X):= \End_{\sfD}(X)/\rad \End_{\sfD}(X)$ for an object $X$ of $\sfD$. 

Let  $X \in \ind  \sfD$.  
A morphism $s: \nu^{-1}(X) \to X$ is called \emph{Auslander--Reiten (AR)-coconnecting} 
if it is a coconnecting morphism of  an AR-triangle starting from $X$. 
In other words, we have an AR-triangle of the following form:
\[
X \to Y \to \nu_{1}^{-1}(X)  \xrightarrow{ \ -s[1] \ } X[1].
\]
%
%

\begin{theorem}[{\cite[p37]{Happel Book}}]\label{Happel's criterion}
Let $X \in \ind \sfD$. 
Then a morphism $s: \nu^{-1}(X) \to X$ is  an AR-coconnecting morphism  
if the following equations hold
\[
\isagl{\id_{X}, s} \neq 0, \ 
\isagl{f, s} = 0
\]
where  $f \in \rad \End_{\sfD}(X)$. 
The converse holds if $\dim \ResEnd_{\sfD}(X) = 1$. 
\end{theorem}

Let $X$ be an indecomposable object of $\sfD$. 
If $\dim \ResEnd_{\sfD}(X) = 1$, then
the  subspace $\rad \End_{\sfD}(X)^{\perp} \subset \Hom_{\sfD}(\nu^{-1}(X), X)$ which is 
orthogonal to $\rad \End_{\sfD}(X)$ 
 with respect to the pairing $\isagl{-,+}$ is one-dimensional. 
%
Thus we have the following lemma.

\begin{lemma}\label{202102071323}
Let $X$ be an  indecomposable object of $X$ such that $\dim \ResEnd_{\sfD}(X) = 1$. 
Then, the following holds. 
\begin{enumerate}[(1)] 
\item Two AR-coconnecting morphisms $s,t$ to $X$ are proportional over $\kk$ to each other. 
More precisely we have 
\[
t = \frac{\isagl{\id_{X}, t}}{\isagl{\id_{X}, s}} s. 
\]

\item An element $s \in \Hom_{\sfD}(\nu^{-1}(X), X)$ is $0$ if and only if it satisfies 
\[
\isagl{\id_{X}, s } \neq 0, \ \ \isagl{f, s} = 0 \textup{ for } f \in \rad\End_{\sfD}(X). 
\]
\end{enumerate}
\end{lemma}

\section{Adjoint $1$-morphisms and their uniqueness}\label{202511221618}

We collect standard facts about adjoint $1$-morphisms using the example of dg-Morita (weak) $2$-category, 
since the uniqueness of adjoint $1$-morphisms plays an important role in the main body of the paper. 

In this section  $A, B$ denote dg-algebras.

\begin{definition}[{\cite[Definition 4.4.1(3)]{Asashiba:book},\cite[Definition 6.1.1]{NY}}]\label{202509221150}

An adjoint pair $(X, Y, \eta, \epsilon)$ is a $4$-tuple consisting of $X \in \sfD(A \otimes B^{\op}), \  Y \in \sfD(B\otimes A^{\op})$ 
and morphisms $\eta: B \to Y \lotimes_{A} X$ in $\sfD(B^{\mre})$ and 
$ \epsilon: X \lotimes_{B} Y \to A$ in $\sfD(A^{\mre})$ that satisfy the triangle identities 
$(\epsilon \lotimes_{A} X)(X \lotimes_{B} \eta ) = \id_{X}$ and 
$(Y \lotimes_{A} \epsilon)(\eta \lotimes_{B}Y) = \id_{Y}$.
\[
\begin{split}
&\id_{X} = \left[ \ X \xrightarrow{ \ X \lotimes_{B} \eta \ } X \lotimes_{B} Y \lotimes_{A} X \xrightarrow{ \ \epsilon \lotimes_{A} X \ } X  \ \right]\\ 
& \id_{Y} = \left[ \  Y \xrightarrow{ \ \eta \lotimes_{B} Y \ } Y \lotimes_{A} X \lotimes_{B} Y 
\xrightarrow{ \ Y \lotimes_{A} \epsilon \ } Y \ \right]
\end{split}
\]
If $(X, Y, \eta, \epsilon)$ is an adjoint pair, $X$ and  $Y$ are called the  left adjoint and the right adjoint $1$-morphisms  and 
$\eta$ and $\epsilon$ are called the unit morphism and the counit morphism.  
To write an adjoint pair $(X, Y, \eta, \epsilon)$, we usually write $(X, Y)$ and omit $\eta, \epsilon$
\end{definition}

We state the uniqueness of the right adjoint $1$-morphism. 
An analogues statement holds for the left adjoint $1$-morphism.

\begin{proposition}[{\cite[Proposition 4.4.10]{Asashiba:book},\cite[Lemma 6.1.6]{NY}}]\label{202509221237}
Let $X \in \sfD(A\otimes B^{\op})$. 
If there are two adjoint pairs 
$(X, Y, \eta, \epsilon)$ and $(X, Y', \eta', \epsilon')$, 
then there exists a unique isomorphism $f: Y \to Y'$ that satisfies 
the equations 
\begin{equation}\label{202509221403}
 (X \lotimes f) \eta = \eta', \ \ \epsilon = \epsilon'(f \lotimes X).
 \end{equation}
 Moreover, if an isomorphism $f: Y \to Y'$ satisfies one of the above equations, then it also satisfies the other. 
\end{proposition}

\begin{proof}
It is shown in \cite[Proposition 4.4.10]{Asashiba:book},\cite[Lemma 6.1.6]{NY} that 
an isomorphism $f: Y \to Y'$ is obtained as 
$f:= (Y' \lotimes \epsilon)(\eta' \lotimes Y)$. 
It is straightforward to check that this isomorphism satisfies the equations \eqref{202509221403}. 

To prove uniqueness, assume that there exists an isomorphism $f: Y \to Y'$ that satisfies 
$\epsilon = \epsilon'(f \lotimes X)$.
Then the following commutative diagram whose bottom line is the identity morphism by the triangle identity, 
shows $f= (Y' \lotimes \epsilon)(\eta' \lotimes Y)$.
\[
\begin{xymatrix}{ 
Y \ar[d]_{f} \ar[r]^-{\eta' \lotimes Y } & 
Y' \lotimes_{A} X \lotimes_{B} Y \ar[d]^{Y' \lotimes X \lotimes f } \ar[r]^-{Y'\lotimes \epsilon} &
Y' \ar@{=}[d] \\
Y'  \ar[r]_-{\eta' \lotimes Y' } & 
Y' \lotimes_{A} X \lotimes_{B} Y'  \ar[r]_-{Y'\lotimes \epsilon'} &
Y' 
}\end{xymatrix}
\]
Thus $f$ satisfies $ (X \lotimes f) \eta = \eta'$. 

The case in which  $f: Y \to Y'$ satisfies $ (X \lotimes f) \eta = \eta'$ can be proved in a similar way. 
\end{proof}

For  reference of the main body of the paper, 
we record the following proposition, 
which  shows that 
an equivalence $1$-morphism ({\cite[Definition 4.1.10]{Asashiba:book},\cite[Definition 5.1.18]{NY}}) 
can be promoted to  an adjoint equivalence ({\cite[Definition 4.4.1(4)]{Asashiba:book},\cite[Definition 6.2.1]{NY}}).

\begin{proposition}[{\cite[Theorem 4.4.4]{Asashiba:book},\cite[Proposition 6.2.4]{NY}}]\label{202509221331}
Let $X \in \sfD(A\otimes B^{\op})$ be an object. 
Assume that there exists $Y \in \sfD(B\otimes A^{\op})$ 
such that $X \lotimes_{B} Y \cong A$ in $\sfD(A^{\mre})$ and $Y \lotimes_{A} X \cong  B $ in $\sfD(B^{\mre})$. 
Fix an isomorphism $\eta: B \xrightarrow{ \cong } Y \lotimes_{A}X$, 
then there exists an isomorphism $\epsilon: X \lotimes_{B} Y \xrightarrow{ \cong } A$ such that the $4$-tuple $(X, Y, \eta, \epsilon)$ 
is an adjoint pair. 
\end{proposition}

\end{document}